\documentclass[11pt,a4paper]{amsart}
\usepackage{amssymb,amsmath,amsthm}

\usepackage[utf8]{inputenc}
\usepackage[T1]{fontenc}

\usepackage[pdftex]{graphicx}

\usepackage[colorlinks=true, pdfstartview=FitV, linkcolor=blue, citecolor=blue, urlcolor=blue,pagebackref=false]{hyperref}

\usepackage{esint}
\usepackage{mathrsfs}

\usepackage{times,latexsym,color}

\newcommand{\BOX}{\ensuremath\Box}

\newtheorem{theorem}{Theorem}
\newtheorem*{theorem*}{Theorem}
\newtheorem{proposition}{Proposition}
\newtheorem{lemma}[proposition]{Lemma}
\newtheorem{corollary}[proposition]{Corollary}

\theoremstyle{remark}
\newtheorem{remark}[proposition]{Remark}

\theoremstyle{definition}

\DeclareMathOperator{\supp}{supp}

\newcommand{\R}{\mathbb{R}}

\newcommand{\ep}{\varepsilon}

\newcommand{\Mp}{M^\flat}

\newcommand{\intbar}{{- \hspace{- 1.05 em}} \int}

\definecolor{darkgreen}{rgb}{0,0.5,0}
\definecolor{darkblue}{rgb}{0,0,0.7}
\definecolor{darkred}{rgb}{0.9,0.1,0.1}
\definecolor{lightblue}{rgb}{0,0.51,1}

\begin{document}

\title[Quantitative regularity via spatial concentration]{Quantitative regularity for the Navier-Stokes equations via spatial concentration}

\author[T. Barker]{Tobias Barker}
\address[T. Barker]{DMA, \'{E}cole Normale Sup\'erieure, CNRS, PSL Research University, 75\, 005 Paris}
\email{tobiasbarker5@gmail.com}

\author[C. Prange]{Christophe Prange}
\address[C. Prange]{Universit\'e de Bordeaux, CNRS, UMR [5251], IMB, Bordeaux, France}
\email{christophe.prange@math.u-bordeaux.fr}

\keywords{}
\subjclass[2010]{}
\date{\today}

\maketitle

\noindent {\bf Abstract} 
This paper is concerned with quantitative estimates for the Navier-Stokes equations. 

First we investigate the relation of 
quantitative bounds to the behaviour of critical norms near a potential singularity with Type I bound 
$\|u\|_{L^{\infty}_{t}L^{3,\infty}_{x}}\leq M$. Namely, we show that if $T^*$ is a first blow-up time and $(0,T^*)$ is a singular point then 
$$\|u(\cdot,t)\|_{L^{3}(B_{0}(R))}\geq C(M)\log\Big(\frac{1}{T^*-t}\Big),\,\,\,\,\,\,R=O((T^*-t)^{\frac{1}{2}-}).$$
We demonstrate that this potential blow-up rate is optimal for a certain class of potential non-zero backward discretely self-similar solutions.

Second, we quantify the result of Seregin (2012), which says that if $u$ is a smooth finite-energy solution to the Navier-Stokes equations on $\mathbb{R}^3\times (0,1)$ with
$$\sup_{n}\|u(\cdot,t_{(n)})\|_{L^{3}(\mathbb{R}^3)}<\infty\,\,\,\textrm{and}\,\,\,t_{(n)}\uparrow 1,$$
then $u$ does not blow-up at $t=1$. 

To prove our results we develop a new strategy for proving quantitative bounds for the Navier-Stokes equations. This hinges on local-in-space smoothing results (near the initial time) established by  Jia and \v{S}ver\'{a}k (2014), together with quantitative arguments using Carleman inequalities given by Tao (2019).

Moreover, the technology developed here enables us in particular to give a quantitative bound for the number of singular points in a Type I blow-up scenario.
\vspace{0.3cm}

\noindent {\bf Keywords}\,Navier-Stokes equations, quantitative estimates, critical norms, Type I blow-up, concentration, Carleman inequalities. 

\vspace{0.3cm}

\noindent {\bf Mathematics Subject Classification (2010)}\, 35A99, 35B44, 35B65, 35Q30, 76D05

\section{Introduction}

In this paper, we consider the three-dimensional incompressible Navier-Stokes equations
\begin{equation}\label{e.nse}
\partial_tu-\Delta u+u\cdot\nabla u+\nabla p=0,\quad \nabla\cdot u=0,\qquad u(\cdot,0)=u_{0}(x)\quad\mbox{in}\quad \R^3\times (0,T),
\end{equation}
where $T\in (0,\infty]$. It is well known that this system of equations is invariant with respect to the following rescaling 
\begin{equation}\label{NSErescale}
(u_{\lambda}(x,t), p_{\lambda}(x,t), u_{0\lambda}(x)):=(\lambda u(\lambda x,\lambda^2 t), \lambda^2 p(\lambda x,\lambda^2 t), \lambda u_{0}(\lambda x)),\,\,\,\,\lambda>0.
\end{equation} 
The question as to whether or not finite-energy solutions\footnote{Throughout this paper, we say $u$ is a finite-energy solution to the Navier-Stokes equations on $(0,T)$ if $u\in C_{w}([0,T]; L^{2}_{\sigma}(\mathbb{R}^3))\cap L^{2}(0,T; \dot{H}^{1}(\mathbb{R}^3))$ and $\|u(\cdot,t)\|_{L^{2}(\mathbb{R}^3)}^2+2\int\limits_{0}^{t}\int\limits_{\mathbb{R}^3}|\nabla u|^2dxdt'\leq \|u(\cdot,0)\|_{L^{2}(\mathbb{R}^3)}^2.$ }, with divergence-free Schwartz class initial data, remain smooth for all times is a Millennium Prize problem \cite{fefferman2006existence}. The first necessary conditions for such a solution to lose smoothness or to `blow-up' at time $T^{*}>0$\footnote{We say that a solution $u$ to the Navier-Stokes equations first blows-up at $T^*>0$ if $u\in L^{\infty}_{loc}(0,T^{*}, L^{\infty}(\mathbb{R}^3))$ but $u\notin L^{\infty}_{loc}(0,T^{*}]; L^{\infty}(\mathbb{R}^3))$} were given in the seminal paper of Leray \cite{Leray}. In particular, in \cite{Leray} it is shown that if $T^{*}$ is a first blow-up time of $u$ then we necessarily have
\begin{equation}\label{Leray}
\|u(\cdot,t)\|_{L^{p}(\mathbb{R}^3)}\geq \frac{C(p)}{(T^{*}-t)^{\frac{1}{2}(1-\frac{3}{p})}},\qquad\mbox{for}\quad p\in(3,\infty].
\end{equation}
The $L^{3}(\mathbb{R}^3)$ norm is scale-invariant or `critical'\footnote{We say $({X}, \|\cdot\|_{X})\subset\mathcal{S}'(\mathbb{R}^3)$ is critical if $u_{0}\in X\Rightarrow u_{0\lambda}(x):= \lambda u(\lambda x)\in X$ with $X$ norm equal to that of $u_{0}$.} with respect to the Navier-Stokes rescaling.
Its role in the regularity theory of the Navier-Stokes equations is much more subtle than that of the subcritical $L^{p}(\mathbb{R}^3)$ norms with $3<p\leq\infty.$
In particular, it is demonstrated by a elementary scaling argument in \cite{barker2017uniqueness}\footnote{The argument in \cite{barker2017uniqueness} is in turn taken from  the talk given by G .Seregin. `A certain necessary condition of possible blow up for the Navier-Stokes equations'. APDE seminar, University of Sussex, 03 March 2014.}
 that there \textit{cannot exist} a universal function $f:(0,\infty) \rightarrow (0,\infty)$ such that the following analogue of \eqref{Leray} holds true: 
\begin{itemize}
\item [] \begin{equation}\label{finfiniteatzero}
\lim_{s\rightarrow 0^{+}} f(s)=\infty,
\end{equation}
\item [] and if $u$ is a finite-energy solution to the Navier-Stokes equations (with Schwartz class initial data) that first blows-up at $T^{*}>0$ then  $u$ necessarily satisfies
\begin{equation}\label{boundedbelowuniversalfunction}
\|u(\cdot,t)\|_{L^{3}(\mathbb{R}^3)}\geq f(T^*-t)
\end{equation}
for all $t\in [0,T^*)$.
\end{itemize}
In the celebrated paper \cite{ESS2003} of Escauriaza, Seregin and \v{S}ver\'{a}k, it was shown that if a finite-energy solution $u$ first blows-u at $T^{*}>0$ then necessarily
\begin{equation}\label{ESS2003}
\limsup_{t\uparrow T^{*}} \|u(\cdot,t)\|_{L^{3}(\mathbb{R}^3)}=\infty.
\end{equation}
The proof in \cite{ESS2003} is by contradiction. A rescaling procedure or `zoom-in' is performed\footnote{It is worth to note that \cite{ESS2003} appears to be the first instance where arguments involving `zooming in' and passage to a limit have been applied to the Navier-Stokes equations. } using \eqref{NSErescale} and a compactness argument is applied. This gives a non-zero limit solution to the Navier-Stokes equations that vanishes at the final moment in time. The contradiction is achieved by showing that the limit function must be zero by applying a Liouville type theorem based on backward uniqueness for parabolic operators satisfying certain differential inequalities.
By now there are many generalizations of \eqref{ESS2003} to cases of other critical norms. See, for example, \cite{CWY19}, \cite{GKP}, \cite{phuc2015navier} and \cite{wang2017blow}. 

Let us mention the arguments in \cite{ESS2003} and the aforementioned works are by contradiction and hence are \textit{qualitative}. It is worth noting that the result in \cite{ESS2003}, together with a proof by contradiction based on the `persistence of singularities' lemma in  \cite{rusin2011minimal} (specifically Lemma 2.2 in \cite{rusin2011minimal}), gives the following. Namely, that there exists an $F:(0,\infty)\rightarrow (0,\infty)$ such that if $u$ is a finite-energy solution to the Navier-Stokes equations then
\begin{equation}\label{L3weaklquant}
\|u\|_{L^{\infty}(0,1; L^{3}(\mathbb{R}^3))}<\infty\Rightarrow\|u\|_{L^{\infty}(\mathbb{R}^3\times (\frac{1}{2},1))}\leq F(\|u\|_{L^{\infty}(0,1; L^{3}(\mathbb{R}^3))}).
\end{equation}
Such an argument is obtained by a compactness method and gives no explicit\footnote{\label{foot.eff}Throughout this paper we will sometimes use the terminology `\textit{effective}' bounds to describe an explicit quantitative bound. An abstract quantitative bound will sometimes be referred to as `\textit{non-effective}'.} information about $F$. In a remarkable recent development \cite{Tao19}, Tao used a new approach to provide the first explicit \textit{quantitative} estimates for solutions of the Navier-Stokes equations belonging to the critical space $L^{\infty}(0,T; L^{3}(\mathbb{R}^3))$. As a consequence of these quantitative estimates, Tao showed in \cite{Tao19} that if a finite-energy solution $u$ first blows-up at $T^{*}>0$ then for some absolute constant $c>0$
\begin{equation}\label{L3Tao}
\limsup_{t\uparrow T^*}\frac{\|u(\cdot,t)\|_{L^{3}(\mathbb{R}^3)}}{\big(\log\log\log\frac{1}{T^*-t}\big)^c}=\infty.
\end{equation}
Since there cannot exist $f$ such that \eqref{finfiniteatzero}-\eqref{boundedbelowuniversalfunction} holds true, at first sight \eqref{L3Tao} may seem somewhat surprising, though it is not conflicting with such a fact. Notice that
$$\frac{\|u(\cdot,t)\|_{L^{3}(\mathbb{R}^3)}}{\big(\log\log\log\frac{1}{T^*-t}\big)^c}$$
is not invariant with respect to the Navier-Stokes scaling \eqref{NSErescale} but is \textit{slightly supercritical}\footnote{We say a quantity $F(u,p)$ is supercritical if, for the rescaling \eqref{NSErescale}, we have $F(u_{\lambda}, p_{\lambda})= \lambda^{-\beta}F(u_{\lambda}, p_{\lambda})$ for some $\beta>0$.}  due to the presence of the logarithmic denominator. Let us also mention that prior to Tao's paper \cite{Tao19}, in the presence of axial symmetry, a different slightly supercritical regularity criteria was obtained in \cite{pan2016regularity}. 

The contribution of our present paper is to develop a new strategy for proving quantitative estimates (see Propositions \ref{prop.main} and \ref{prop.maints}) for the Navier-Stokes equations, which then enables us to build upon Tao's work \cite{Tao19} to quantify critical norms. Our first Theorem involves applying the backward propagation of concentration stated in Proposition \ref{prop.main} below to give a new necessary condition for solutions to the Navier-Stokes equations to possess a Type I blow-up. Note that if $u$ is a finite-energy solution that first blows-up at $T^{*}>0$ we say that $T^{*}$ is a Type I blow-up if
\begin{equation}\label{TypeIdef}
\|u\|_{L^{\infty}(0,T^{*}; L^{3,\infty}(\mathbb{R}^3))}\leq M.
\end{equation}
In the case of a Type I blow-up at $T^{*}$ the nonlinearity in \eqref{NSErescale} is heuristically balanced with the diffusion.
Despite this, it remains a long standing open problem whether or not Type I blow-ups can be ruled out when $M$ is large. Let us now state our first theorem.

\begin{theorem}[rate of blow-up, Type I]\label{theologL3}
There exists a universal constant $M_0\in [1,\infty)$ 
such that for all $M\geq M_0$ and $\delta\in (0,1)$ the following holds true.

Assume that $u$ is a mild solution to the Navier-Stokes equations on $\mathbb{R}^3\times [0,T^*)$ with $u\in L^{\infty}_{loc}([0,T^{*}); L^{\infty}(\mathbb{R}^3))$.\footnote{\label{footdefsol}Under these assumptions, $u$ is smooth on the epoch $(0,T)$ for any $T<T^*$ and belongs to $L^{\infty}((0,T); L^{4}(\mathbb{R}^3))\cap  L^{\infty}((0,T); L^{5}(\mathbb{R}^3))$ 
 by interpolation, which enables us to satisfy the hypothesis needed in Sections \ref{sec.quantannulus}-\ref{sec.6}. 
 Furthermore using Lemma 2.4 in \cite{JS13}, it gives that $u$ coincides with all local energy solutions (we refer to footnote \ref{footles} for a definition), with initial data $u(\cdot,s)$, on $\mathbb{R}^3\times (s,T^*)$ for any $0<s<T^*$. We call such a solution a \textbf{`smooth solution with sufficient decay'} on the interval $[0,T]$, for $T<T^*$. A mild solution with Schwartz class initial data and maximal time of existence $T^*$ will be such a solution, and so Theorem \ref{theologL3} applies to that setting. Notice that smoothness is needed here in order to get estimate \eqref{L3localisedlog} for all $t$ in the ad hoc interval.
 The framework of `smooth solutions with enough decay' is needed to apply Theorem \ref{theologL3} to the setting of Corollary \ref{optimalrateDSS}, where the solution is not of finite energy.} 
Assume that
\begin{enumerate}
\item $\|u\|_{L^{\infty}_{t}L^{3,\infty}_{x}(\mathbb{R}^3\times (0,T^*))}\leq M$
\item $u$ has a singular point at $(x,t)=(0,T^*)$. In particular $u\notin L^{\infty}_{x,t}(Q_{(0,T^*)}(r))$ for all sufficiently small $r>0$.
\end{enumerate}
Then the above assumptions imply that there exists  $c(\delta,M,T^*)\in(0,\infty)$, which we will specify in the proof, such that for any $t\in (\max(\frac{T^*}{2},T^*-c(\delta,M,T^*)), T^*)$ we have
\begin{equation}\label{L3localisedlog}
\int\limits_{B_{0}\big((T^*)^{\frac{1}{2}}(T^*-t)^{\frac{1-\delta}{2}}\big)} |u(x,t)|^3 dx\geq
\frac{\log\Big(\frac{1}{(T^*-t)^{\frac{\delta}{2}}}\Big)}{\exp(\exp(M^{1025}))}.
\end{equation}
\end{theorem}

This theorem is proved in Subsection \ref{sec.proofmainrests} below. Notice that in Theorem \ref{theologL3}, not only is the rate new but also the fact that the $L^{3}$ norm blows up on a ball of radius $O((T^*-t)^{\frac{1}{2}-})$ around \textit{any} Type I singularity. Previously in \cite{LOW18} (specifically Theorem 1.3 in \cite{LOW18}), it was shown that if a solution blows up (without Type I bound) then the $L^{3}$ norm blows up on certain \textit{non-explicit} concentrating sets.

The Navier-Stokes scaling symmetry \eqref{NSErescale} plays a role in considering blow-up ansatzes having certain symmetry properties. In \cite{Leray}, Leray suggested the blow-up ansatz of \textit{backward self-similar solutions}\footnote{We say $u:\mathbb{R}^3\times (-\infty,0)\rightarrow\mathbb{R}^3$ is a backward self similar solution if $u(x,t)=\frac{1}{\sqrt{-t}}u\Big(\frac{x}{\sqrt{-t}}, 1\Big)$ for all $(x,t)\in\mathbb{R}^3\times (-\infty,0)$.}, which are invariant with respect to the Navier-Stokes rescaling.
Although  the existence of non-zero  backward self-similar solutions to the Navier-Stokes equations has been ruled out under general circumstances in \cite{nevcas1996leray} and \cite{tsai1998leray}, the existence of non-zero backward discretely self-similar solutions remains open.  Here we say that $u$ is a backward discretely self-similar solution ($\lambda$-DSS) if there exists $\lambda\in (1,\infty)$ such that $u(x,t)=\lambda u(\lambda x,\lambda^2 t)$ for all $(x,t)\in \mathbb{R}^3\times (-\infty,0)$.
As a corollary to Theorem \ref{theologL3}, we show that if there exists a non-zero $\lambda$-DSS (having certain decay properties which we will specify), then the localized blow-up rate \eqref{L3localisedlog} in Theorem \ref{theologL3} is optimal.
\begin{corollary}\label{optimalrateDSS} 
Suppose $u:\mathbb{R}^3\times (-\infty,0)\rightarrow 0$ is a non-zero $\lambda$-DSS to the Navier-Stokes equations such that
\begin{equation}\label{DSShypothesis}
u\in C^{\infty}(\mathbb{R}^3\times (-\infty,0))\cap C((-\infty,0); L^{p}(\mathbb{R}^3)),
\end{equation}
for some $p\in[3,\infty)$.
There exists $M>1$ such that for every $\delta\in(0,1)$ there is a $C(\delta,M)\in (0,\infty)$ with the holding true. Namely, for all $t\in [\max(-\tfrac{1}{2}, -C(\delta,M)),0)$ we have
\begin{equation}\label{L3localisedlogDSS}
\frac{\log\Big(\frac{1}{(-t)^{\frac{\delta}{2}}}\Big)}
{\exp(\exp(M^{1025}))}\leq \int\limits_{B_{0}(1)} |u(x,t)|^3 \leq M^3\log\Big(\frac{2}{\sqrt{-t}}\Big).
\end{equation}
\end{corollary}

This corollary is proved in Subsection \ref{sec.proofmainrests} below.

In \cite{ESS2003}, it is shown that if $u$ is a finite-energy  solution to the Navier-Stokes equations in 
$C^{\infty}(\mathbb{R}^3\times (0,1))$, with Schwartz initial data, then
\begin{equation}\label{LinfinityL3}
\|u\|_{L^{\infty}((0,1); L^{3}(\mathbb{R}^3))}<\infty
\end{equation}
implies that $u$ does not blow-up at time $1$ (namely $u\in L^{\infty}_{t, loc}((0,1]; L^{\infty}(\mathbb{R}^3))$).
In \cite{seregin2012}, Seregin refined the assumption \eqref{LinfinityL3} to
\begin{equation}\label{Seregin2012}
\sup_{n}\|u(\cdot,t_{(n)})\|_{L^{3}(\mathbb{R}^3)}<\infty\,\,\,\textrm{with}\,\,\,t_{(n)}\uparrow 1.
\end{equation}
Seregin's result implies that if $u$ is a finite-energy solution that first loses smoothness at $T^{*}>0$ then
\begin{equation}\label{wholelimitL3}
\lim_{t\uparrow T^{*}}\|u(\cdot,t)\|_{L^{3}(\mathbb{R}^3)}=\infty.
\end{equation}
This result has been further refined to other wider critical spaces and to domains other than $\mathbb{R}^3$. See, for example, \cite{Albritton18}, \cite{AlbrittonBarkerBesov2018}-\cite{AlbrittonBarker2018local}, \cite{BS17} and \cite{MMP17b}.
All these arguments are qualitative and achieved by contradiction and compactness arguments. It is interesting to note that in contrast to \eqref{L3weaklquant} it is not known\footnote{Assume for contradiction that \eqref{L3tsweaklquant} does not hold. Then we have a sequence of solutions $(u^{k})_{k\in\mathbb{N}}$ such that $\|u^{(k)}(\cdot,1)\|_{L^{\infty}(\mathbb{R}^3)}\uparrow \infty$ and, for each $k$, a sequence of associated time slices $t_{(n)}^{k}\uparrow 1$ such that $\sup_{n,k}\|u^{(k)}(\cdot,t_{(n)}^k)\|_{L^{3}(\mathbb{R}^3)}=M<\infty$. The main block for the contradiction argument to go through is that 
the sequence $(t_{(n)}^{k})_{n\in\mathbb{N}}$ may be different for distinct indices $k$.}, even abstractly, 
if there exists a $G:(0,\infty)\rightarrow(0,\infty)$ such that if $u$ is a finite-energy solution of the Navier-Stokes equations belonging to $C^{\infty}(\mathbb{R}^3\times (0,1])$  then 
\begin{equation}\label{L3tsweaklquant}
\sup_{n}\|u(\cdot,t_{(n)})\|_{L^{3}(\mathbb{R}^3)}<\infty\,\,\,\textrm{with}\,\,\,t_{(n)}\uparrow 1\,\Rightarrow\|u(\cdot,1)\|_{L^{\infty}(\mathbb{R}^3)}\leq G\big(\sup_{n}\|u(\cdot,t_{(n)})\|_{L^{3}(\mathbb{R}^3)}\big).
\end{equation}

In our second main theorem, we apply Proposition \ref{prop.maints} to fully quantify Seregin's result in \cite{seregin2012}, which generalizes Theorem 1.2 in \cite{Tao19}. Now let us state our second theorem. 
\begin{theorem}[main quantitative estimate, time slices; quantification of Seregin's result]\label{theo.mainbis} 
There exists a universal constant $M_{1}\in[1,\infty)$. Let $M\in[M_{1},\infty)$.
We define 
 $\Mp$ by\footnote{\label{footenergy}In particular, $\Mp$ is chosen such that the following is true. If $(u,p)$ is a  suitable finite-energy solution (defined in Section \ref{subsec.not} `Notations') to Navier-Stokes on $\mathbb{R}^3\times [0,T)$ with $L^{3}$ initial data $\|u_{0}\|_{L^{3}(\mathbb{R}^3)}\leq M$, then $w:= u-e^{t\Delta}u_{0}$ satisfies $\|w(\cdot,t)\|_{L^{2}(\mathbb{R}^3)}^2+\int\limits_{0}^{t}\int\limits_{\mathbb{R}^3} |\nabla w|^2 dxdt'\leq (\Mp)^4 t^{\frac{1}{2}}$ for $t\in (0,T)$ and $M$ larger than a universal constant. See Lemma \ref{lem.backconctimeslice}.}
\begin{equation}\label{e.defM'theo}
\Mp:= \exp\Big(\frac{L_*M^{5}}{2}\Big),
\end{equation}
for an appropriate constant $L_*\in(0,\infty)$. Let $(u,p)$ be a finite-energy $C^{\infty}(\mathbb{R}^3\times (-1,0))$ solution to the Navier-Stokes equations \eqref{e.nse} on $\R^3\times[-1,0]$.\footnote{Notice that smoothness is needed here to have the energy inequality starting from every time $t_k$, and not for almost every $t'\in(-1,0)$, see \eqref{energyinequalityturbulent} as would be the case if $(u,p)$ was just a suitable finite-energy solution.} 
 Assume  that there exists $t_{(k)}\in [-1,0)$ such that
\begin{equation}\label{e.tsAtheo}
t_{(k)}\uparrow 0\,\,\,\,\textrm{with}\,\,\,\, \sup_{k} \|u(\cdot, t_{(k)})\|_{L^{3}(\mathbb{R}^3)}\leq M.
\end{equation}
Select any ``well-separated'' subsequence (still denoted $t_{(k)}$) such that
\begin{equation}\label{wellsepproptstheo}
\sup_{k}\frac{-t_{(k+1)}}{-t_{(k)}}<\exp(-2(\Mp)^{1223}).
\end{equation}
Then for 
\begin{equation}\label{jdeftheo}
j:=\lceil\exp(\exp(({\Mp})^{1224}))\rceil+1,
\end{equation}
we have the bound
\begin{equation}\label{Linfinityboundtstheo}
\|u\|_{L^{\infty}\big(\mathbb{R}^3\times \big(\tfrac{t_{(j+1)}}{4}, 0\big)\big)}\leq \frac{C_{1}M^{-23}}{(-t_{(j+1)})^{\frac{1}{2}}},
\end{equation}
for  a universal constant $C_1\in(0,\infty)$.  
\end{theorem}

This theorem is proved in Subsection \ref{sec.proofmainrests} below.

\subsubsection*{Further applications}
Section \ref{sec.further} contains three further applications of the technology developed in the present paper: (i) Proposition \ref{prop.type1limsup}, a regularity criteria based on an effective\footnote{See footnote \ref{foot.eff} for the definition of `effective' bounds.} relative smallness condition in the Type I setting, (ii) Corollary \ref{cor.number}, an effective bound for the number of singular points in a Type I blow-up scenario, (iii) Proposition \ref{prop.relative}, a regularity criteria based on an effective relative smallness condition on the $L^3$ norm at initial and final time. {Non-effective} quantitative bounds of the above results were previously obtained by compactness methods: for (ii) see \cite[Theorem 2]{CWY19}, for (iii) see \cite[Theorem 4.1 (i)]{AlbrittonBarkerBesov2018}.

\subsection{Comparison to previous literature and novelty of our results}
Theorems \ref{theologL3} and \ref{theo.mainbis} in this paper follow from new quantitative estimates for the Navier-Stokes equations (Propositions \ref{prop.main} and \ref{prop.maints}), which build upon recent breakthrough work by Tao in \cite{Tao19}. In particular, Tao shows that for \textit{classical}\footnote{In \cite{Tao19}, these are solutions that are smooth in $\mathbb{R}^{3}\times (0,1)$ and such that all derivatives of $u$ and $p$ lie in $L^{\infty}_{t}L^{2}_{x}(\mathbb{R}^2\times (0,1))$.}
solutions to the Navier-Stokes equations 
\begin{equation}\label{taomainest}
\|u\|_{L^{\infty}_{t}L^{3}_{x}(\mathbb{R}^3\times (0,1))}\leq A\Rightarrow \|u(\cdot,t)\|_{L^{\infty}(\mathbb{R}^3)}\leq \exp(\exp(\exp(A^{O(1)})))t^{-\frac{1}{2}}\,\,\,\,\textrm{for}\,\,\,0<t\leq 1.
\end{equation}
Before describing our contribution, we first find it instructive to outline Tao's approach in \cite{Tao19}.

Fundamental to Tao's approach for showing \eqref{taomainest} is the following fact\footnote{Let $\varphi\in C^{\infty}_{0}(B_{0}(1))$ with $\varphi\equiv 1$ on $B_{0}(\frac{1}{2})$. The Littlewood-Paley projection $P_{N}$ is defined for any $N>0$ by $\widehat{P_{N}}f(\xi):=\Big(\varphi(\frac{\xi}{N})-\varphi(\frac{2\xi}{N})\Big)\widehat{f}(\xi).$} (see Section 6 in \cite{Tao19}). There exists a universal constant $\varepsilon_{0}$ such that if $u$ is a classical solution to the Navier-Stokes equations with
\begin{align}\label{scaleinvarsmallhighfreqbis}
&\|u\|_{L^{\infty}_{t}L^{3}_{x}(\mathbb{R}^3\times (0,1))}\leq A\\
\mbox{and}\quad &N^{-1}\|P_{N}u\|_{L^{\infty}_{x,t}(\mathbb{R}^3\times (\frac{1}{2},1))}<\varepsilon_{0}\,\,\,\textrm{for}\,\,\textrm{all}\,\,\,N\geq N_{*},\label{scaleinvarsmallhighfreq}
\end{align}
then $\|u\|_{L^{\infty}_{x,t}(\mathbb{R}^3\times (\frac{7}{8},1))}$ can be estimated explicitly in terms of $A$ and $N_{*}$.
Related observations were made previously in \cite{cheskidov2010regularity}, but in a slightly different context and without the bounds explicitly stated.

In this perspective, Tao's aim is the following:
\begin{itemize}
\item []\textbf{Tao's goal}: Under the scale-invariant  assumption \eqref{scaleinvarsmallhighfreqbis}, if \eqref{scaleinvarsmallhighfreq} fails for $\varepsilon_{0}= A^{-O(1)}$ and $N=N_{0}$, what is an upper bound for $N_{0}$? 
\end{itemize}
In \cite{Tao19} (Theorem 5.1 in \cite{Tao19}), it is shown that $N_{0}\lesssim\exp\exp\exp(A^{O(1)})$, which implies \eqref{taomainest} by means of the quantitative regularity mechanism \eqref{scaleinvarsmallhighfreq} with $N_{*}=2N_{0}$.
We emphasize that since the regularity mechanism \eqref{scaleinvarsmallhighfreq} is \textit{global}: all quantitative estimates obtained in this way are in terms of globally defined quantities.

The strategy in \cite{Tao19} for showing Tao's goal with $N_{0}\lesssim\exp(\exp(\exp(A^{O(1)})))$ can be summarized in four steps. We refer the reader to the Introduction in \cite{Tao19} for more details.
\begin{itemize}
\item[] \textbf{1) Frequency bubbles of concentration (Proposition 3.2 in \cite{Tao19})}.\newline
Suppose $\|u\|_{L^{\infty}_{t}L^{3}_{x}(\mathbb{R}^3\times (t_{0}-T,t_{0}))}\leq A$  is such that
\begin{equation}\label{frequencyconc}
N_{0}^{-1}|P_{N_{0}}u(x_{0},t_{0})|> A^{-O(1)}.
\end{equation}
Then for all $n\in\mathbb{N}$ there exists $N_{n}>0$, $(x_{n}, t_{n})\in\mathbb{R}^3\times (t_{0}-T,t_{n-1})$ such that
\begin{equation}\label{frequencyconciterate}
N_{n}^{-1}|P_{N_{n}}u(x_{n},t_{n})|> A^{-O(1)}
\end{equation}
with 
\begin{equation}\label{parabolicdependence}
x_{n}=x_{0}+O((t_{0}-t_{n})^{\frac 1 2}), N_{n}\sim |t_{0}-t_{n}|^{-\frac{1}{2}}.
\end{equation}
\item[] \textbf{2) Localized lower bounds on vorticity (p.37 in \cite{Tao19}).} 
For certain scales $S>0$ and an `epoch of regularity' $I_S\subset [t_{0}-S, t_{0}-A^{-\alpha}S]$, where the solution enjoys `good' quantitative estimates on $\mathbb{R}^3\times I_S$ (in terms of $A$ and $S$), Tao shows the following: the previous step and $\|u\|_{L^{\infty}_{t}L^{3}_{x}(\mathbb{R}^3\times [t_{0}-T,t_{0}])}\leq A$ imply
\begin{equation}\label{vortlowerboundtao}
\int\limits_{B_{x_{0}}(A^{\beta} S^{\frac 12 })}|\omega(x,t)|^2 dx\geq A^{-\gamma}S^{-\frac 12}\,\,\,\textrm{for}\,\,\,\textrm{all}\,\,\,t\in I_S.
\end{equation}
Here, $\alpha$, $\beta$ and $\gamma$ are positive universal constants.
\item[] \textbf{3) Lower bound on the $L^{3}$ norm at the final moment in time $t_{0}$ (p.37-40 in \cite{Tao19}).} Using quantitative versions of the Carleman inequalities in \cite{ESS2003} (Propositions 4.2-4.3 in \cite{Tao19}), Tao shows that the lower bounds in step 2 can be transferred to a lower bound on the $L^{3}$ norm of $u$ at the final moment of time $t_{0}$. The applicability of the Carleman inequalities to the vorticity equation requires the `epochs of regularity' in the previous step and the existence of `good spatial annuli' where the solution enjoys good quantitative estimates.  Specifically, Tao shows that step 2 on $I_S$ implies
\begin{equation}\label{L3lowerTao}
\int\limits_{R_S\leq|x-x_{0}|\leq R_S'} |u(x,t_{0})|^3 dx\geq \exp(-\exp(A^{O(1)})).
\end{equation} 
\item[]\textbf{4) Conclusion: summing scales to bound $TN_{0}^2$.}
Letting $S$ vary for certain permissible $S$, 
the annuli in \eqref{L3lowerTao} become disjoint. The sum of \eqref{L3lowerTao} over such disjoint permissible annuli is bounded above by $\|u(\cdot, t_{0})\|_{L^{3}(\mathbb{R}^3)}$ and the lower bound due to the summing of scales is $\exp(-\exp(A^{O(1)}))\log(TN_{0}^2)$.
This gives the desired bound on $N_{0}$, namely
$$TN_{0}^2\lesssim \exp(\exp(\exp( A^{O(1)}))). $$
\end{itemize}
Let us emphasize once more that the approach in \cite{Tao19} produces quantitative estimates involving \textit{globally defined quantities}, since the quantitative regularity mechanism \eqref{scaleinvarsmallhighfreq} is inherently global. We would also like to emphasize that the fact that $\|u\|_{L^{\infty}_{t}L^{3}_{x}}<A$ is crucial for showing steps 1-2 in the above strategy.

The goal of the present paper is to develop a new robust strategy for obtaining new quantitative estimates of the Navier-Stokes equations, which are then applied to obtain Theorems \ref{theologL3} and \ref{theo.mainbis}. The main novelty (which we explain in more detail below) is that our strategy allows us to obtain \textit{local} quantitative estimates and even applies to certain situations where we are outside the regime of scale-invariant controls. 
 For simplicity, we will outline the strategy for the case when $\|u\|_{L^{\infty}_{t}L^{3,\infty}_{x}(\mathbb{R}^3\times (t_{0}-T, t_{0}))}\leq M$, before remarking on this strategy for cases without such a scale-invariant control (Theorem \ref{theo.mainbis}).

Fundamental to our strategy is the use of local-in-space smoothing near the initial time for the Navier-Stokes equations pioneered by Jia and \v{S}ver\'{a}k in \cite{JS14} (see also \cite{BP18} for extensions to critical cases). In particular, the result of \cite{JS14}, together with rescaling arguments from \cite{BP18}, implies the following. If $u:\mathbb{R}^3\times [t_{0}-T, t_{0}]\rightarrow\mathbb{R}^3$ is a smooth solution with sufficient decay\footnote{In this paper, `smooth solution with sufficient decay' always denotes the notion described in footnote \ref{footdefsol}.} of the Navier-Stokes equations and $\|u\|_{L^{\infty}_{t}L^{3,\infty}_{x}(\mathbb{R}^3\times (t_{0}-T, t_{0}))}\leq M$, then
\begin{equation}\label{locinspaceintro}
\int\limits_{B_{x_{0}}(4\sqrt{{S}^{\sharp}(M)}^{-1}(t_{0}-t^*_{0})^{\frac 12})}|\omega(x,t^*_{0})|^2 dx\leq \frac{M^2 \sqrt{S^{\sharp}(M)}}{(t_{0}-t^*_{0})^{\frac{1}{2}}}
\end{equation}
for $t^*_{0}\in (t_{0}-T, t_{0})$ implies that  
\begin{equation}\label{Linfinitylocainspaceintro}
\|u\|_{L^{\infty}_{x,t}\big(B_{x_{0}}(\tfrac{1}{2}\sqrt{{S}^{\sharp}(M)}^{-1}(t_{0}-t^*_{0})^{\frac 12})\times (\tfrac{3}{4}(t_{0}-t^*_{0})+t_{0}, t_{0})\big)}
\end{equation}
can be estimated explicitly in terms of $M$ and $t_{0}-t^*_{0}$. Here, ${S}^{\sharp}(M)=O(1)M^{-100}$ is as in Theorem \ref{theo.locshortime}.

In this perspective, the aim of our strategy is the following
\begin{itemize}
\item[] \textbf{Our goal:} If \eqref{locinspaceintro} fails for $t^*_0=t'_{0}$, what is a lower bound for $t_{0}-t'_{0}$?
\end{itemize} 
This is the main aim of Proposition \ref{prop.main}. Taking 
$s_{0}$ such that $t_{0}-t'_{0}\geq 2Ts_{0}$, we can then apply \eqref{locinspaceintro}-\eqref{Linfinitylocainspaceintro} with $t^*_{0}= t_{0}-Ts_{0}$. One might think of the main goal of our strategy as a physical space analogy to Tao's goal with 
$$N_{0}\sim |t_{0}-t'_{0}|^{-\frac{1}{2}}. $$ In contrast to \eqref{scaleinvarsmallhighfreq}, the regularity mechanism \eqref{locinspaceintro}-\eqref{Linfinitylocainspaceintro} produces quantitative estimates that are in terms of locally defined quantities, which is crucial for obtaining the \textit{localized} results as in Theorem \ref{theologL3}.
Our strategy for obtaining  a lower bound of $t_{0}-t_{0}'$  (see Proposition \ref{prop.main}) can be summarized in three steps; see also Figure \ref{fig.summary}.
\begin{itemize} 
\item[] \textbf{1) Backward propagation of vorticity concentration (Lemma \ref{lem.backconc}).}\\
Let $\|u\|_{L^{\infty}_{t}L^{3,\infty}_{x}(\mathbb{R}^3\times (t_{0}-T,t_{0}))}\leq M$. Suppose $t'_{0}\in (t_{0}-T,t_{0})$ is not too close to $t_{0}-T$ and is such that 
\begin{equation}\label{vortconcintro}
\int\limits_{B_{x_{0}}(4\sqrt{{S}^{\sharp}(M)}^{-1}(t_{0}-t'_{0})^{\frac 12})}|\omega(x,t'_{0})|^2 dx> \frac{M^2 \sqrt{S^{\sharp}(M)}}{(t_{0}-t'_{0})^{\frac{1}{2}}}.
\end{equation}
We show that for all $t''_{0}\in (t_{0}-T,t'_{0})$, such that
$t_{0}-t''_{0}$ is sufficiently large compared to $t_{0}-t'_{0}$ (in other words $t_0''$ is well-separated from $t_0'$), we have
\begin{equation}\label{backwardvortconcintro}
\int\limits_{B_{x_{0}}(4\sqrt{{S}^{\sharp}(M)}^{-1}(t_{0}-t''_{0})^{\frac 12})}|\omega(x,t''_{0})|^2 dx> \frac{M^2 \sqrt{S^{\sharp}(M)}}{(t_{0}-t''_{0})^{\frac{1}{2}}}.
\end{equation}
We refer the reader to Lemma \ref{lem.backconc} for precise statements for the rescaled/translated situation $\mathbb{R}^3\times (t_{0}-T,t_{0})= \mathbb{R}^3\times (-1,0).$
\item[] \textbf{2) Lower bound on localized $L^{3}$ norm at the final moment in time $t_{0}$.} Using the previous step, together with the same arguments as \cite{Tao19} involving quantitative Carleman inequalities, we show that for certain permissible annuli that
\begin{equation}\label{L3lowerintro}
\int\limits_{R\leq|x-x_{0}|\leq R'} |u(x,t_{0})|^3 dx\geq \exp(-\exp(M^{O(1)})).
\end{equation}
We wish to mention that the role of the Type I bound is to show the solution $u$ obeys good quantitative estimates in certain space-time regions, which is needed to apply the Carleman inequalities to the vorticity equation.
\item[] \textbf{3) Conclusion: summing scales to bound $t_{0}-t'_{0}$ from below.}
Summing \eqref{L3lowerintro} over all permissible disjoint annuli finally gives us the desired lower bound for $t_{0}-t'_{0}$ in Proposition \ref{prop.main}. We note that the localized $L^{3}$ norm of $u$ at time $t_{0}$ plays a distinct role to that of Type I condition described in the previous step. Its sole purpose is to bound the number of permissible disjoint annuli that can be summed, which in turn gives the lower bound of $t_{0}-t'_{0}$. Together with the assumed global Type I assumption, this is essentially why the lower bound in Theorem \ref{theologL3} on the localized $L^{3}$ norm near a Type I singularity is a single logarithm and holds at pointwise times.
\end{itemize}
Although the above relates to the case of Proposition \ref{prop.main} and Theorem \ref{theologL3} where $$\|u\|_{L^{\infty}_{t}L^{3,\infty}_{x}(\mathbb{R}^3\times (t_{0}-T,t_{0}))}\leq M,$$ we stress that the above strategy (with certain adjustments) is robust enough to apply to certain settings without a Type I control (Theorem \ref{theo.mainbis}). 

\begin{figure}[t]\label{fig.summary}
\begin{center}
\includegraphics[scale=0.75]{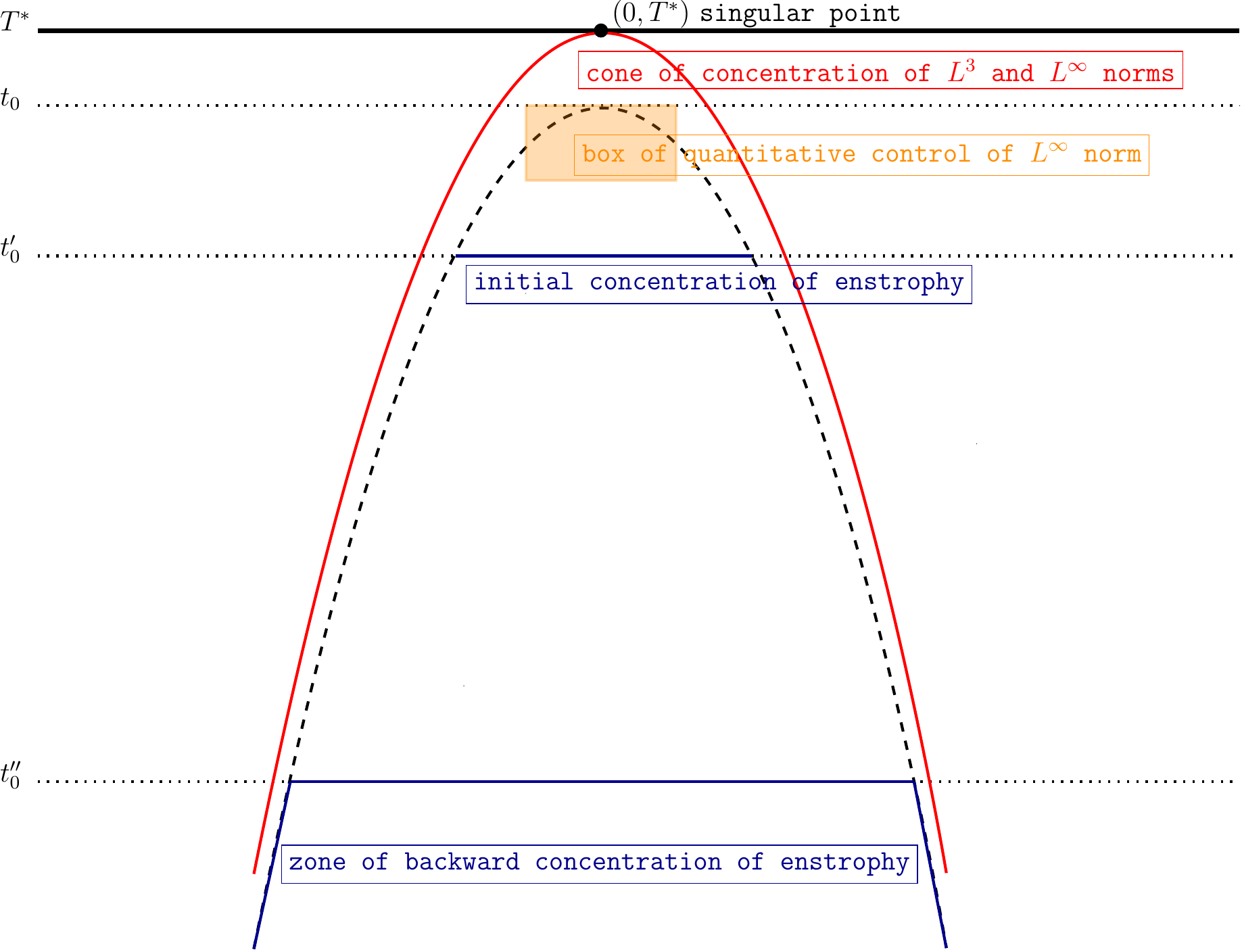}
\caption{Zones of concentration and quantitative regularity}
\end{center}
\end{figure}

Recall that Theorem \ref{theo.mainbis} is concerned with quantitative estimates on $u:\mathbb{R}^3\times (-1,0)\rightarrow\mathbb{R}^3$, where we assume
\begin{equation}\label{L3sequenceintro}
t_{(k)}\uparrow 0\,\,\,\textrm{with}\,\,\,\sup_{k}\|u(\cdot,t_{(k)})\|_{L^{3}(\mathbb{R}^3)}\leq M.
\end{equation}
First we remark that the local quantitative regularity statement \eqref{locinspaceintro}-\eqref{Linfinitylocainspaceintro} remains true (with ${t^*_{0}}$ replaced by $t_{k}$) if $u$ is a $C^{\infty}(\mathbb{R}^3\times (-1,0])$ finite-energy solution and the Type I condition is replaced by the weaker assumption that $\|u(\cdot,t_{(k)})\|_{L^{3}(\mathbb{R}^3)}\leq M$. Our goal then becomes the following
\begin{itemize}
\item[] \textbf{Our second goal:} If \eqref{locinspaceintro} fails for $t^*_{0}=t_{j}$ (with $t_0=0$ and $T=1$), what is an upper bound for $j$?
\end{itemize}
In this setting `\textbf{1) Backward propagation of vorticity concentration}' still remains valid if a sufficiently well-separated subsequence of $t_{(k)}$ is taken (see Lemma \ref{lem.backconctimeslice} and Proposition \ref{prop.maints}). To show this we use energy estimates in \cite{SeSv17} for solutions to the Navier-Stokes equations with $L^{3}(\mathbb{R}^3)$ initial data. Such estimates are also central to gain good quantitative control of the solution in certain space-time regions, which are required for applying the quantitative Carleman inequalities. The price one pays in this setting (when compared to the estimates  in \cite{Tao19}), is a gain of an additional exponential in the estimates. The reason is the control on  the energy of  $u(\cdot,t)-e^{t\Delta}u_{0}$ ( with $u_{0}\in L^{3}(\mathbb{R}^3)$) requires the use of Gronwall's lemma. 

In the strategy in \cite{Tao19} the lower bound on vorticity \eqref{vortlowerboundtao}, which is needed for getting a lower bound on the localized $L^{3}$ norm at $t_{0}$ via quantitative Carleman inequalities, is obtained from the frequency bubbles of concentration. 
In order for this transfer of scale-invariant information to take place, it appears essential that the solution has a scale-invariant control such as $\|u\|_{L^{\infty}_{t}L^{3}_{x}}\leq A.$
In our strategy, we instead work directly with quantities involving vorticity (similar to \eqref{vortlowerboundtao}), which are tailored for the immediate use of quantitative Carleman inequalities. In this way, we crucially avoid any need to transfer scale-invariant information, giving our strategy a certain degree of robustness.
\subsection{Final Remarks and Comments}
We give some heuristics about the quantitative estimates of the form \begin{equation}\label{quantgenform}
\|u\|_{L^{\infty}(\mathbb{R}^3\times (\frac{1}{2}, 1))}\leq G(\|u\|_{X})
\end{equation} that one can expect for the Navier-Stokes equations, when a finite-energy solution $u$ solution belongs to certain normed spaces $X\subset\mathcal{S}'(\mathbb{R}^3\times (0,1))$.
\subsubsection{Subcritical case} 
 Consider a space $X\subset\mathcal{S}'(\mathbb{R}^3\times (0,1))$ whose norm $\|\cdot\|_{X}$  is subcritical\footnote{A quantity $F(u)\in [0,\infty)$ is said to be \textit{subcritical} if, for the rescaling \eqref{NSErescale}, there exists $\alpha>0$ such that
$F(u_{\lambda})=\lambda^{\alpha} F(u)$.} (for example $L^{5+\delta}_{x,t}(\mathbb{R}^3\times (0,1))$ with $\delta>0$). If $u$ is a finite-energy solution with a finite subcritical norm on $\mathbb{R}^3\times (0,1)$, then it is known that $u$ must be belong to $C^{\infty}(\mathbb{R}^3\times (0,1])$. See, for example, \cite{ladyzhenskaya1967uniqueness}. Moreover, one typically has a quantitative estimate of the form \eqref{quantgenform} with 
$$G(x)= cx^{\beta}\,\,\,\textrm{with}\,\,\,\beta>0. $$ 
To demonstrate this, consider $u$ belong to $L^{5+\delta}_{x,t}(\mathbb{R}^3\times (0,1))$. An application of Caffarelli, Kohn and Nirenberg's result \cite{CKN82} (see also Proposition \ref{CKN}) gives that \eqref{quantgenform} holds true with $G(x)\sim x^{\frac{\delta+5}{\delta}}$. Such a quantitative estimate is invariant with respect to the Navier-Stokes scaling \eqref{NSErescale}. In this context, one could also gain similar quantitative estimates based on parabolic bootstrap arguments applied to the vorticity equation
\begin{equation}\label{vort}
\partial_{t}\omega-\Delta\omega=\omega\cdot\nabla u-u\cdot\nabla\omega,\,\,\,\,\omega=\nabla\times u
\end{equation}
as was done by Serrin in \cite{Serrin1962}.
\subsubsection{Critical case}
In the subcritical norm case, we saw that seeking estimates of the form \eqref{quantgenform} that are invariant with respect to the scaling \eqref{NSErescale}, gives a suitable candidate for $G$ that can be realised.
The case when the norm $\|\cdot\|_{X}$ is critical is more subtle, since a scaling argument does not provide a suitable candidate for $G$. 
We first mention that the case of sufficiently small critical norms, for example
\begin{equation}\label{smallL5}
\|u\|_{L^{5}(\mathbb{R}^3\times (0,1))}<\varepsilon_{0},
\end{equation}
is essentially of a similar category to the subcritical case (though a scaling argument is not applicable).
Indeed, a similar argument outlined as before (based on \cite{CKN82}, see also Proposition \ref{CKN})  gives that in this case we have \eqref{quantgenform} with $G(x)\sim x$. Is this consistent with the fact that solutions with small scale-invariant norms exhibit similar behaviour to the linear system and hence are typically expected to satisfy linear estimates.\newline
For obtaining quantitative estimates of the form \eqref{quantgenform} when the scale-invariant norm is large, it is less clear what the candidate for $G$ might be. This seems to be the case even for large \textit{global} scale-invariant norms that exhibit smallness at small \textit{local} scales\footnote{In particular, $u\in L^{5}(\mathbb{R}^3\times (0,1))\Rightarrow \lim_{r\downarrow 0}\|u\|_{L^{5}(B_{0}(r)\times (-r^2,0))}=0.$} (for example $L^{5}(\mathbb{R}^3\times (0,1))$). Such local smallness properties have been utilized to prove \textit{qualitative} regularity by essentially linear methods. See \cite{struwe1988partial}, for example. \newline
For the case of a  smooth finite-energy solution $u$ having finite scale-invariant $L^{5}(\mathbb{R}^3\times (0,1))$ norm, one way to obtain quantitative estimates\footnote{For very similar computations, see (for example) Chapter 11 of \cite{lemarie2016navier} and references therein.} is to consider the vorticity equation \eqref{vort} with initial vorticity $\omega_{0}\in L^{2}(\mathbb{R}^3)$. Performing an energy estimate yields for $t\in [0,1]$
\begin{equation}\label{vortenergyest}
\|\omega(\cdot,t)\|_{L^{2}(\mathbb{R}^3)}^2+2\int\limits_{0}^{t}\int\limits_{\mathbb{R}^3} |\nabla \omega|^2 dxdt'=\|\omega_{0}\|_{L^{2}(\mathbb{R}^3)}^2+ 2\int\limits_{0}^{t}\int\limits_{\mathbb{R}^3} (\omega\cdot\nabla u)\cdot\omega dxdt',
\end{equation} 
where the second term in right-hand side 
 is due to the \textit{vortex stretching term} $\omega\cdot\nabla u$ in \eqref{vort}. For the case that $u\in L^{5}(\mathbb{R}^3\times (0,1))$, application of H\"{o}lder's inequality, Lebesgue interpolation, Sobolev embedding theorems and Young's inequality lead to 
\begin{equation}\label{vortenergyestprodiserrin}
\|\omega(\cdot,t)\|_{L^{2}(\mathbb{R}^3)}^2+\int\limits_{0}^{t}\int\limits_{\mathbb{R}^3} |\nabla \omega|^2 dxdt'\leq \|\omega_{0}\|_{L^{2}(\mathbb{R}^3)}^2+ \int\limits_{0}^{t} \|u(\cdot,t')\|_{L^{5}(\mathbb{R}^3)}^5 \|\omega(\cdot,t')\|_{L^{2}(\mathbb{R}^3)}^2 dt'.
\end{equation}
Gronwall's lemma, followed by arguments similar to the subcritical case, yields
\begin{equation}\label{vortgronwall}
\|u\|_{L^{\infty}(\mathbb{R}^3\times (\frac{1}{2},1))}\lesssim \|\omega\|_{L^{\infty}(0,1;L^{2}(\mathbb{R}^3))}^2\leq \|\omega_{0}\|_{L^{2}(\mathbb{R}^3)}^2 \exp(\|u\|^{5}_{L^{5}(\mathbb{R}^3\times (0,1))}).
\end{equation}
Though this is not exactly of the form \eqref{quantgenform}, a slightly different argument gives that for any finite-energy solution $u$ in $L^{5}(\mathbb{R}^3\times (0,1))$ we get that \eqref{quantgenform} holds with $G(x)\sim \exp(O(1)x^{5})$. In particular, this can be achieved using $L^{q}$ energy estimates in \cite{montgomery2005conditions}, the pidgeonhole principle and reasoning in the previous subsection.

The above argument \eqref{vortenergyest}-\eqref{vortgronwall} shows that being able to substantially improve upon $G(x)\sim \exp(O(1)x^{5})$ would most likely require the utilization of a nonlinear mechanism that reduces the influence of the vortex stretching term $\omega\cdot\nabla u$ in \eqref{vort}. It seems plausible that the discovery of such a mechanism would have implications for the regularity theory of the Navier-Stokes equations (such as Type I blow-ups). 

\subsection{Outline of the paper}

In each of the Sections \ref{sec.2}-\ref{sec.6}, we distinguish between cases where: (i) one assumes a Type I control on the solution and (ii) one assumes a control on the velocity field on time slices only. 

In Section \ref{sec.2}, we state our main quantitative estimates (Propositions \ref{prop.main} and \ref{prop.maints}) and we demonstrate how these statements imply the main results of this paper: Theorem \ref{theologL3}, Corollary \ref{optimalrateDSS} and Theorem \ref{theo.mainbis}. Section \ref{sec.3} is devoted to the proof of Propositions \ref{prop.main} and \ref{prop.maints}. Section \ref{sec.further} contains three further applications of the technology developed in the present paper, in particular Corollary \ref{cor.number} concerning a quantitative bound for the number of singularities in a Type I blow-up scenario. In Section \ref{sec.4}, we quantify Jia and \v{S}ver\'{a}k's results regarding local-in-space short-time smoothing, which is a main tool for proving the quantitative estimates in Section \ref{sec.3}. The main result in Section \ref{sec.4} is Theorem \ref{theo.locshortime}. 
In Section \ref{sec.quantannulus}, we give a new proof of Tao's result that solutions possess `quantitative annuli of regularity', which is required for proving our main propositions in Section \ref{sec.3}. The central results in Section \ref{sec.quantannulus} are Lemma \ref{lem.ga} and Lemma \ref{lem.annulusts}.  Section \ref{sec.6} is concerned with the utilization of arguments from the papers of Leray and Tao to show existence of quantitative epochs of regularity (Lemma \ref{epochTypeI} and Lemma  \ref{epochtimeslice}). In Appendix \ref{sec.A} we recall known results about mild solutions and local energy solutions, and we give pressure formulas. In Appendix \ref{sec.B}, we recall the quantitative Carleman inequalities proven by Tao. 

\subsection{Notations}
\label{subsec.not}
\subsubsection{Universal constants}
For universal constants in the statements of propositions and lemmas associated to the Type I case (specifically Proposition \ref{prop.main} and Lemma \ref{lem.backconc}), we adopt the convention of a superscript $\sharp$.
For universal constants in the statements of propositions and lemmas associated to the Type I case (specifically Proposition \ref{prop.maints} and Lemma \ref{lem.backconctimeslice}), we adopt the convention of a superscript $\flat$.

In Lemma \ref{lem.backconc}, Lemma \ref{lem.backconctimeslice} and Section \ref{sec.4}, we track the numerical constants arising. 
Elsewhere in this paper, we adopt the convention that $C$ denotes a positive universal constant which may vary from line to line.

We use the notation $X\lesssim Y$, which means that there exists a positive universal constant $C$ such that $X\leq CY.$ 

In several places in this paper (notably Section \ref{sec.3} and Appendix \ref{sec.B}) the notation $O(1)$ is used to denote a positive universal constant and $-O(1)$ denotes a negative universal constant.

Whenever we refer to a quantity ($M$ for example) being `sufficiently large', we understand this as $M$ being larger than some universal constant that can (in principle) be specified.

\subsubsection{Vectors and Domains}
For a vector $a$, $a_{i}$ denotes the $i^{th}$ component of $a$. 

For $(x,t)\in\mathbb{R}^4$ and $r>0$ we denote
$B_{x}(r):=\{y\in\mathbb{R}^3: |y-x|<r \}$ and $Q_{(x,t)}(r):= B_{r}(x)\times (t-r^2,t)$. Here, $|\cdot|$ denotes the Euclidean metric. As is usual, for $a,\, b\in\R^3$, $(a\otimes b)_{\alpha\beta}=a_\alpha b_\beta$, and for $A,\, B\in M_3(\R)$, $A:B=A_{\alpha\beta}B_{\alpha\beta}$. Here and in the whole paper we use Einstein's convention on repeated indices.  For $F:\Omega\subseteq\mathbb{R}^3\rightarrow\mathbb{R}^3$, we define $\nabla F\in M_{3}(\mathbb{R})$ by $(\nabla F(x))_{\alpha\beta}:= \partial_{\beta} F_{\alpha}$. 

\textbf{Let us stress that in Section \ref{sec.4} only} we use cubes instead of balls: $B_{x}(r)=x+(-r,r)^3$. This is for computational convenience, since we track numerical constants in Section \ref{sec.4}. We emphasize that the results in Section \ref{sec.4} hold for spherical balls too, with certain universal constants adjusted.
\subsubsection{Mild, suitable and finite-energy solutions to the Navier-Stokes equations}\label{subsec.def}

Throughout this paper, we refer to $u: \mathbb{R}^3\times [0,T]$ as a \textit{mild solution} of the Navier-Stokes equations \eqref{e.nse} if it satisfies the Duhamel formula:
\begin{equation*}
u(x,t)=e^{t\Delta}u(\cdot,0)+\int\limits_{0}^{t}\mathbb{P}\partial_{i}e^{(t-s)\Delta}u_{i}(\cdot,s)u_{j}(\cdot,s)ds, 
\end{equation*}  
for all $t\in[0,T]$. 
Here,  $e^{t\Delta}$ is the heat semigroup, $\mathbb{P}$ is the projection onto divergence-free vector fields. A mild solution on $[0,T^*)$ is a function that is a mild solution on $[0,T]$ for any $T\in(0,T^*)$.

Let $\Omega\subset\mathbb{R}^3$. We say that $(u,p)$ is a \textit{suitable weak solution} to the Navier-Stokes equations \eqref{e.nse} in $\Omega\times (T_{1},T)$ if it fulfills the properties described in \cite{gregory2014lecture} (Definition 6.1 p.133 in \cite{gregory2014lecture}).

We say that $u$ is a \textit{suitable finite-energy solution} to the Navier-Stokes equations on $\mathbb{R}^3\times (T_{1},T)$
  if it is a solution to \eqref{e.nse} in the sense of distributions and
\begin{itemize}
\item $u\in C_{w}([T_{1},T]; L^{2}_{\sigma}(\mathbb{R}^3))\cap L^{2}_{t}(T_{1},T; \dot{H}^{1}(\mathbb{R}^3))$,
\item it satisfies the global energy inequality
\begin{equation}\label{energyinequalityLeray}
\|u(\cdot,t)\|_{L^{2}}^2+2\int\limits_{T_{1}}^{t}\int\limits_{\mathbb{R}^3}|\nabla u|^2 dyds\leq \|u(\cdot,T_{1})\|_{L^{2}}^2\,\,\,\,\textrm{for}\,\,\textrm{all}\,\,t\in [T_{1},T],
\end{equation}
\item $(u,p)$ is a \textit{suitable weak solution} on $B_{1}(x)\times (T_{1},T)$ for all $x\in\mathbb{R}^3.$
\end{itemize}
It is known that the above defining properties of suitable weak solutions imply that there exists $\Sigma\subset [T_{1},T]$ with 
full Lebesgue measure $|\Sigma|=T-T_{1}$ such that 
\begin{multline}\label{energyinequalityturbulent}
\|u(\cdot,t)\|_{L^{2}}^2+2\int\limits_{t'}^{t}\int\limits_{\mathbb{R}^3}|\nabla u|^2 dyds\leq \|u(\cdot,t')\|_{L^{2}}^2\,\,\,\,\textrm{and}\,\,\,\,\|\nabla u(\cdot,t')\|_{L^{2}}<\infty,\\
\textrm{for}\,\,\textrm{all}\,\,t\in [t',T]\,\,\,\textrm{and}\,\,\,t'\in\Sigma.
\end{multline}

\subsubsection{Lorentz spaces}
For a measurable subset $\Omega\subseteq\mathbb{R}^{d}$ and a measurable function $f:\Omega\rightarrow\mathbb{R}$ we define
\begin{equation}\label{defdist}
d_{f,\Omega}(\alpha):=\mu(\{x\in \Omega : |f(x)|>\alpha\}),
\end{equation}
where $\mu$ denotes the Lebesgue measure. The Lorentz space $L^{p,q}(\Omega)$, with $p\in [1,\infty[$, $q\in [1,\infty]$, is the set of all measurable functions $g$ on $\Omega$ such that the quasinorm $\|g\|_{L^{p,q}(\Omega)}$ is finite. The quasinorm is defined by
\begin{equation}\label{Lorentznorm}
\|g\|_{L^{p,q}(\Omega)}:= \Big(p\int\limits_{0}^{\infty}\alpha^{q}d_{g,\Omega}(\alpha)^{\frac{q}{p}}\frac{d\alpha}{\alpha}\Big)^{\frac{1}{q}},
\end{equation}
\begin{equation}\label{Lorentznorminfty}
\|g\|_{L^{p,\infty}(\Omega)}:= \sup_{\alpha>0}\alpha d_{g,\Omega}(\alpha)^{\frac{1}{p}}.
\end{equation}\\
Notice that there exists a norm, which is equivalent to the quasinorm defined above, for which $L^{p,q}(\Omega)$ is a Banach space. 
For $p\in [1,\infty)$ and $1\leq q_{1}< q_{2}\leq \infty$, we have the following continuous embeddings 
\begin{equation}\label{Lorentzcontinuousembedding}
L^{p,q_1}(\Omega) \hookrightarrow  L^{p,q_2}(\Omega)
\end{equation}
and the inclusion is strict.

\section{Main quantitative estimates}
\label{sec.2}

\subsection{Quantitative estimates in the Type I and time slices case}

\begin{proposition}[main quantitative estimate, Type I]\label{prop.main}
There exists a universal constant $M_2\in[1,\infty)$. Let $M\in[M_2,\infty)$, $t_0\in\R$ and $T\in(0,\infty)$. There exists $S^\sharp(M)\in(0,\frac14]$, such that the following holds. Let $(u,p)$  
be a smooth solution with sufficient decay\footnote{See footnote \ref{footdefsol}. Notice that by this definition $u$ is bounded up to $t_{0}$.} to the Navier-Stokes equations \eqref{e.nse} in $I=[t_0-T,t_0]$, which satisfies
\begin{equation}\label{e.t1A}
\|u\|_{L^\infty_tL^{3,\infty}_x(\R^3\times(t_0-T,t_0))}\leq M.
\end{equation}
Assume that there exists $t_0'\in [t_0-T,t_0)$ such that $t_0'$ is not too close to $t_0-T$ in the sense
\begin{equation*}
0\leq \frac{t_0-t_0'}{T}<C^\sharp M^{-548}
\end{equation*}
and such that the vorticity concentrates at time $t_0'$ in the following sense
\begin{equation*}
\int\limits_{B_0(4\sqrt{S^\sharp}^{-1}(t_0-t_0')^\frac12)}|\omega(x,t_0')|^2\, dx> M^2(t_0-t_0')^{-\frac12}\sqrt{S^\sharp}.
\end{equation*}
Then, we have the following lower bound
\begin{equation}
\frac{t_0-t_0'}{T}\geq \frac{C^\sharp}{8}M^{-749}\exp\Bigg\{-\exp(\exp(M^{1024}))\int\limits_{B_{0}(\exp(M^{1023})T^{\frac{1}{2}})}|u(x,t_{0})|^3\, dx\Bigg\}.
\end{equation}
Furthermore, for 
\begin{equation}\label{s0def}
-s_{0}:=\frac{C^\sharp}{16}M^{-749}\exp\Bigg\{-\exp(\exp(M^{1024})))\int\limits_{B_{0}(\exp(M^{1023})T^{\frac{1}{2}})}|u(x,t_{0})|^3\, dx\Bigg\},
\end{equation}
we also have the bound
\begin{equation}\label{Linfinitybound}
\|u\|_{L^{\infty}\big(B_{0}\big({C_{2}T^{\frac{1}{2}}M^{50}}(-s_{0})^{\frac{1}{2}}\big)\times \big(\tfrac{s_{0}T}{4}+t_{0}, t_{0}\big)\big)}\leq \frac{C_{1}M^{-23}}{(-s_{0})^{\frac{1}{2}} T^{\frac{1}{2}}},
\end{equation}
for universal constants $C_1, C_{2}\in(0,\infty)$. Here $C^\sharp$ and $S^\sharp(M)$ are the constants given by Lemma \ref{lem.backconc}.

Furthermore, if for fixed $\lambda\in (0, \exp(M^{1023}))$ we additionally assume that
\begin{equation}\label{L3lowerboundassumption}
\int\limits_{B_{0}(\lambda T^{\frac{1}{2}})}|u(x,t_{0})|^3 dx\geq \frac{3}{2}\exp(-\exp(M^{1024}))
\end{equation}
then we instead have the lower bound
\begin{equation}\label{lowerwithaddas}
\frac{t_0-t_0'}{T}\geq \frac{C^\sharp \lambda^2}{8}M^{-749}\exp\Bigg\{-4M^{1023}\exp(\exp(M^{1024}))\int\limits_{B_{0}(\lambda T^{\frac{1}{2}})}|u(x,t_{0})|^3\, dx\Bigg\}.
\end{equation}
Furthermore, for 
\begin{equation}\label{s1def}
-s_{1}:=\frac{C^\sharp\lambda^2}{16}M^{-749}\exp\Bigg\{-4M^{1023}\exp(\exp(M^{1024}))\int\limits_{B_{0}(\lambda T^{\frac{1}{2}})}|u(x,t_{0})|^3\, dx\Bigg\},
\end{equation}
we also have the bound
\begin{equation}\label{Linfinitybounds1}
\|u\|_{L^{\infty}\big(B_{0}\big(C_{2}{T^{\frac{1}{2}}M^{50}}(-s_{1})^{\frac{1}{2}}\big)\times \big(\tfrac{s_{1}T}{4}+t_{0}, t_{0}\big)\big)}\leq \frac{C_{1}M^{-23}}{(-s_{1})^{\frac{1}{2}} T^{\frac{1}{2}}}.
\end{equation}
\end{proposition}

Figure \ref{fig.summary} illustrates Proposition \ref{prop.main}.

\begin{proposition}[main quantitative estimate, time slices]\label{prop.maints}
There exists a universal constant $M_{1}\in[1,\infty)$. Let $M\in[M_{1},\infty)$. 
We define $\Mp$ by \eqref{e.defM'theo}. 
There exists $S^\flat(M)\in(0,\frac14]$, such that the following holds. Let $(u,p)$ 
be a $C^{\infty}(\mathbb{R}^3\times (-1,0))$ 
finite-energy solution to the Navier-Stokes equations \eqref{e.nse} in $I=[-1,0]$. 
Assume  that there exists $t_{(k)}\in [-1,0)$ such that
\begin{equation}\label{e.tsA}
t_{(k)}\uparrow 0\,\,\,\,\textrm{with}\,\,\,\, \sup_{k} \|u(\cdot, t_{(k)})\|_{L^{3}(\mathbb{R}^3)}\leq M.
\end{equation}
Select any ``well-separated'' subsequence (still denoted $t_{(k)}$) such that
\begin{equation}\label{wellseppropts}
\sup_{k}\frac{-t_{(k+1)}}{-t_{(k)}}<\exp(-2(\Mp)^{1223}).
\end{equation}
For this well-separated subsequence, assume that there exists $j+1$ such that the vorticity concentrates at time $t_{(j+1)}$ in the following sense
\begin{equation}\label{e.concts}
\int\limits_{B_0(4\sqrt{S^\flat}^{-1}(-t_{(j+1)})^\frac12)}|\omega(x,t_{(j+1)})|^2\, dx> M^2(-t_{(j+1)})^{-\frac12}\sqrt{S^\flat}.
\end{equation}
Then, we have the following upper bound on $j$
\begin{equation}
j\leq \exp(\exp(({\Mp})^{1224}))
\end{equation}
Here $S^\flat(M)$ is the constant given by Lemma \ref{lem.backconctimeslice}. 
\end{proposition}

\subsection{Proofs of the main results}
\label{sec.proofmainrests}

In this section we prove the main results stated in the Introduction.

\begin{proof}[Proof of Theorem \ref{theologL3}]
Take $M\geq M_{0}$. Here, $M_{0}:=\max(M_{2},M_{6})\geq 1$ with $M_{2}$ being from Proposition \ref{prop.main} and $M_{6}$ being from Corollary \ref{corL_inftyconc}. We prove here a slightly stronger statement than \eqref{L3localisedlog}, namely
\begin{equation}\label{L3localisedlogbis}
\int\limits_{B_{0}\big((T^*)^{\frac{1}{2}}(T^*-t)^{\frac{1-\delta}{2}}\big)} |u(x,t)|^3 dx\geq
\frac{\log\Big(\frac{1}{(T^*-t)^{\frac{\delta}{2}}}\Big)}{4M^{1023}\exp(\exp(M^{1024}))}.
\end{equation}
First we note that \cite{BP18} (specifically Theorem 2 in \cite{BP18}), together with assumptions (1)-(2) in the statement of Theorem \ref{theologL3} imply that there exist
 $ S_{BP}(M)$ and $\gamma_{univ}>0$ such that
\begin{equation*}
\int\limits_{B_{0}\big(2\sqrt{\frac{T^*-t}{S_{BP}(M)}}\big)} |u(x,t)|^3 dx\geq \gamma_{univ}^3\,\,\,\textrm{for}\,\,\textrm{all}\,\,t\in (0, T^*).
\end{equation*}
Thus we have
\begin{equation}\label{L3concBP}
\int\limits_{B_{0}(t^{\frac{1}{2}}(T^*-t)^{\frac{1-\delta}{2}})} |u(x,t)|^3 dx\geq \gamma_{univ}^3\,\,\,\,\textrm{for\,\,all}\,\,\, t\in \Big(\max\Big( T^*-\Big(\frac{T^*S_{BP}(M)}{8}\Big)^{\frac{1}{\delta}},\frac{T^*}{2}\Big),T^*\Big).
\end{equation}
Suppose for contradiction that \eqref{L3localisedlogbis}
 does not hold for some 
\begin{equation}\label{tcond}
t\in\big(\max(\tfrac{T^*}{2},T^*-c(\delta,M,T^*)), T^*\big),
\end{equation} 
where 
\begin{equation}\label{defCdeltaMT}
c(\delta,M,T^*):=\min\Big(c(\delta)(T^*M^{-802})^{\frac{2}{\delta}}, \Big(\frac{T^*S_{BP}(M)}{8}\Big)^{\frac{1}{\delta}} ,\exp(2M^{1023})\Big)
\end{equation}
for an appropriate $c(\delta)\in(0,\infty).$ This implies that for $M$ sufficiently large
\begin{equation}\label{maintheocontraimp}
\int\limits_{B_{0}(t^{\frac{1}{2}}(T^*-t)^{\frac{1-\delta}{2}})} |u(x,t)|^3 dx<\frac{1}{4M^{1023}\exp(\exp(M^{1024}))}\log\Big(\frac{1}{(T^*-t)^{\frac{\delta}{2}}}\Big).
\end{equation}
Considering $u$ on $\mathbb{R}^3\times [0,t]$ and observing Proposition \ref{prop.main}, we see that \eqref{L3concBP} implies that the assumption \eqref{L3lowerboundassumption} is satisfied with $\lambda:= (T^*-t)^{\frac{1-\delta}{2}}$ and $T:=t$. 
Furthermore, by \eqref{tcond}-\eqref{defCdeltaMT} we have that $\lambda\in (0, \exp(M^{1023})).$
 Hence, we can apply Proposition \ref{prop.main}. 
Namely by \eqref{Linfinitybounds1} for
$$
-s_{1}:=\frac{C^\sharp(T^*-t)^{1-\delta}}{16}M^{-749}\exp\Bigg\{-4M^{1023}\exp(\exp(M^{1024}))\int\limits_{B_{0}( t^{\frac{1}{2}}(T^*-t)^{\frac{1-\delta}{2}})}|u(x,t)|^3\, dx\Bigg\},
$$
we have the bound
\begin{equation}\label{locLinfinitymaintheo}
\|u(\cdot,t)\|_{L^{\infty}\big(B_{0}\big(C_{2}{t^{\frac{1}{2}}M^{50}}(-s_{1})^{\frac{1}{2}}\big)\big)}\leq \frac{C_{1}M^{-23}}{(-s_{1})^{\frac{1}{2}} t^{\frac{1}{2}}}.
\end{equation}
Using that $(0,T^*)$ is a singular point of $u$, the Type I bound on $u$ and Corollary \ref{corL_inftyconc}, we see that there exists a universal constant $C_{univ}$ such that
\begin{equation}\label{Linfinityconcmaintheo}
\|u(\cdot,t)\|_{L^{\infty}\big(B_{0}\big(\frac{2 M^{50}}{C_{univ}} (T^*-t)^{\frac{1}{2}}\big)\big)}>\frac{C_{univ}M^{-49}}{(T^*-t)^{\frac{1}{2}}}.
\end{equation}
From our contradiction assumption (which implies \eqref{maintheocontraimp}) we see that
\begin{equation}\label{contras1}
-{s}_{1}> M^{-750}(T^*-t)^{1-\frac{\delta}{2}}
\end{equation}
Using this and \eqref{tcond} for an appropriate $c(\delta)\in (0,\infty)$, we get that
$$B_{0}\Big(\frac{2M^{50}}{C_{univ}}(T^*-t)^{\frac{1}{2}}\Big)\subset B_{0}\Big(\frac{(T^*)^{\frac{1}{2}}C_{2}M^{50}}{2^{\frac12}}(-{s}_{1})^{\frac{1}{2}}\Big)$$
and 
$$\frac{2^{\frac12}C_{1}M^{-23}}{(-{s}_{1})^{\frac{1}{2}} (T^*)^{\frac{1}{2}}}\leq \frac{C_{univ}M^{-49}}{(T^*-t)^{\frac{1}{2}}}.$$
With these two facts, we see that \eqref{locLinfinitymaintheo} contradicts \eqref{Linfinityconcmaintheo}.
\end{proof}

\begin{proof}[Proof of Corollary \ref{optimalrateDSS}]
From \cite{chae2017removing} (specifically Theorem 1.1 in \cite{chae2017removing}), there exists $M>1$ such that
\begin{equation}\label{DSSboundchaewolf}
|u(x,t)|\leq \frac{M}{|x|+\sqrt{-t}}\,\,\,\textrm{for}\,\,\textrm{all}\,\,\,(x,t)\in (\mathbb{R}^3\times (-\infty,0])\setminus\{(0,0)\}.
\end{equation}
Integration of this then immediately gives the upper bound of \eqref{L3localisedlogDSS}, which in fact holds true for $t\in (-1,0)$.
Next, note that \eqref{DSSboundchaewolf} implies 
\begin{equation}\label{TypeIDSS}
\|u\|_{L^{\infty}_{t}L^{3,\infty}_{x}(\mathbb{R}^3\times (-\infty,0))}\leq M. 
\end{equation}
We also remark that since $u$ is non-zero and $\lambda$-DSS we must have
\begin{equation}\label{DSSsing} u\notin L^{\infty}_{x,t}(Q_{(0,0)}(r))\,\,\,\textrm{for}\,\,\textrm{all}\,\,\textrm{sufficiently}\,\,\textrm{small}\,\,r. \end{equation}
Indeed, suppose for contradiction that $u\in L^{\infty}_{x,t}(Q_{r}(0,0))$ then for any $(x,t)\in \mathbb{R}^3\times (-\infty,0)$ we have $(\lambda^{-k}x, \lambda^{-2k} t)\in Q_{r}(0,0)$ for all sufficiently large $k$. Using that $u$ is $\lambda$-DSS we have 
$$|u(x,t)|=|\lambda^{-k}u(\lambda^{-k}x, \lambda^{-2k}t)|\leq \lambda^{-k}\|u\|_{L^{\infty}(Q_{(0,0)}(r))}\downarrow 0. $$
We then see that \eqref{DSSboundchaewolf}-\eqref{DSSsing} allow us to apply Theorem \ref{theologL3} on $\mathbb{R}^3\times (-1,0)$ to get the lower bound in \eqref{L3localisedlogDSS}.
\end{proof}

\begin{proof}[Proof of Theorem \ref{theo.mainbis}]
Applying Proposition \ref{prop.maints} we see that for $j=\lceil\exp(\exp({(\Mp)^{1224}}))\rceil+1 $
 we have the contrapositive of \eqref{e.concts}.
In particular,
$$
\int\limits_{B_0(4\sqrt{S^\flat}^{-1}(-t_{(j+1)})^\frac12)}|\omega(x,t_{(j+1)})|^2\, dx< M^2(-t_{(j+1)})^{-\frac12}\sqrt{S^\flat}.
$$ 
Almost identical arguments to those utilized in the proof of Lemma \ref{lem.backconc}, except using the bound \eqref{e.conclbddu} instead of \eqref{e.conclbddnablau}, give
$$\|u\|_{L^{\infty}\big(B_{0}\big(C_{2}{M^{50}}(-t_{(j+1)})^{\frac{1}{2}}\big)\times \big(\tfrac{t_{(j+1)}}{4},0\big)\big)}\leq \frac{C_1M^{-23}}{(-t_{(j+1)})^{\frac{1}{2}}}.$$
Since all estimates are independent of the spatial point where \eqref{e.concts} occurs, we conclude that
$$\|u\|_{L^{\infty}\big(\mathbb{R}^3\times \big(\tfrac{t_{(j+1)}}{4},0\big)\big)}\leq \frac{C_1M^{-23}}{(-t_{(j+1)})^{\frac{1}{2}}}.$$
This concludes the proof of the theorem.
\end{proof}

\section{Proofs of the main quantitative estimates}
\label{sec.3}

\subsection{Backward propagation of concentration}

Here we state two pivotal results. These are concerned with backward propagation of concentration in the Type I case and in the time slices case. Figure \ref{fig.summary} illustrates Lemma \ref{lem.backconc} below. 

\begin{lemma}[backward propagation of concentration, Type I]\label{lem.backconc}
There exists two universal constants 
 $C^\sharp\in (0,\frac{1}{16}),\, M_3\in[1,\infty)$. 
For all $M\in[M_3,\infty)$, there exists $S^\sharp(M)\in(0,\frac14]$,  such that the following holds. Let $(u,p)$ 
be a `smooth solution with sufficient decay'\footnote{See footnote \ref{footdefsol}.} of the Navier-Stokes equations \eqref{e.nse} in $I=[-1,0]$ satisfying the Type I bound \eqref{e.t1A}. 
Assume that there exists $t_0'\in [-1,0)$ such that $t_0'$ is not too close to $-1$ in the sense
\begin{equation*}
0< -t_0'<C^\sharp M^{-548}
\end{equation*}
and such that the vorticity concentrates at time $t_0'$ in the following sense
\begin{equation}\label{e.conct_0'}
\int\limits_{B_0(4\sqrt{S^\sharp}^{-1}(-t_0')^\frac12)}|\omega(x,t_0')|^2\, dx> M^2(-t_0')^{-\frac12}\sqrt{S^\sharp}.
\end{equation}
Then, the vorticity concentrates in the following sense
\begin{equation}\label{e.conct_0''}
\int\limits_{B_0(4\sqrt{S^\sharp}^{-1}(-t_0'')^\frac12)}|\omega(x,t_0'')|^2\, dx> M^2(-t_0'')^{-\frac12}\sqrt{S^\sharp}
\end{equation}
at any well-separated backward time $t_0''\in [-1,t_0']$ such that
\begin{equation}\label{e.wellsep}
\frac{-t_0'}{-t_0''}<C^\sharp M^{-548}.
\end{equation}
Here $S^\sharp (M)$ is defined explicitly by \eqref{e.defS*two} and we have $S^\sharp (M)=O(1)M^{-100}$. 
\end{lemma}

\begin{proof}[Proof of Lemma \ref{lem.backconc}]
The proof is by contradiction. It relies on Theorem \ref{theo.locshortime} below about local-in-space short-time smoothing. We define $S^\sharp \in (0,\frac14]$ in the following way:
\begin{equation}\label{e.defS*two}
S^\sharp =S^\sharp (M):=S_*\big(C_{weak}M,32C_{Sob}(1+C_{ellip})C_{weak}M\big),
\end{equation}
where $S_*$ is the constant defined in Theorem \ref{theo.locshortime} (see also the formula \eqref{e.defS*}), $C_{Sob}\in (0,\infty)$ is the best constant in the Sobolev embedding $H^1(B_0(2))\subset L^6(B_0(2))$ and $C_{ellip}\in (0,\infty)$ is the best constant in the estimate
\begin{equation}\label{e.ellipest}
\|\nabla U(\cdot,0)\|_{L^2(B_0(2))}\leq C_{ellip}\left(\|\Omega(\cdot,0)\|_{L^2(B_0(4))}+\|U(\cdot,0)\|_{L^2(B_0(4))}\right)
\end{equation}
for weak solutions to
\begin{equation*}
-\Delta U(\cdot,0)=\nabla\times\Omega(\cdot,0)\quad\mbox{in}\quad B_0(4).
\end{equation*}
Furthermore, $C_{weak}\in [1,\infty)$ is a universal constant from the embedding $L^{3,\infty}(\mathbb{R}^3)\subset L^{2}_{uloc}(\mathbb{R}^3)$.  
See, for example, Lemma 6.2 in Bradshaw and Tsai's paper \cite{BT19}.

Assume that 
\begin{equation*}
\int\limits_{B_0(4\sqrt{S^\sharp }^{-1}(-t_0'')^\frac12)}|\omega(x,t_0'')|^2\, dx\leq M^2(-t_0'')^{-\frac12}\sqrt{S^\sharp }
\end{equation*}
and 
\begin{equation*}
\int\limits_{B_0(4\sqrt{S^\sharp }^{-1}(-t_0')^\frac12)}|\omega(x,t_0')|^2\, dx> M^2(-t_0')^{-\frac12}\sqrt{S^\sharp }
\end{equation*}
for times $t_0',\, t_0''$ satisfying the condition that they are well-separated \eqref{e.wellsep}. 
Let $r:=\sqrt{S^\sharp }^{-1}(-t_0'')^\frac12$ and rescale in the following way $U(y,s):=ru(ry,r^2s+t_0'')$, 
$\Omega(y,s):=r^2\omega(ry,r^2s+t_0'')$. 
We have 
\begin{align*}
\|U(\cdot,0)\|_{L^6(B_0(2))}\leq\ &C_{Sob}\left(\|U(\cdot,0)\|_{L^2(B_0(2))}+\|\nabla U(\cdot,0)\|_{L^2(B_0(2))}\right)\\
\leq\ &C_{Sob}(1+C_{ellip})\left(\|U(\cdot,0)\|_{L^2(B_0(4))}+\|\Omega(\cdot,0)\|_{L^2(B_0(4))}\right)\\
\leq\ &C_{Sob}(1+C_{ellip})(8^\frac32+1)C_{weak}M\\
\leq\ &32C_{Sob}(1+C_{ellip})C_{weak}M.
\end{align*}
Here we used the scale-invariant bound \eqref{e.t1A} and  the embedding $L^{3,\infty}(\R^3)\subset L^2_{uloc}(\R^3)$.  Taking $M\geq M_3$ sufficiently large, we now apply the bound \eqref{e.conclbddnablau} with $C_{weak}M$ and $N:=32C_{Sob}(1+C_{ellip})C_{weak}M$. 
Using $C^{\sharp}\in (0,\frac{1}{16})$ and \eqref{e.wellsep}, we have $-t_0'<\frac1{16}(-t_0'')$. Therefore, we have  $t_0'\in (t_0''+\frac{15}{16}S^\sharp r^2,t_0''+S^\sharp r^2)$, hence
\begin{align*}
\int\limits_{B_{0}(4\sqrt{S^\sharp }^{-1}(-t_0')^\frac12)}|\omega(x,t_0')|^2\, dx\leq\ &\sup_{t\in (t_0''+\frac{15}{16}S^\sharp r^2,t_0''+S^\sharp r^2)}\int\limits_{B_{0}(4\sqrt{S^\sharp }^{-1}(-t_0')^\frac12)}|\omega(x,t)|^2\, dx\\
\leq\ &\sup_{t\in (t_0''+\frac{15}{16}S^\sharp r^2,t_0''+S^\sharp r^2)}\int\limits_{B_{0}(\frac14\sqrt{S^\sharp }^{-1}(-t_0'')^\frac12)}|\omega(x,t)|^2\, dx\\
\leq\ &r^{-1}\sup_{s\in (\frac{15}{16}S^\sharp ,S^\sharp )}\int\limits_{B_{0}(\frac 14)}|\Omega(y,s)|^2\, dx\\
\leq\ &CM^{65}N^{161}(-t_0'')^{-\frac12}\\
\leq\ &CM^{226}(-t_0'')^{-\frac12}.
\end{align*}
The contradiction follows then for a well chosen universal constant $C^\sharp \in (0,\frac{1}{16})$.
\end{proof}

\begin{remark}[Concentration of the enstrophy]\label{rem.concenstro}
A variant of the proof of Lemma \ref{lem.backconc} gives the following concentration result for the enstrophy near a Type I singularity. 

For all sufficiently large $M\in[1,\infty)$, let $S^\sharp(M)\in(0,\frac14]$ be the constant defined by \eqref{e.defS*two}.   Let $(u,p)$ be a suitable finite-energy solution\footnote{See Subsection \ref{subsec.def}.} of the Navier-Stokes equations \eqref{e.nse} in $I=[-1,0]$ satisfying the Type I bound \eqref{e.t1A}. 
Assume that the space-time point $(0,0)$ is a singularity for $u$. Then, for all $t'\in\Sigma$, where $\Sigma$ is a full measure subset of $[-1,0)$ defined in Subsection \ref{subsec.def}, the vorticity concentrates in the following sense
\begin{equation*}
\int\limits_{B_0(4\sqrt{S^\sharp}^{-1}(-t')^\frac12)}|\omega(x,t')|^2\, dx> M^2(-t')^{-\frac12}\sqrt{S^\sharp}.
\end{equation*} 
\end{remark}

\begin{lemma}[backward propagation of concentration, time slices]\label{lem.backconctimeslice}
There exists a universal constant 
$M_4\in[1,\infty)$ such that the following holds true. Let $M\in[M_4,\infty).$  
We define $\Mp$ by \eqref{e.defM'theo}.
 Fix any $\alpha\geq \Mp$ and let $t'_0, t''_{0}\in[-1,0)$ be such that 
 $$\frac{t''_{0}}{8\alpha^{201}}< t'_{0}<0.$$
  There exists $S^\flat (M)\in(0,\frac14]$,  such that the following holds. Let $(u,p)$ 
  be a $C^{\infty}(\mathbb{R}^3\times (-1,0))$ 
finite-energy solution of the Navier-Stokes equations \eqref{e.nse} in $I=[-1,0]$ satisfying 
\begin{equation}\label{timeslicecondition}
\|u(\cdot,t'_{0})\|_{L^{3}}\leq M\,\,\,\,\textrm{and}\,\,\,\,\|u(\cdot,t''_{0})\|_{L^{3}}\leq M.
\end{equation}
Suppose further that the vorticity concentrates at time $t_0'$ in the following sense
\begin{equation}\label{e.conct_0'timeslice}
\int\limits_{B_0(4\sqrt{S^\flat}^{-1}(-t_0')^\frac12)}|\omega(x,t_0')|^2\, dx> M^2(-t_0')^{-\frac12}\sqrt{S^\flat}.
\end{equation}
With the additional separation condition that
\begin{equation}\label{e.wellseptimeslice}
\frac{-t_0'}{-t_0''}<
\alpha^{-1051},
\end{equation}
the above assumptions imply that for  any $s_{0}\in [t''_{0}, \frac{t''_{0}}{8\alpha^{201}}]$  the vorticity concentrates in the following sense
\begin{equation}\label{e.concs_0}
\int\limits_{B_0(4(-s_{0})^\frac12 \alpha^{106})}|\omega(x,s_{0})|^2\, dx> \frac{(M+1)^2}{(-s_{0})^{\frac12}\alpha^{106}}.
\end{equation}
Here $S^\flat(M)=O(1)M^{-100}$. 
\end{lemma}
\begin{proof}
For $s\in [t''_{0},0]$ we decompose $u$ as 
\begin{equation}\label{udecomptimesliceconc}
u(\cdot,s)= e^{(s-t''_{0})\Delta}u(\cdot,t''_{0})+V(\cdot, s).
\end{equation} 
We then have 
\begin{equation}\label{linestconcslice}
\|e^{(s-t''_{0})\Delta}u(\cdot,t''_{0})\|_{L^{3}_{x}}\leq M.
\end{equation}
\begin{equation}\label{semigroupL3L4conc}
\|e^{(s-t''_{0})\Delta}u(\cdot,t''_{0})\|_{L^{4}_{x}}\leq \frac{CM}{(s_{0}-t''_{0})^{\frac{1}{8}}}.
\end{equation}
Furthermore, arguments from \cite{Giga86} imply that
\begin{equation}\label{semigroupL3L5conc}
\|e^{(t-t''_{0})\Delta}u(\cdot,t''_{0})\|_{L^{5}(\mathbb{R}^3\times (t''_{0},\infty))}\leq CM
\end{equation}
Moreover, similar arguments as those used in Proposition 2.2 of \cite{SeSv17}\footnote{Based on the energy method, Lebesgue interpolation, Sobolev embedding, H\"{o}lder's inequality and Young's inequality.}  yield that for $s\in [t''_{0},0]$
\begin{align*}
&\|V(\cdot,s)\|_{L^{2}_{x}}^{2}+\int\limits_{t''_{0}}^{s}\int\limits_{\mathbb{R}^3} |\nabla V|^2 dxdt
\\& \leq C\int\limits_{t''_{0}}^{s}\int\limits_{\mathbb{R}^3} |e^{(t-t''_{0})\Delta}u(\cdot,t''_{0})|^4 dxdt+C\int\limits_{t''_{0}}^{s}\|V(\cdot,t)\|_{L^{2}_{x}}^{2}\|e^{(t-t''_{0})\Delta}u(\cdot,t''_{0})\|_{L^{5}_{x}}^5 dt .
\end{align*}
Using \eqref{linestconcslice}-\eqref{semigroupL3L5conc} and Gronwall's lemma, we infer (for $M$ larger than some universal constant) that
$$
\|V(\cdot,s)\|_{L^{2}_{x}}^2+\int\limits_{t''_{0}}^{s}\int\limits_{\mathbb{R}^3} |\nabla V(x,t)|^2dxdt\leq C(\Mp)^4 (s-t''_{0})^{\frac{1}{2}}.
$$
Here $\Mp$ is defined by \eqref{e.defM'theo} for an appropriate universal constant $L^*\in(0,\infty)$ coming from the Gronwall estimate. 
In particular, using that $s\in [t''_{0}, \frac{t''_{0}}{{8\alpha^{201}}}]$ we have
\begin{equation}\label{L2almostscaleinvar}
\|V(\cdot,s)\|_{L^{2}_{x}}^2\leq C8^{\frac 12}\alpha^{101}(\Mp)^4 (-s)^{\frac{1}{2}}< \alpha^{106}(-s)^{\frac{1}{2}}.
\end{equation}
Here, we used the fact that $\alpha\geq \Mp$. 
Next assume for contradiction the under the assumptions of Lemma \ref{lem.backconctimeslice}, we have the converse of \eqref{e.concs_0}. Namely, there exists $s_0\in [t''_{0}, \frac{t''_{0}}{{8\alpha^{201}}}]$ such that
\begin{equation}\label{contradictionconc}
\int\limits_{B_{0}(4\alpha^{106}(-s_0)^{\frac{1}{2}})} |\omega(x,s_{0})|^2 dx\leq\frac{(M+1)^2}{(-s_{0})^{\frac{1}{2}}\alpha^{106}}.
\end{equation}
Define
\begin{equation}\label{rescaleparametersliceconc}
\lambda:= (-s_{0})^{\frac{1}{2}} \alpha^{106}
\end{equation}
and rescale to get $U^{\lambda}: \mathbb{R}^{3}\times (0,\alpha^{-212})\rightarrow \mathbb{R}^3$ and $P^{\lambda}: \mathbb{R}^{3}\times (0,\alpha^{-212})\rightarrow \mathbb{R}$. Here,
\begin{equation}\label{rescalesliceconc}
U^{\lambda}(y,t):= \lambda u(\lambda y, \lambda^2 t+s_{0})\,\,\,\,\textrm{and}\,\,\,\,\,P^{\lambda}(y,t):= \lambda^2 p(\lambda y, \lambda^2 t+s_{0}).
\end{equation}
Using \eqref{linestconcslice} and \eqref{L2almostscaleinvar}, we see that
\begin{equation}\label{rescaleidconcslice}
\|U^{\lambda}(y,0)\|_{L^{2}_{uloc}}\leq C_{Leb}M+1.
\end{equation}
Here, $C_{Leb}\in [1,\infty)$ is a universal constant from the embedding $L^{3}(\mathbb{R}^3)\subset L^{2}_{uloc}(\mathbb{R}^3)$.

Furthermore, defining $\Omega^{\lambda}=\nabla\times U^{\lambda}$, we see that \eqref{contradictionconc} implies that
\begin{equation}\label{rescalevortslice}
\int\limits_{B_{0}(4)} |\Omega^{\lambda}(y,0)|^2 dy\leq(M+1)^2.
\end{equation}
Similarly to Lemma \ref{lem.backconc}, we define $$S^\flat=S^\flat(M):=S_*\big(C_{Leb}M+1,32C_{Sob}(1+C_{ellip})(C_{Leb}M+1)\big)$$
where $S_*$ is the constant defined in Theorem \ref{theo.locshortime}, $C_{Sob}\in (0,\infty)$ is the best constant in the Sobolev embedding $H^1(B_0(2))\subset L^6(B_0(2))$ and $C_{ellip}\in (0,\infty)$ is the best constant in the estimate
$$
\|\nabla U(\cdot,0)\|_{L^2(B_0(2))}\leq C_{ellip}\left(\|\Omega(\cdot,0)\|_{L^2(B_0(4))}+\|U(\cdot,0)\|_{L^2(B_0(4))}\right)
$$
for weak solutions to
\begin{equation*}
-\Delta U(\cdot,0)=\nabla\times\Omega(\cdot,0)\quad\mbox{in}\quad B_0(4).
\end{equation*}
Next notice that from \eqref{e.defM'theo}, if we have for $M$ sufficiently large 
$$\alpha^{-212}<(\Mp)^{-212}<(1+M)^{{-101}}<S^\flat. $$
Using \eqref{rescaleidconcslice}-\eqref{rescalevortslice}, together with a similar reasoning to Lemma \ref{lem.backconc}, we can apply Theorem \ref{theo.locshortime} with $M\geq M_{4}$ being sufficiently large. Specifically, we apply Remark \ref{gradientboundongeneraltimeint} with $\beta= \alpha^{-212}$.
This gives
\begin{equation}\label{rescalegradboundslice}
\|\nabla U^{\lambda}\|_{L^\infty_tL^2_x(B_0(\frac16)\times(\frac{255}{256}\alpha^{-212},{\alpha^{-212}}))}\\
\leq C_{univ}M^{4}\alpha^{265}.
\end{equation}
This implies that 
\begin{equation}\label{gradboundslice}
\|\nabla u\|_{L^\infty_tL^2_x\big(B_0\big(\frac{(-s_0)^{\frac{1}{2}}\alpha^{106}}{6}\big)\times(\frac{s_{0}}{256},0)\big)}^2\\
\leq \frac{C_{univ}^2M^{8}\alpha^{424}}{(-s_0)^{\frac{1}{2}}}.
\end{equation}
Using that $s_{0}<t'_{0}<0$ and that $S^\flat(M)=O(1)M^{-100}$, we have for $M$ sufficiently large
 that 
$$B_{0}(4\sqrt{S^\flat}^{-1}(-t'_{0})^{\frac{1}{2}})\subset B_0(\tfrac16(-s_{0})^{\frac{1}{2}}(\Mp)^{106}). $$
So for $\frac{s_{0}}{256}<t'_{0}$, we see that \eqref{gradboundslice} implies that
$$\int\limits_{B_0(4\sqrt{S^\flat}^{-1}(-t_0')^\frac12)}|\omega(x,t_0')|^2 dx\leq \frac{C_{univ}^2M^{8}\alpha^{424}}{(-s_{0})^{\frac{1}{2}}}\leq \frac{M^{8}\alpha^{525}}{(-t''_{0})^{\frac{1}{2}}}.$$
Here, we used $s_{0}\in [t''_{0}, \frac{t''_{0}}{8\alpha^{201}}].$ Thus,
$$\int\limits_{B_0(4\sqrt{S^\flat}^{-1}(-t_0')^\frac12)}|\omega(x,t_0')|^2 dx\leq  {M^{2}(-t'_{0})^{-\frac{1}{2}}\sqrt{S^\flat}}\times \frac{\alpha^{525} M^6}{\sqrt{S^\flat}} \Big(\frac{-t'_{0}}{-t''_{0}}\Big)^{\frac{1}{2}}. $$
Now, $$\frac{{\alpha^{525}} M^6}{\sqrt{S^\flat}} \Big(\frac{-t'_{0}}{-t''_{0}}\Big)^{\frac{1}{2}}\leq C_{univ} M^{56}\alpha^{525}\Big(\frac{-t'_{0}}{-t''_{0}}\Big)^{\frac{1}{2}}.$$
Therefore, if 
\begin{equation}\label{sepconditionslice}
C_{univ} M^{56}\alpha^{525}\Big(\frac{-t'_{0}}{-t''_{0}}\Big)^{\frac{1}{2}}<1
\end{equation}
we contradict \eqref{e.conct_0'timeslice}. Thus, if \eqref{sepconditionslice} holds we must have that 
$$\int\limits_{B_{0}(4\alpha^{106}(-s_0)^{\frac{1}{2}})} |\omega(x,s_{0})|^2 dx>\frac{(M+1)^2}{(-s_{0})^{\frac{1}{2}}\alpha^{106}} $$ for all $s_{0}\in [t''_{0}, \frac{t''_{0}}{8\alpha^{201}}]$ as desired. Note that \eqref{e.wellseptimeslice} implies \eqref{sepconditionslice}.
\end{proof}

\subsection{Proof of the main quantitative estimate in the Type I case} \label{subsec.quantest}

This part is devoted to the proof of Proposition \ref{prop.main}. Following Tao \cite{Tao19}, the idea of the proof is to transfer the concentration of the enstrophy at times $t_0''$ far away in the past to large-scale lower bounds for the enstrophy at time $t_0$. This is done in Step 1-3 below. The last step, Step 4 below, consists in transferring the lower bound on the enstrophy at time $t_0$ to a lower bound for the $L^3$ norm at time $t_0$ and summing appropriate scales. In Step 5 we sum scales under the additional assumption \eqref{L3lowerboundassumption}.

Without loss of generality, we now take $t_0=0$. We also assume that $T=1$. The general statement is obtained by scaling. Let  
$M\in[M_3,\infty)$ where $M_3$ is a constant in Lemma \ref{lem.backconc}. 
In the course of the proof we will need to take $M$ larger, always larger than universal constants. Let $u:\, \R^3\times [-1,0]\rightarrow\R^3$ be a `smooth solution with sufficient decay'\footnote{See footnote \ref{footdefsol}.} of the Navier-Stokes equations \eqref{e.nse} in $I=[-1,0]$ satisfying the Type I bound \eqref{e.t1A}. 
Assume that there exists $t_0'\in [-1,0)$ such that $t_0'$ is not too close to $-1$ in the sense
\begin{equation*}
0< -t_0'<C^\sharp M^{-548}
\end{equation*}
and such that the vorticity concentrates at time $t_0'$ in the following sense
\begin{equation}\label{e.conct_0'bis}
\int\limits_{B_0(4\sqrt{S^\sharp }^{-1}(-t_0')^\frac12)}|\omega(x,t_0')|^2\, dx> M^2(-t_0')^{-\frac12}\sqrt{S^\sharp },
\end{equation}
where we recall that $S^\sharp =O(1)M^{-100}$. 
Lemma \ref{lem.backconc} then implies that
\begin{equation}\label{e.conct_0''bis}
\int\limits_{B_0(4\sqrt{S^\sharp }^{-1}(-t'')^\frac12)}|\omega(x,t'')|^2\, dx> M^2(-t'')^{-\frac12}\sqrt{S^\sharp }.
\end{equation}
at any well-separated backward time $t''\in [-1,t_0']$ such that\footnote{\label{foot.ae}Notice that the whole argument of Section \ref{subsec.quantest} goes through assuming that \eqref{e.conct_0''bis} holds for almost any 
$t''\in[-1,\frac1{C^\sharp }M^{548}t_0')$.}
\begin{equation}\label{e.wellsepbis}
\frac1{C^\sharp }M^{548}t_0'>t''.
\end{equation}
The rest of the proof relies on the Carleman inequalities of Proposition \ref{prop.firstcarl} and Proposition \ref{prop.sndcarl}. These are the tools used to transfer the concentration  information \eqref{e.conct_0''bis} from the time $t''$ 
to time $0$ and from the small scales $B_0(4\sqrt{S^\sharp }^{-1}(-t'')^\frac12)$ to large scales.

\noindent{\bf Step 1: quantitative unique continuation.} The purpose of this step is to prove the following estimate:
\begin{align}\label{e.claimstep1}
T_1^{\frac12}e^{-\frac{O(1)M^{149}R^2}{T_1}}
\lesssim\ & \int\limits_{-T_1}^{-\frac{T_1}2}\int\limits_{B_0(2R)\setminus B_0(R/2)}|\omega(x,t)|^2\, dxdt,
\end{align}
for all $T_1$ and $R$ such that 
\begin{equation}\label{e.wellsepT1}
\frac2{C^\sharp}M^{548}(-t_0')<T_1\leq \frac{1}{2}\quad\mbox{and}\quad R\geq M^{100}\Big(\frac{T_1}{2}\Big)^\frac12.
\end{equation} 
Let $t_0''$ be such that \eqref{e.wellsepbis} is satisfied  with $t''=\frac{t_0''}2$. 
Let $T_1:=-t_0''$ and $I_1:=(t_0'',t_0''+\frac{T_1}2)=(-T_1,-\frac{T_1}{2})\subset[-\tfrac12,0]\subset[-1,0]$. Thus, we can apply Lemma \ref{epochTypeI} and Remark \ref{rem.estcarlepo} with $t_{0}=0$ and $T=1$. The bound \eqref{carlemantypeIepochbound} in Remark \ref{rem.estcarlepo} implies that there exists an epoch of regularity $I_1''=[t_1''-T_1'',t_1'']\subset I_1$ such that 
\begin{equation}\label{e.sizeI1'}
T_1''=|I_1''|=\frac{M^{-48}}{4C_{univ}^3}|I_1|=\frac{M^{-48}}{8C_{univ}^3}T_1
\end{equation}
and for $j=0, 1, 2$,
\begin{equation}\label{e.epochI1'}
\|\nabla^j u\|_{L^{\infty}_{t}L^{\infty}_{x}(\mathbb{R}^3\times I_1'')}\leq \frac{1}{2^{j+1}} |I_1''|^{\frac{-(j+1)}{2}}=\frac{1}{2^{j+1}}(T_1'')^{\frac{-(j+1)}{2}}.
\end{equation}
Let $T_1''':=\frac34T_1''$ and $s''\in [t_1''-\frac{T_1''}4,t_1'']$. Let $x_1\in\R^3$ be such that $|x_1|\geq M^{100}(\frac{T_1}2)^\frac12$ and let 
$r_1:=M^{50}|x_1|\geq M^{150}(\frac{T_1}2)^\frac12$. 
Notice that for $M$ large enough 
\begin{equation}\label{e.condr1}
r_1:=M^{50}|x_1|\geq M^{150}\Big(\frac{T_1}2\Big)^\frac12\geq M^{99}\cdot 4\sqrt{S^\sharp}^{-1}(-t_0'')^\frac12
\end{equation}
and 
\begin{equation*}
r_1^2\geq 4000T_1'''.
\end{equation*}
We apply the second Carleman inequality, Proposition \ref{prop.sndcarl} (quantitative unique continuation), on the cylinder $\mathcal C_1=\{(x,t)\in\mathbb R^3\times\mathbb R\, :\ t\in [0,T_1'''],\ |x|\leq r_1\}$ to the function $w:\, \R^3\times [0,T_1''']\rightarrow\R^3$, defined by for all $(x,t)\in \R^3\times [0,T_1''']$, 
\begin{equation*}
w(x,t):=\omega(x_1+x,s''-t).
\end{equation*}
Notice that  the quantitative regularity \eqref{e.epochI1'} and the vorticity equation \eqref{vort} implies that on $\mathcal C_1$
\begin{equation*}
|(\partial_t+\Delta)w|\leq \frac3{16}{T_1'''}^{-1}|w|+\frac{\sqrt{3}}4{T_1'''}^{-\frac12}|\nabla w|,
\end{equation*}
so that \eqref{e.diffineqC} is satisfied with $S=S_1:=T_1'''$ and $C_{Carl}=\frac{16}3$. Let 
\begin{equation*}
\hat s_1=\frac{T_1'''}{20000},\qquad \check s_1=M^{-150}T_1'''.
\end{equation*}
For $M$ sufficiently large we have $0<\check s_1\leq \hat s_1\leq\frac{T_1'''}{10000}$. Hence by \eqref{e.conclcarltwo} we have 
\begin{equation}\label{e.estXYZ1}
Z_1\lesssim e^{-\frac{r_1^2}{500\hat s_1}}X_1+(\hat s_1)^\frac32\Big(\frac{e\hat s_1}{\check s_1}\Big)^{\frac{O(1)r_1^2}{\hat s_1}}Y_1,
\end{equation}
where
\begin{align*}
&X_1:=\int\limits_{s''-T_1'''}^{s''}\int\limits_{B_{x_1}(M^{50}|x_1|)}((T_1''')^{-1}|\omega|^2+|\nabla\omega|^2)\, dxds,\\
&Y_1:=\int\limits_{B_{x_1}(M^{50}|x_1|)}|\omega(x,s'')|^2(\check s_1)^{-\frac32}e^{-\frac{|x-x_1|^2}{4\check s_1}}\, dx,\\
&Z_1:=\int\limits_{s''-\frac{T_1'''}{10000}}^{s''-\frac{T_1'''}{20000}}\int\limits_{B_{x_1}(\frac{M^{50}|x_1|}2)}((T_1''')^{-1}|\omega|^2+|\nabla\omega|^2)e^{-\frac{|x-x_1|^2}{4(s''-s)}}\, dxds.
\end{align*}
We first use the concentration \eqref{e.conct_0''bis} for times $s\in [s''-\frac{T_1'''}{10000},s''-\frac{T_1'''}{20000}]$ to bound $Z_1$ from below. By \eqref{e.condr1}, we have
\begin{equation*}
B_0(4\sqrt{S^\sharp }^{-1}(-s)^\frac12)\subset B_{x_{1}}(2|x_{1}|)\subset B_{x_1}\Big(\frac{M^{50}|x_1|}{2}\Big)
\end{equation*}
for all $s\in [s''-\frac{T_1'''}{10000},s''-\frac{T_1'''}{20000}]$ and for $M$ sufficiently large.
We have 
\begin{align*}
Z_1\gtrsim\ &\int\limits_{s''-\frac{T_1'''}{10000}}^{s''-\frac{T_1'''}{20000}}\int\limits_{B_0(4\sqrt{S^\sharp }^{-1}(-s)^\frac12)}(T_1''')^{-1}|\omega(x,s)|^2\, dxds\, e^{-\frac{O(1)|x_1|^2}{T_1'''}}\\
\gtrsim\ &\int\limits_{s''-\frac{T_1'''}{10000}}^{s''-\frac{T_1'''}{20000}}M^{-48}(-s)^{-\frac12}\, ds(T_1'')^{-1}e^{-\frac{O(1)|x_1|^2}{T_1''}}\\
\gtrsim\ & M^{-48}\frac{T_1'''}{(-s''+\frac{T_1'''}{10000})^{\frac{1}{2}}}(T_1'')^{-1}e^{-\frac{O(1)|x_1|^2}{T_1''}}\\
\gtrsim\ &M^{-48}(T_1)^{-\frac12}e^{-\frac{O(1)|x_1|^2}{T_1''}}\\
\gtrsim\ &M^{-48}(M^{48}T_1'')^{-\frac12}e^{-\frac{O(1)|x_1|^2}{T_1''}}\\=\ & M^{-72}(T_1'')^{-\frac{1}{2}}e^{-\frac{O(1)|x_1|^2}{T_1''}}.
\end{align*}
Second, we bound from above $X_1$. We rely on the quantitative regularity \eqref{e.epochI1'} to obtain
\begin{align*}
X_1\lesssim  (T_1'')^{-2}M^{150}|x_1|^3.
\end{align*}
Hence,
\begin{align*}
e^{-\frac{r_1^2}{500\hat s_1}}X_1\lesssim\ & (T_1'')^{-2}M^{150}|x_1|^3e^{-\frac{O(1)M^{100}|x_1|^2}{T_1''}}\\
\lesssim\ &(T_1'')^{-\frac12}e^{-\frac{O(1)M^{100}|x_1|^2}{T_1''}}.
\end{align*}
Third, for $Y_1$ we decompose and estimate as follows
\begin{align*}
Y_1:=\ &\int\limits_{B_{x_1}(\frac{|x_1|}2)}|\omega(x,s'')|^2(\check s_1)^{-\frac32}e^{-\frac{|x-x_1|^2}{4\check s_1}}\, dx\\
&+\int\limits_{B_{x_1}(M^{50}|x_1|)\setminus B_{x_1}(\frac{|x_1|}2)}|\omega(x,s'')|^2(\check s_1)^{-\frac32}e^{-\frac{|x-x_1|^2}{4\check s_1}}\, dx\\
\lesssim\ &M^{225}(T_1'')^{-\frac32}\Bigg(\int\limits_{B_{x_1}(\frac{|x_1|}2)}|\omega(x,s'')|^2\, dx\\
&+\int\limits_{B_{x_1}(M^{50}|x_1|)\setminus B_{x_1}(\frac{|x_1|}2)}|\omega(x,s'')|^2
e^{-\frac{O(1)M^{150}|x_1|^2}{T_1''}}\, dx\Bigg)\\
\lesssim\ &M^{225}(T_1'')^{-\frac32}\Bigg(\int\limits_{B_{x_1}(\frac{|x_1|}2)}|\omega(x,s'')|^2\, dx+M^{150}|x_1|^3(T_1'')^{-2}e^{-\frac{O(1)M^{150}|x_1|^2}{T_1''}}\Bigg)\\
\lesssim\ &M^{225}(T_1'')^{-\frac32}\Bigg(\int\limits_{B_{x_1}(\frac{|x_1|}2)}|\omega(x,s'')|^2\, dx+(T_1'')^{-\frac12}e^{-\frac{O(1)M^{150}|x_1|^2}{T_1''}}\Bigg),
\end{align*}
where we used the quantitative regularity \eqref{e.epochI1'}. Hence,
\begin{align*}
(\hat s_1)^\frac32\Big(\frac{e\hat s_1}{\check s_1}\Big)^{\frac{O(1)r_1^2}{\hat s_1}}Y_1\lesssim\ &(T_1'')^\frac32e^{\frac{O(1)M^{100}|x_1|^2}{T_1''}\log(\frac{eM^{150}}{20000})}Y_1\\
\lesssim\ &M^{225}e^{\frac{O(1)M^{101}|x_1|^2}{T_1''}}\int\limits_{B_{x_1}(\frac{|x_1|}2)}|\omega(x,s'')|^2\, dx\\
&+M^{225}(T_1'')^{-\frac12}e^{-\frac{O(1)M^{150}|x_1|^2}{T_1''}}.
\end{align*}
Gathering these bounds and combining with \eqref{e.estXYZ1} yields
\begin{multline*}
M^{-72}(T_1'')^{-\frac12}e^{-\frac{O(1)|x_1|^2}{T_1''}}\lesssim (T_1'')^{-\frac12}e^{-\frac{O(1)M^{100}|x_1|^2}{T_1''}}\\
+M^{225}e^{\frac{O(1)M^{101}|x_1|^2}{T_1''}}\int\limits_{B_{x_1}(\frac{|x_1|}2)}|\omega(x,s'')|^2\, dx+M^{225}(T_1'')^{-\frac12}e^{-\frac{O(1)M^{150}|x_1|^2}{T_1''}}.
\end{multline*}
Using \eqref{e.sizeI1'} and $|x_1|\geq M^{100}(\frac{T_1}{2})^{\frac{1}{2}}$, we see that for $M$ sufficiently large
\begin{align*}
M^{-297}(T_1'')^{-\frac12}e^{-\frac{O(1)M^{101}|x_1|^2}{T_1''}}\lesssim\ & \int\limits_{B_{x_1}(\frac{|x_1|}2)}|\omega(x,s'')|^2\, dx.
\end{align*}
Hence, 
for all $s''\in I_1''=[t_1''-\frac{T_1''}{4},t_1'']$, for all $|x_1|\geq M^{100}(\frac{T_1}{2})^\frac12$,
\begin{equation*}
\int\limits_{B_{x_1}(\frac{|x_1|}{2})}|\omega(x,s'')|^2\, dx\gtrsim M^{-297}(T_1'')^{-\frac12}e^{-\frac{O(1)M^{101}|x_1|^2}{T_1''}}.
\end{equation*}
Let $R\geq M^{100}(\frac{T_1}2)^\frac12$ and $x_{1}\in\mathbb{R}^3$ be such that $|x_{1}|=R$. 
Integrating in time $[t_1''-\frac{T_1''}4,t_1'']$ yields the estimate
\begin{align*}
M^{-321}e^{O(1)M^{349}}T_1^{\frac12}e^{-\frac{2O(1)M^{149}R^2}{T_1}}\lesssim\ &M^{-321}T_1^{\frac12}e^{-\frac{O(1)M^{149}R^2}{T_1}}\\
\lesssim\ & \int\limits_{t_1''-\frac{T_1''}4}^{t_1''}\int\limits_{B_0(2R)\setminus B_0(R/2)}|\omega(x,t)|^2\, dxdt
\end{align*}
which yields the claim \eqref{e.claimstep1} of Step 1.

\noindent{\bf Step 2: quantitative backward uniqueness.} The goal of this step and Step 3 below is to prove the following claim:
\begin{align}\label{e.claimstep2}
\begin{split}
T_2^{-\frac12}\exp\big(-\exp({M^{1021}})\big)\lesssim\ &\int\limits_{B_0\big(\tfrac{3}{4}C(100)M^{1000}R_{2}'\big)\setminus B_0(2R_{2}')}|\omega(x,0)|^2\, dx,
\end{split}
\end{align}
for all $\frac{8}{C^\sharp }M^{749}(-t_0')< T_2\leq 1$ and $M$ sufficiently large. Here, $R_{2}$, $R_{2}'$ and $C(100)$ are as in \eqref{e.restrR2}-\eqref{e.defann2}. This is the key estimate for Step 4 below and the proof of Proposition \ref{prop.main}.

We apply here the results of Section \ref{sec.quantannulus} for the quantitative existence of an annulus of regularity. Although the parameter $\mu$ in Section \ref{sec.quantannulus} is any positive real number, 
here we need to take $\mu$ sufficiently large in order to have a large enough annulus of quantitative regularity, and hence a large $r_+$ below in the application of the first Carleman inequality Proposition \ref{prop.firstcarl}. 
To fix the ideas, we take $\mu=100$.\footnote{More specifically, we see that $\mu$ is chosen so that $10\mu>350$ in order to obtain \eqref{e.step2lowerbd1} from \eqref{e.conclcarlonebis} and \eqref{e.lowerZ2}.} Let $T_1$ and $T_2$ such that
\begin{equation}\label{e.sepT2}
\frac{8}{C^\sharp }M^{548+201}(-t_0')\leq T_2\leq 1\quad\mbox{and}\quad T_1:=\frac{T_2}{4M^{201}}.\footnote{The reason  for this is to ensure we can apply Step 1 to get a lower bound \eqref{e.lowerZ2} for $Z_{2}$. }
\end{equation}
Let 
\begin{equation}\label{e.restrR2}
R_2:=K^\sharp(T_2)^\frac12,
\end{equation}
for a universal constant $K^\sharp\geq 1$ to be chosen sufficiently large below. In particular it is chosen in Step 3 such that \eqref{e.choiceKsharp} holds, which makes it possible to absorb the upper bound \eqref{e.controlX3} of $X_3$ in the left hand side of \eqref{e.conclcarltwobis}. 
By Corollary \ref{cor.12}, for $M\geq M_{1}(100)$ there exists a scale
\begin{equation}\label{e.bdR2prime}
2R_2\leq R_2'\leq 2R_2\exp({C(100) M^{1020}})
\end{equation}
and a good cylindrical annulus 
\begin{equation}\label{e.defann2}
\mathcal A_2:=\{R_2'<|x|<c(100){M^{1000}}R_2'\}\times\Big(-\frac{T_2}{32},0\Big)
\end{equation} 
such that for $j=0,1$, 
\begin{align}\label{e.linftyderivstep2}
\begin{split}
\|\nabla^j u\|_{L^{\infty}(\mathcal A_2)}
\leq\ & 2^{\frac{j+1}{2}}\bar{C}_{j}C(100) M^{-{300}
}T_2^{-\frac{j+1}{2}},\\
\|\nabla \omega\|_{L^{\infty}(\mathcal A_2)}
\leq\ & 2^{\frac{3}{2}}\bar{C}_{2}C(100) M^{-{300}
}T_2^{-\frac32}.
\end{split}
\end{align}
We apply now the quantitative backward uniqueness, Proposition \ref{prop.firstcarl} to the function $w:\, \R^3\times [0,\frac{T_2}{M^{201}}]\rightarrow\R^3$ defined by for all $(x,t)\in\R^3\times [0,\frac{T_2}{M^{201}}]$,
\begin{equation*}
w(x,t)=\omega(x,-t).
\end{equation*}
An important remark is that although we have a large cylindrical annulus of quantitative  regularity $\mathcal A_2$, we apply the Carleman estimate on a much smaller annulus, namely
\begin{equation}\label{e.defann2bis}
\widetilde{\mathcal A}_2:=\Big\{4R_2'<|x|<\frac{c(100)}{4}{M^{1000}}R_2'\Big\}\times\Big(-\frac{T_2}{M^{201}},0\Big).
\end{equation}

Choosing $M$ sufficiently large such that $2\bar{C}_{j}C(100) M^{-{300}}\leq 1$ and $2^\frac32\bar{C}_{2}C(100) M^{-{300}}\leq 1$, we see that the bounds \eqref{e.linftyderivstep2} imply that the differential inequality \eqref{e.diffineq} is satisfied with $S=S_2:=\frac{T_2}{M^{201}}$ and  
$C_{Carl}=M^{201}$. Take 
\begin{equation*}
r_-=4R_2',\qquad r_+=\tfrac14c(100){M^{1000}}R_2'.
\end{equation*}
Then, 
\begin{equation*}
B_0(160R_2')\setminus B_0(40R_2')=B_0(40r_-)\setminus B_0(10r_-)\subset\left\{40R_2'<|x|<\tfrac{c(100)}{8}{M^{1000}}R_2'\right\}
\end{equation*}
on condition that $M$ is sufficiently large: one needs $c(100)M^{1000}>1280$. 
Note also that
\begin{equation*}
r_{-}^2= 16(R_{2}')^2\geq 64R_{2}^2=64(K^{\sharp})^2T_{2}>64T_{2}>4SC_{Carl}.
\end{equation*}
By \eqref{e.conclcarlone}, we get
\begin{equation}\label{e.conclcarlonebis}
Z_2\lesssim e^{-\frac{O(1)M^{1000}(R_2')^2}{T_2}}\big(X_2+e^{\frac{O(1)M^{2000}(R_2')^2}{T_2}}Y_2\big),
\end{equation}
where 
\begin{align*}
X_2:=\ &\int\limits_{-\frac{T_{2}}{M^{201}}}^{0}\int\limits_{r_{-}\leq |x|\leq r_{+}}e^{\frac{4|x|^2}{T_2}}(M^{201}T_2^{-1}|\omega|^2+|\nabla \omega|^2)\, dxdt,\\ 
Y_2:=\ &\int\limits_{r_-\leq |x|\leq r_+}|\omega(x,0)|^2\, dx,\\
Z_2:=\ &\int\limits_{-\frac{T_2}{4 M^{201}}}^{0}\int\limits_{10r_-\leq|x|\leq \frac{r_+}2}(M^{201}T_2^{-1}|\omega|^2+|\nabla \omega|^2)\, dxdt.
\end{align*}
Thanks to the separation condition \eqref{e.sepT2} and to the fact that for $M$ large enough \eqref{e.restrR2} implies
\begin{equation*}
20r_{-}\geq 10R_2'\geq 20R_2=20K^\sharp T_2^\frac12\geq M^{100}\Big(\frac{T_2}{8M^{201}}\Big)^\frac12=M^{100}\Big(\frac{T_1}{2}\Big)^\frac12,
\end{equation*} 
we can apply the concentration result of Step 1, taking there $T_1=\frac{T_2}{4 M^{201}}=\frac {S_2}4$ and $R=20r_-$. By \eqref{e.claimstep1} we have that
\begin{equation}\label{e.lowerZ2}
Z_2\gtrsim M^{201}\left(\frac{T_2}{4M^{201}}\right)^\frac12e^{-\frac{O(1)M^{350}(R_2')^2}{T_2}}T_2^{-1}\gtrsim T_2^{-\frac12}e^{-\frac{O(1)M^{350}(R_2')^2}{T_2}}.
\end{equation}
Therefore, one of the following two lower bounds holds
\begin{align}
T_2^{-\frac12}\exp\Big(\frac{O(1)M^{1000}(R_2')^2}{T_2}\Big)\lesssim X_2,\label{e.step2lowerbd1}\\
T_2^{-\frac12}\exp(-\exp(M^{1021}))
\lesssim e^{-\frac{O(1)M^{2000}(R_2')^2}{T_2}}T_{2}^{-\frac12}\lesssim Y_2,\label{e.step2lowerbd2}
\end{align}
where we used the upper bound \eqref{e.bdR2prime} for \eqref{e.step2lowerbd2}. 
The bound \eqref{e.step2lowerbd2} can be used directly in Step 4 below. On the contrary, if \eqref{e.step2lowerbd1} holds more work needs to be done to transfer the lower bound on the enstrophy at time $0$. This is the objective of Step 3 below.

\noindent{\bf Step 3: a final application of quantitative unique continuation.} Assume that the bound \eqref{e.step2lowerbd1} holds. We will apply the pigeonhole principle three times successively in order to end up in a situation where we can rely on the quantitative unique continuation to get a lower bound at time $0$. We first remark that this with the definition \eqref{e.defann2bis} of the annulus $\widetilde{\mathcal A}_2$ implies the following lower bound
\begin{multline*}
T_2^{-\frac12}\exp\Big(\frac{O(1)M^{1000}(R_2')^2}{T_2}\Big)\\
\lesssim \int\limits_{-\frac{T_2}{M^{201}}}^{0}\int\limits_{4R_2'\leq|x|\leq \frac14c(100)M^{1000}R_2'}e^{\frac{4|x|^2}{T_2}}(M^{201}T_2^{-1}|\omega|^2+|\nabla \omega|^2)\, dxdt.
\end{multline*}
By the pigeonhole principle, there exists 
\begin{equation}\label{e.condR3}
8R_2'\leq R_3\leq \tfrac12c(100){M^{1000}}R_2'
\end{equation}
such that 
\begin{equation*}
T_2^{-\frac12}\exp\Big(-\frac{4R_3^2}{T_2}\Big)
\lesssim \int\limits_{-\frac{T_2}{M^{201}}}^0\int\limits_{B_0(R_3)\setminus B_0(\frac{R_3}2)}(T_2^{-1}|\omega|^2+|\nabla \omega|^2)\, dxdt.
\end{equation*}
Using the bounds \eqref{e.linftyderivstep2}, we have that 
\begin{equation*}
T_2^{-\frac12}\exp\Big(-\frac{4R_3^2}{T_2}\Big)\lesssim \int\limits_{-\frac{T_2}{M^{201}}}^{-\exp(-\frac{8R_3^2}{T_2})T_2}\int\limits_{B_0(R_3)\setminus B_0(\frac{R_3}2)}(T_2^{-1}|\omega|^2+|\nabla \omega|^2)\, dxdt.
\end{equation*}
By the pigeonhole principle, there exists 
\begin{equation}\label{e.condt3}
\frac{1}{2}\exp\Big(-\frac{8R_3^2}{T_2}\Big)T_2\leq -t_3\leq\frac{T_2}{M^{201}}
\end{equation}
such that
\begin{equation*}
T_2^{-\frac12}\exp\Big(-\frac{5R_3^2}{T_2}\Big)\lesssim \int\limits_{2t_3}^{t_3}\int\limits_{B_0(R_3)\setminus B_0(\frac{R_3}2)}(T_2^{-1}|\omega|^2+|\nabla \omega|^2)\, dxdt.
\end{equation*}
We finally cover the annulus $B_0(R_3)\setminus B_0(\frac{R_3}2)$ with
\begin{equation*}
O(1)\frac{R_3^3}{(-t_3)^\frac32}\lesssim \frac{R_3^3}{T_2^\frac32}\exp\Big(\frac{12R_3^2}{T_2}\Big)\lesssim \exp\Big(\frac{13R_3^2}{T_2}\Big)
\end{equation*}
balls of radius $(-t_3)^\frac12$, and apply the pigeonhole principle a third time to find that there exists $x_3\in B_0(R_3)\setminus B_0(\frac{R_3}2)$ such that 
\begin{equation}\label{e.conct3}
T_2^{-\frac12}\exp\Big(-\frac{18R_3^2}{T_2}\Big)\lesssim \int\limits_{2t_3}^{t_3}\int\limits_{B_{x_3}((-t_3)^\frac12)}(T_2^{-1}|\omega|^2+|\nabla \omega|^2)\, dxdt.
\end{equation}
We apply now the second Carleman inequality, Proposition \ref{prop.sndcarl}, to the function $w:\, \R^3\times [0,-20000t_{3}]\rightarrow\R^3$ defined by for all $(x,t)\in\R^3\times [0,-20000t_{3}]$,
\begin{equation*}
w(x,t)=\omega(x+x_3,-t).
\end{equation*}
Let $S_3:=-20000t_{3}$. We take\footnote{\label{foot.tao}We follow here an idea of Tao which enables to remove one exponential from the final estimate. This idea appears on his blog \url{https://terrytao.wordpress.com/2019/08/15/quantitative-bounds-for-critically-bounded-solutions-to-the-navier-stokes-equations/} in a comment dated December 28, 2019. See also footnote \ref{foot.taobis} and \eqref{e.trivR3r3}.}
\begin{equation}\label{e.choicer_3}
r_3:=1000R_{3}\Big(-\frac{t_{3}}{T_{2}}\Big)^{\frac{1}{2}},\quad \hat s_3=\check s_3=-t_3.
\end{equation}
Notice that due to \eqref{e.restrR2}-\eqref{e.bdR2prime} and \eqref{e.condR3}, we have that
\begin{align}\label{e.condx3}
r_3^2=\ &{10^6 R_{3}^2\Big(-\frac{t_{3}}{T_{2}}\Big)}\geq (2.56\times 10^8)(K^\sharp)^2(-t_{3})
\geq 4000S_{3}=(8\times 10^7)(-t_{3}),\\
\frac{r_{3}}{2}\geq\ &8000R_{2}\Big(-\frac{t_{3}}{T_{2}}\Big)^{\frac{1}{2}}=8000K^\sharp(-t_{3})^{\frac{1}{2}}> (-t_3)^\frac12,
\end{align}
so that \eqref{e.lowerr} is satisfied. Furthermore, from \eqref{e.condt3} we have
$$\frac{|x_{3}|}{2}\geq \frac{R_{3}}{4}\geq 1000R_{3}\Big(\frac{1}{M^{201}}\Big)^{\frac{1}{2}}\geq r_{3}. $$
Thus
\begin{multline}\label{domaininclus}
B_{x_3}((-t_3)^\frac12)\subset B_{x_3}(\tfrac{r_3}2)\subset B_{x_{3}}(r_{3})\subset B_{x_3}\Big(\frac{|x_{3}|}{2}\Big)\\
\subset\{\tfrac{R_3}{4}<|y|<\tfrac{3}{2}R_3\}\subset\{2R_2'<|y|<\tfrac{3}{4}c(100)M^{1000}R_2'\}. 
\end{multline}
Moreover, 
\begin{equation*}
0\leq\hat s_3=\check s_3=-t_3\leq -2t_{3} =\frac{S_3}{10^4}.
\end{equation*} 
By \eqref{e.condt3}, we see that for $M$ large enough $S_{3}\leq \frac{T_{2}}{32}$, hence the bounds \eqref{e.linftyderivstep2} imply that the differential inequality \eqref{e.diffineq} is satisfied on $B_{0}(r)\times [0,S]$ with $S=S_3$, $r=r_{3}$ and $C_{Carl}=1$. 
Therefore, by \eqref{e.conclcarltwo} we have
\begin{equation}\label{e.conclcarltwobis}
Z_3\leq C_{univ} e^{\frac{r_3^2}{500t_3}}X_3+ C_{univ}(-t_3)^\frac32e^{-\frac{O(1)r_3^2}{t_3}}Y_3,
\end{equation}
where 
\begin{align*}
X_3:=\ &\int\limits_{-S_{3}}^0\int\limits_{B_{x_3}(r_3)}(S_{3}^{-1}|\omega|^2+|\nabla\omega|^2)\, dxdt,\qquad Y_3:=\int\limits_{B_{x_3}(r_3)}|\omega(x,0)|^2(-t_3)^{-\frac32}e^{\frac{|x-x_3|^2}{4t_3}}\, dx,\\
Z_3:=\ &\int\limits_{2t_3}^{t_3}\int\limits_{B_{x_3}(\frac{r_3}2)}(S_{3}^{-1}|\omega|^2+|\nabla \omega|^2)e^{\frac{|x-x_3|^2}{4t}}\, dxdt.
\end{align*}
Using \eqref{e.conct3} and $T_{2}^{-1}\leq S_{3}^{-1}$ we have
\begin{equation}\label{Z3lowerbound}
  T_2^{-\frac12}\exp\Big(-\frac{18R_3^2}{T_2}\Big)\lesssim \int\limits_{2t_3}^{t_3}\int\limits_{B_{x_3}((-t_{3})^{\frac{1}{2}})}(T_2^{-1}|\omega|^2+|\nabla \omega|^2)e^{\frac{|x-x_3|^2}{4t}}\, dxdt\leq Z_{3}
\end{equation} 
Using the bounds \eqref{e.linftyderivstep2} along with \eqref{e.condt3}, we find that
\begin{multline}\label{e.controlX3}
C_{univ}e^{\frac{r_3^2}{500t_3}}X_3\lesssim S_{3}^{-2}r_{3}^3e^{\frac{r_3^2}{500t_3}}\lesssim {(-t_{3})}^{-\frac12}e^{\frac{r_3^2}{1000t_3}}\lesssim T_2^{-\frac12}e^{\frac{4R_3^2}{T_2}}e^{\frac{r_3^2}{1000t_3}}\\
\lesssim T_2^{-\frac12}e^{-\frac{996 R_3^2}{T_2}}\lesssim T_2^{-\frac12}e^{-\frac{18 R_3^2}{T_2}}e^{-\frac{978 R_3^2}{T_2}}\leq C'_{univ} T_2^{-\frac12}e^{-\frac{18 R_3^2}{T_2}}e^{-978\cdot 256(K^\sharp)^2}.
\end{multline}
We choose $K^\sharp$ sufficiently large so that 
\begin{equation}\label{e.choiceKsharp}
C'_{univ}e^{-978\cdot 256(K^\sharp)^2}\leq \frac12,
\end{equation}
where $C'_{univ}\in(0,\infty)$ is the universal constant appearing in the last inequality of \eqref{e.controlX3}. 
Therefore, the term in the right hand side of \eqref{e.controlX3} is negligible 
with respect to the lower bound \eqref{Z3lowerbound} of $Z_3$. Combining now \eqref{e.conclcarltwobis} with the lower bound \eqref{Z3lowerbound}, we obtain\footnote{\label{foot.taobis}Here one notices a key advantage of taking $r_3$ to depend linearly on $-t_3$ as in \eqref{e.choicer_3}. Otherwise the trivial bound
\begin{equation}\label{e.trivR3r3}
\frac{r_3^2}{-t_3}\leq \frac{R_3^2}{4(-t_3)}\lesssim \exp\Big(\frac{8R_3^3}{T_2}\Big)\frac{R_3^2}{T_2},
\end{equation}
where we used the lower bound \eqref{e.condt3} on $-t_3$, would lead one more exponential in the final estimate. Taking $r_3$ as in \eqref{e.choicer_3} is Tao's idea; see footnote \ref{foot.tao}.}
\begin{align*}
T_2^{-\frac12}\exp\Big(-\frac{18R_3^2}{T_2}\Big)\lesssim\ 
 & \exp\Big(-{\frac{O(1)r_3^2}{t_3}}\Big)\int\limits_{B_{x_3}(r_3)}|\omega(x,0)|^2\, dx\\
\lesssim\ & \exp\Big(O(1){\frac{R_3^2}{T_2}}\Big)\int\limits_{B_{x_3}(r_3)}|\omega(x,0)|^2\, dx.
\end{align*}
Hence,
\begin{equation*}
T_2^{-\frac12}\exp\Big(-O(1)\frac{R_3^2}{T_2}\Big)\lesssim \int\limits_{B_{x_3}(r_3)}|\omega(x,0)|^2\, dx.
\end{equation*}
Using \eqref{e.restrR2}, \eqref{domaininclus} and the upper bound 
\begin{equation*}
R_3\leq\tfrac{1}2c(100)M^{1000}R_2'\leq c(100)M^{1000}\exp(C(100)M^{1020})R_2,
\end{equation*}
it follows that 
\begin{align}\label{e.lastbddstep3}
\begin{split}
T_2^{-\frac12}\exp(-\exp(M^{1021}))
\lesssim\ &\int\limits_{B_0\big(\tfrac{3}{4}C(100)M^{1000}R_{2}'\big)\setminus B_0(2R_{2}')}|\omega(x,0)|^2\, dx.
\end{split}
\end{align}

\noindent{\bf Step 4, conclusion: summing the scales and lower bound for the global $L^3$ norm.} The key estimate is \eqref{e.claimstep2}. From \eqref{e.restrR2}-\eqref{e.bdR2prime}, we see that the volume of ${B_0\big(\tfrac{3}{4}C(100)M^{1000}R_{2}'\big)\setminus B_0\big(2R_{2}'\big)}$ is less than or equal to $T_{2}^{\frac{3}{2}}\exp(M^{1021})$. 
By the pigeonhole principle, there exists $i\in \{1,2,3\}$ and $$x_4\in B_0\big(\tfrac{3}{4}C(100)M^{1000}R_{2}'\big)\setminus B_0\big(2R_{2}'\big)\,\,\,\,\textrm{such\,\,that}\,\,\,\,|\omega_{i}(x_4,0)|\geq 2T_2^{-1}\exp(-\exp(M^{1022})).
$$
Let $r_{4}:= T_{2}^{\frac{1}{2}} \exp(-\exp(M^{1022})).$ Using \eqref{e.restrR2}-\eqref{e.defann2}, we see that  $B_{r_{4}}(x_{4})\times \{0\}\subset\mathcal{A}_{2}.$
Thus the quantitative estimate \eqref{e.linftyderivstep2} gives that
$$|\omega_{i}(x,0)|\geq T_{2}^{-1}\exp(-\exp(M^{1022}))\,\,\textrm{in}\,\, B_{r_{4}}(x_{4})$$
and that $\omega_{i}(x,0)$ has constant sign in $B_{r_{4}}(x_{4})$. This along with H\"older's inequality yields that
\begin{align*}
T_2^{-1}\exp(-\exp(M^{1022}))\leq\ &\Big|\int\limits_{B_0(1)}\omega_{i}(x_4-r_4z,0)\varphi(z)\, dz\Big|\\
\leq\ &r_4^{-1}\Big|\int\limits_{B_0(1)}u(x_4-r_{4}z,0)\nabla\times\varphi(z)\, dz\Big|\\
\leq\ &r_4^{-2}\|u\|_{L^3(B_{0}(C(100)M^{1000}R_{2}')\setminus B_{0}(R_{2}'))}\|\nabla\times\varphi\|_{L^{\frac{3}{2}}(B_{0}(1))}
\end{align*}
for a fixed positive $\varphi\in C^\infty_c(B_0(1))$. Hence, using \eqref{e.restrR2}-\eqref{e.bdR2prime} we get 
\begin{equation}\label{e.lowerL3}
\int\limits_{B_0\big(\exp(M^{1023})T_2^\frac12\big)\setminus B_0\big(T_2^\frac12\big)}|u(x,0)|^3\, dx\geq \exp\big(-\exp(M^{1023})\big),
\end{equation}
for all $\frac{8}{C^\sharp }M^{749}(-t_0')\leq T_2\leq 1$. Next we divide into two cases.\\\\
\textbf{Case 1:} $-t'_{0}>\frac{{C}^{\sharp}}{8}M^{-749}\exp(-2M^{1023}) $\\\\\
In this case, we use \eqref{e.lowerL3} with $T_{2}=1$ to immediately get (for $M$ greater than a sufficiently large universal constant) 
$$-t_{0}'\geq \frac{C^\sharp }{8}M^{-749}\exp\Bigg\{-\exp(\exp(M^{1024})))\int\limits_{B_{0}(\exp(M^{1023}))}|u(x,0)|^3\, dx\Bigg\}. $$
\textbf{Case 2:} $-t'_{0}\leq \frac{{C}^{\sharp}}{8}M^{-749}\exp(-2M^{1023}) $\\\\\
In this case we sum \eqref{e.lowerL3} on the
\begin{equation*}
k:=\lfloor\tfrac12M^{-1023}\log(\tfrac{C^\sharp }{8}M^{-749}(-t_0')^{-1})\rfloor\geq 1
\end{equation*}
scales $T_2$,
\begin{align*}
\big(\tfrac{8}{C^\sharp }M^{749}(-t_0')\big)^\frac12\leq\ & \exp(M^{1023})\big(\tfrac{8}{C^\sharp }M^{749}(-t_0')\big)^\frac12\\
\leq\ &\ldots\leq \exp(kM^{1023})\big(\tfrac{8}{C^\sharp }M^{749}(-t_0')\big)^\frac12\leq 1,
\end{align*}
we obtain that 
\begin{align*}
&\exp\big(-\exp({M^{1024}})\big)\log(\tfrac{C^\sharp }{8}M^{-749}(-t_0')^{-1})\\
\leq\ &\int\limits_{B_0(\exp(M^{1023}))\setminus B_0\big(\big(\frac{8}{C^\sharp }M^{749}(-t_0')\big)^\frac12\big)}|u(x,0)|^3\, dx\\
\leq\ &\int\limits_{\R^3}|u(x,0)|^3\, dx.
\end{align*}
This gives $$-t_{0}'\geq \frac{C^\sharp }{8}M^{-749}\exp\Bigg\{-\exp(\exp(M^{1024})))\int\limits_{B_{0}(\exp(M^{1023}))}|u(x,0)|^3\, dx\Bigg\}, $$
which was also obtained in Case 1 and hence applies in all cases. Defining
$$-s_{0}:=\frac{C^\sharp }{16}M^{-749}\exp\Bigg\{-\exp(\exp(M^{1024})))\int\limits_{B_{0}(\exp(M^{1023}))}|u(x,0)|^3\, dx\Bigg\},$$
we see  we have the contrapositive of \eqref{e.conct_0'bis}. In particular, $$ \int\limits_{B_0(4\sqrt{S^\sharp }^{-1}(-s_0)^\frac12)}|\omega(x,s_{0})|^2\, dx\leq M^2(-s_0)^{-\frac12}\sqrt{S^\sharp }.$$ 
almost identical arguments to those utilized in the proof of Lemma \ref{lem.backconc}, except using the bound \eqref{e.conclbddu} instead of \eqref{e.conclbddnablau}, give
$$\|u\|_{L^{\infty}\big(B_{0}\big(C_{2}{M^{50}}(-s_{0})^{\frac{1}{2}}\big)\times \big(\tfrac{s_{0}}{4},0\big)\big)}\leq \frac{C_1M^{-23}}{(-s_{0})^{\frac{1}{2}}}.$$

\noindent{\bf Step 5, conclusion: summing of scales under additional assumption \eqref{L3lowerboundassumption}.}\\
\textbf{Case 1:} $-t'_{0}>\frac{{C}^{\sharp}\lambda^2}{8}M^{-749}\exp(-6M^{1023}) $\\\\\
In this case, we use the additional assumption 
 \eqref{L3lowerboundassumption}  to immediately get 
$$-t_{0}'> \frac{C^\sharp\lambda^2 }{8}M^{-749}\exp\Bigg\{-4M^{1023}\exp(\exp(M^{1024}))\int\limits_{B_{0}(\lambda)}|u(x,0)|^3\, dx\Bigg\}. $$
\textbf{Case 2:} $-t'_{0}\leq \frac{{C}^{\sharp}\lambda^2}{8}M^{-749}\exp(-6 M^{1023}) $\\\\\
First notice that in this case
\begin{equation*}
M^{-1023}\log\Big(\frac{{C}^{\sharp}\lambda^2}{8(-t'_{0})}M^{-749}\Big)\geq 6
\end{equation*}
which implies
\begin{equation}\label{enoughscales1}
{k+1:=\lfloor\tfrac12 M^{-1023}\log\Big(\frac{{C}^{\sharp}\lambda^2}{8(-t'_{0})}M^{-749}\Big)\rfloor\geq 2},
\end{equation}
\begin{equation}\label{enoughscales2}
k+1\geq \tfrac1 4 M^{-1023}\log\Big(\frac{{C}^{\sharp}\lambda^2}{8(-t'_{0})}M^{-749}\Big)
\end{equation}
and 
\begin{equation}\label{lambdabigscalecontrol}
\exp((k+1)M^{1023})\big(\tfrac{8}{C^\sharp }M^{749}(-t_0')\big)^\frac12\leq \lambda<\exp(M^{1023}).
\end{equation}
In this case we sum \eqref{e.lowerL3} on the $k+1\geq 2$ scales $T_2$,
\begin{align*}
\big(\tfrac{8}{C^\sharp }M^{749}(-t_0')\big)^\frac12\leq\ & \exp(M^{1023})\big(\tfrac{8}{C^\sharp }M^{749}(-t_0')\big)^\frac12\\
\leq\ &\ldots\leq \exp(kM^{1023})\big(\tfrac{8}{C^\sharp }M^{749}(-t_0')\big)^\frac12\leq 1.
\end{align*}
Using \eqref{enoughscales2}-\eqref{lambdabigscalecontrol} we obtain 
 
\begin{align*}
&\exp\big(-\exp({M^{1024}})\big)\tfrac{1}{4}M^{-1023}\log(\tfrac{C^\sharp \lambda^2}{8}M^{-749}(-t_0')^{-1})\\
\leq &\int\limits_{B_0(\lambda)\setminus B_0\big(\big(\frac{8}{C^\sharp }M^{749}(-t_0')\big)^\frac12\big)}|u(x,0)|^3\, dx\\
\lesssim\ &\int\limits_{\R^3}|u(x,0)|^3\, dx.
\end{align*}

This gives $${-t_{0}'\geq \frac{C^\sharp\lambda^2 }{8}M^{-749}\exp\Bigg\{-4M^{1023}\exp(\exp(M^{1024}))\int\limits_{B_{0}(\lambda)}|u(x,0)|^3\, dx\Bigg\}}, $$
{which was also obtained in Case 1 and hence applies in all cases.}\\
Defining
$$-s_{1}:=\frac{C^\sharp\lambda^2 }{16}M^{-749}\exp\Bigg\{-4M^{1023}\exp(\exp(M^{1024}))\int\limits_{B_{0}(\lambda)}|u(x,0)|^3\, dx\Bigg\},$$
we see  we have the contrapositive of \eqref{e.conct_0'bis}. In particular, $$ \int\limits_{B_0(4\sqrt{S^\sharp }^{-1}(-s_1)^\frac12)}|\omega(x,s_{1})|^2\, dx\leq M^2(-s_1)^{-\frac12}\sqrt{S^\sharp }.$$ 
almost identical arguments to those utilized in the proof of Lemma \ref{lem.backconc}, except using the bound \eqref{e.conclbddu} instead of \eqref{e.conclbddnablau}, give
$$\|u\|_{L^{\infty}\big(B_{0}\big(C_{2}{M^{50}}(-s_{1})^{\frac{1}{2}}\big)\times \big(\tfrac{s_{1}}{4},0\big)\big)}\leq \frac{C_1M^{-23}}{(-s_{1})^{\frac{1}{2}}}.$$ 
This concludes the proof of Proposition \ref{prop.main}.

\subsection{Proof of the main estimate in the time slices case}
\label{subsec.carlts}

We give the full proof of Proposition \ref{prop.maints} for the sake of completeness. Notice that the proof follows the same scheme as the proof of Proposition \ref{prop.main}. In most estimates $M$ is simply replaced by $\Mp$, albeit with slightly different powers. However, since in Step 2 below concentration is needed on the very small time interval $[-\frac{T_2}{4(\Mp)^{201}},0]$, some care is needed when applying Lemma \ref{epochtimeslice} on the epoch of quantitative regularity and Lemma \ref{lem.backconctimeslice} on the backward propagation of concentration.

Let 
$M\in[M_4,\infty)$ where $M_4$ is a constant in Lemma \ref{lem.backconctimeslice}. 
In the course of the proof we will need to take $M$ larger, always larger than universal constants. Let $u:\, \R^3\times [-1,0]\rightarrow\R^3$ be a $C^{\infty}(\mathbb{R}^3\times (-1,0))$ finite-energy solution to the Navier-Stokes equations \eqref{e.nse} in $I=[-1,0]$. Assume  that there exists $t_{(k)}\in [-1,0)$ such that
\begin{equation}\label{e.tsArecap}
t_{(k)}\uparrow 0\,\,\,\,\textrm{with}\,\,\,\, \sup_{k} \|u(\cdot, t_{(k)})\|_{L^{3}(\mathbb{R}^3)}\leq M.
\end{equation}
Selecting any ``well-separated'' subsequence (still denoted $t_{(k)}$) such that\footnote{This separation condition is stronger than that of Lemma \ref{lem.backconctimeslice}. This stronger condition is needed to sum disjoint annuli in Step 4 below.}
\begin{equation}\label{wellsepproptsrecap}
\sup_{k}\frac{-t_{(k+1)}}{-t_{(k)}}<\exp(-2(\Mp)^{1223}).
\end{equation}
For this well-separated subsequence, assume that there exists $j+1$ such that the vorticity concentrates at time $t_{(j+1)}$ in the following sense
\begin{equation}\label{e.conctsrecap}
\int\limits_{B_0(4\sqrt{S^\flat}^{-1}(-t_{(j+1)})^\frac12)}|\omega(x,t_{(j+1)})|^2\, dx> M^2(-t_{(j+1)})^{-\frac12}\sqrt{S^\flat}.
\end{equation}
where we recall that $S^\flat =O(1)M^{-100}$. 
Fix $k\in \{1,2,\ldots, j\}.$ 
Note that \eqref{wellsepproptsrecap} implies that for $M$ sufficiently large
$$
\frac{-t_{(j+1)}}{-t_{(k)}}<(\Mp)^{-1051}.
$$
Lemma \ref{lem.backconctimeslice} then implies imply that the vorticity concentrates in the following sense
\begin{equation}\label{e.concs_0recap}
\int\limits_{B_0(4(-s)^\frac12 (\Mp)^{106})}|\omega(x,s)|^2\, dx> \frac{(M+1)^2}{(-s)^{\frac12}(\Mp)^{106}}.
\end{equation}
for  any \begin{equation}\label{stkcondition}
s\in \Big[t_{(k)}, \frac{t_{(k)}}{8(\Mp)^{201}}\Big].
\end{equation}

\noindent{\bf Step 1: quantitative unique continuation.} The purpose of this step is to prove the following estimate:
\begin{align}\label{e.claimstep1ts}
T_1^{\frac12}e^{-\frac{O(1)(\Mp)^{965}R^2}{T_1}}
\lesssim\ & \int\limits_{-T_1}^{-\frac{T_1}2}\int\limits_{B_0(2R)\setminus B_0(R/2)}|\omega(x,t)|^2\, dxdt,
\end{align}
for all $T_1$, $s_{0}$ and $R$ such that 
\begin{equation}\label{e.wellsepT1ts}
s_{0}\in\Big[\frac{t_{(k)}}{2}, \frac{t_{(k)}}{4(\Mp)^{201}}\Big]\,\,\,\,\,\,T_{1}:=-s_{0}\,\,\,\,\quad\mbox{and}\quad R\geq (\Mp)^{100}\Big(\frac{T_1}{2}\Big)^\frac12.
\end{equation} 
Here, $k\in\{1,\ldots j\}$ is fixed. 
Let  $I_1:=(-T_1,-\frac{T_1}{2})\subset[\frac{t_{(k)}}{2}, \frac{t_{(k)}}{8(\Mp)^{201}}]\subset [-1,0]$. The bound \eqref{carlemantimesliceepochbound} in Remark \ref{rem.estcarlepotimeslice} implies that there exists an epoch of regularity $I_1''=[t_1''-T_1'',t_1'']\subset I_1$ such that 
\begin{equation}\label{e.sizeI1'ts}
T_1''=|I_1''|=\frac{(\Mp)^{-864}}{4C_{4}^3}|I_1|=\frac{(\Mp)^{-864}}{8C_{4}^3}T_1
\end{equation}
and for $j=0, 1, 2$,
\begin{equation}\label{e.epochI1'ts}
\|\nabla^j u\|_{L^{\infty}_{t}L^{\infty}_{x}(\mathbb{R}^3\times I_1'')}\leq \frac{1}{2^{j+1}} |I_1''|^{\frac{-(j+1)}{2}}=\frac{1}{2^{j+1}}(T_1'')^{\frac{-(j+1)}{2}}.
\end{equation}
Let $T_1''':=\frac34T_1''$ and $s''\in [t_1''-\frac{T_1''}4,t_1'']$. Let $x_1\in\R^3$ be such that $|x_1|\geq (\Mp)^{100}(\frac{T_1}2)^\frac12$ and let 
$r_1:=(\Mp)^{50}|x_1|\geq (\Mp)^{150}(\frac{T_1}2)^\frac12$. 
Notice that for $M$ large enough 
\begin{equation}\label{e.condr1ts}
\frac{(\Mp)^7|x_{1}|}{2}\geq \frac{(\Mp)^{107}}{2}\Big(\frac{T_1}2\Big)^\frac12\geq 4(-s_{0})^{\frac{1}{2}}(\Mp)^{106}
\end{equation}
and 
\begin{equation*}
r_1^2\geq 4000T_1'''.
\end{equation*}
We apply the second Carleman inequality, Proposition \ref{prop.sndcarl} (quantitative unique continuation), on the cylinder $\mathcal C_1=\{(x,t)\in\mathbb R^3\times\mathbb R\, :\ t\in [0,T_1'''],\ |x|\leq r_1\}$ to the function $w:\, \R^3\times [0,T_1''']\rightarrow\R^3$, defined by for all $(x,t)\in \R^3\times [0,T_1''']$, 
\begin{equation*}
w(x,t):=\omega(x_1+x,s''-t).
\end{equation*}
Notice that  the quantitative regularity \eqref{e.epochI1'ts} and the vorticity equation \eqref{vort} implies that on $\mathcal C_1$
\begin{equation*}
|(\partial_t+\Delta)w|\leq \frac3{16}{T_1'''}^{-1}|w|+\frac{\sqrt{3}}4{T_1'''}^{-\frac12}|\nabla w|,
\end{equation*}
so that \eqref{e.diffineqC} is satisfied with $S=S_1:=T_1'''$ and $C_{Carl}=\frac{16}3$. Let 
\begin{equation*}
\hat s_1=\frac{T_1'''}{20000},\qquad \check s_1=(\Mp)^{-150}T_1'''.
\end{equation*}
For $M$ sufficiently large we have $0<\check s_1\leq \hat s_1\leq\frac{T_1'''}{10000}$. Hence by \eqref{e.conclcarltwo} we have 
\begin{equation}\label{e.estXYZ1ts}
Z_1\lesssim e^{-\frac{r_1^2}{500\hat s_1}}X_1+(\hat s_1)^\frac32\Big(\frac{e\hat s_1}{\check s_1}\Big)^{\frac{O(1)r_1^2}{\hat s_1}}Y_1,
\end{equation}
where
\begin{align*}
&X_1:=\int\limits_{s''-T_1'''}^{s''}\int\limits_{B_{x_1}((\Mp)^{50}|x_1|)}((T_1''')^{-1}|\omega|^2+|\nabla\omega|^2)\, dxds,\\
&Y_1:=\int\limits_{B_{x_1}((\Mp)^{50}|x_1|)}|\omega(x,s'')|^2(\check s_1)^{-\frac32}e^{-\frac{|x-x_1|^2}{4\check s_1}}\, dx,\\
&Z_1:=\int\limits_{s''-\frac{T_1'''}{10000}}^{s''-\frac{T_1'''}{20000}}\int\limits_{B_{x_1}(\frac{(\Mp)^{50}|x_1|}2)}((T_1''')^{-1}|\omega|^2+|\nabla\omega|^2)e^{-\frac{|x-x_1|^2}{4(s''-s)}}\, dxds.
\end{align*}
We first use the concentration \eqref{e.concs_0recap} for times 
\begin{equation*}
s\in \Big[s''-\frac{T_1'''}{10000},s''-\frac{T_1'''}{20000}\Big]\subset \Big(s_{0}, \frac{s_{0}}{2}\Big)\subset \Big(t_{(k)}, \frac{t_{(k)}}{8(\Mp)^{201}}\Big)
\end{equation*}
to bound $Z_1$ from below. By \eqref{e.condr1ts}, we have
\begin{equation*}
B_0(4(-s)^\frac12 (\Mp)^{106})\subset B_0(4(-s_{0})^\frac12 (\Mp)^{106})\subset
B_{0}\Big(\frac{(\Mp)^7|x_1|}{2}\Big)\subset B_{x_1}((\Mp)^7|x_1|)
\end{equation*}
for all $s\in [s''-\frac{T_1'''}{10000},s''-\frac{T_1'''}{20000}]$ and for $M$ sufficiently large. 
Hence, we have
\begin{align*}
Z_1\gtrsim\ &\int\limits_{s''-\frac{T_1'''}{10000}}^{s''-\frac{T_1'''}{20000}}\int\limits_{B_0(4(-s)^\frac12(\Mp)^{106})}(T_1''')^{-1}|\omega(x,s)|^2\, dxds\, e^{-\frac{O(1)(\Mp)^{14}|x_1|^2}{T_1'''}}\\
\gtrsim\ &\int\limits_{s''-\frac{T_1'''}{10000}}^{s''-\frac{T_1'''}{20000}}(\Mp)^{-106}(-s)^{-\frac12}\, ds(T_1'')^{-1}e^{-\frac{O(1)(\Mp)^{14}|x_1|^2}{T_1''}}\\
\gtrsim\ & (\Mp)^{-106}\frac{T_1'''}{(-s''+\frac{T_1'''}{10000})^{\frac{1}{2}}}(T_1'')^{-1}e^{-\frac{O(1)(\Mp)^{14}|x_1|^2}{T_1''}}\\
\gtrsim\ &(\Mp)^{-106}(T_1)^{-\frac12}e^{-\frac{O(1)(\Mp)^{14}|x_1|^2}{T_1''}}\\
\gtrsim\ &(\Mp)^{-106}((\Mp)^{864}T_1'')^{-\frac12}e^{-\frac{O(1)(\Mp)^{14}|x_1|^2}{T_1''}}\\=& (\Mp)^{-538}(T_1'')^{-\frac{1}{2}}e^{-\frac{O(1)(\Mp)^{14}|x_1|^2}{T_1''}}.
\end{align*}
Second, we bound from above $X_1$. We rely on the quantitative regularity \eqref{e.epochI1'ts} to obtain
\begin{align*}
X_1\lesssim  (T_1'')^{-2}(\Mp)^{150}|x_1|^3.
\end{align*}
Hence,
\begin{align*}
e^{-\frac{r_1^2}{500\hat s_1}}X_1\lesssim\ & (T_1'')^{-2}(\Mp)^{150}|x_1|^3e^{-\frac{O(1)(\Mp)^{100}|x_1|^2}{T_1''}}\\
\lesssim\ &(T_1'')^{-\frac12}e^{-\frac{O(1)(\Mp)^{100}|x_1|^2}{T_1''}}.
\end{align*}
Third, for $Y_1$ we decompose and estimate as follows
\begin{align*}
Y_1:=\ &\int\limits_{B_{x_1}(\frac{|x_1|}2)}|\omega(x,s'')|^2(\check s_1)^{-\frac32}e^{-\frac{|x-x_1|^2}{4\check s_1}}\, dx\\
&+\int\limits_{B_{x_1}((\Mp)^{50}|x_1|)\setminus B_{x_1}(\frac{|x_1|}2)}|\omega(x,s'')|^2(\check s_1)^{-\frac32}e^{-\frac{|x-x_1|^2}{4\check s_1}}\, dx\\
\lesssim\ &(\Mp)^{225}(T_1'')^{-\frac32}\Bigg(\int\limits_{B_{x_1}(\frac{|x_1|}2)}|\omega(x,s'')|^2\, dx\\
&+\int\limits_{B_{x_1}((\Mp)^{50}|x_1|)\setminus B_{x_1}(\frac{|x_1|}2)}|\omega(x,s'')|^2
e^{-\frac{O(1)(\Mp)^{150}|x_1|^2}{T_1''}}\, dx\Bigg)\\
\lesssim\ &(\Mp)^{225}(T_1'')^{-\frac32}\Bigg(\int\limits_{B_{x_1}(\frac{|x_1|}2)}|\omega(x,s'')|^2\, dx+(\Mp)^{150}|x_1|^3(T_1'')^{-2}e^{-\frac{O(1)(\Mp)^{150}|x_1|^2}{T_1''}}\Bigg)\\
\lesssim\ &(\Mp)^{225}(T_1'')^{-\frac32}\Bigg(\int\limits_{B_{x_1}(\frac{|x_1|}2)}|\omega(x,s'')|^2\, dx+(T_1'')^{-\frac12}e^{-\frac{O(1)(\Mp)^{150}|x_1|^2}{T_1''}}\Bigg),
\end{align*}
where we used the quantitative regularity \eqref{e.epochI1'ts}. Hence,
\begin{align*}
(\hat s_1)^\frac32\Big(\frac{e\hat s_1}{\check s_1}\Big)^{\frac{O(1)r_1^2}{\hat s_1}}Y_1\lesssim\ &(T_1'')^\frac32e^{\frac{O(1)(\Mp)^{100}|x_1|^2}{T_1''}\log(\frac{e(\Mp)^{150}}{20000})}Y_1\\
\lesssim\ &(\Mp)^{225}e^{\frac{O(1)(\Mp)^{101}|x_1|^2}{T_1''}}\int\limits_{B_{x_1}(\frac{|x_1|}2)}|\omega(x,s'')|^2\, dx\\
&+(\Mp)^{225}(T_1'')^{-\frac12}e^{-\frac{O(1)(\Mp)^{150}|x_1|^2}{T_1''}}.
\end{align*}
Gathering these bounds and combining with \eqref{e.estXYZ1ts} yields 
\begin{align*}
(\Mp)^{-538}(T_1'')^{-\frac12}e^{-\frac{O(1)(\Mp)^{14}|x_1|^2}{T_1''}}\lesssim\ & (T_1'')^{-\frac12}e^{-\frac{O(1)(\Mp)^{100}|x_1|^2}{T_1''}}\\&+(\Mp)^{225}e^{\frac{O(1)(\Mp)^{101}|x_1|^2}{T_1''}}\int\limits_{B_{x_1}(\frac{|x_1|}2)}|\omega(x,s'')|^2\, dx\\&+(\Mp)^{225}(T_1'')^{-\frac12}e^{-\frac{O(1)(\Mp)^{150}|x_1|^2}{T_1''}},
\end{align*}
which implies 
\begin{align*}
(\Mp)^{-763}(T_1'')^{-\frac12}e^{-\frac{O(1)(\Mp)^{101}|x_1|^2}{T_1''}}\lesssim\ & \int\limits_{B_{x_1}(\frac{|x_1|}2)}|\omega(x,s'')|^2\, dx.
\end{align*}
Hence, for all  
 $s''\in I_1''=[t_1''-\frac{T_1''}{4},t_1'']$, for all $|x_1|\geq (\Mp)^{100}(\frac{T_1}{2})^\frac12$,
\begin{equation*}
\int\limits_{B_{x_1}(\frac{|x_1|}{2})}|\omega(x,s'')|^2\, dx\gtrsim (\Mp)^{-763}(T_1'')^{-\frac12}e^{-\frac{O(1)(\Mp)^{101}|x_{1}|^2}{T_1''}}.
\end{equation*}
Let $R\geq (\Mp)^{100}(\frac{T_1}2)^\frac12$ and $x_{1}\in\mathbb{R}^3$ be such that $|x_{1}|=R$. 
Integrating in time $[t_1''-\frac{T_1''}4,t_1'']$ yields the estimate 
$$
(\Mp)^{-763}(T_1'')^{\frac12}e^{-\frac{O(1)(\Mp)^{101}R^2}{T_1''}}
\lesssim \int\limits_{t_1''-\frac{T_1''}4}^{t_1''}\int\limits_{B_0(2R)\setminus B_0(R/2)}|\omega(x,t)|^2\, dxdt
$$
which yields the claim \eqref{e.claimstep1ts} of Step 1.

\noindent{\bf Step 2: quantitative backward uniqueness.} The goal of this step and Step 3 below is to prove the following claim:
\begin{align}\label{e.claimstep2ts}
\begin{split}
T_2^{-\frac12}\exp\big(-\exp({{(\Mp)^{1221}}})\big)\lesssim\ &\int\limits_{B_0\big(\frac{3(\Mp)^{1200}}{4}R'_{2}\big)\setminus B_0(2R'_{2})}|\omega(x,0)|^2\, dx,
\end{split}
\end{align}
for $T_{2}=-t_{(k)}$ with $k\in\{1,\ldots j\}$. Here, $R_{2}$ and $R'_{2}$ are as in \eqref{e.restrR2ts}-\eqref{e.bdR2primets}. This is the key estimate for Step 4 below and the proof of Proposition \ref{prop.maints}.

We apply here the results of Section \ref{sec.quantannulus} for the quantitative existence of an annulus of regularity. Although the parameter $\mu$ in Section \ref{sec.quantannulus} is any positive real number, 
here we need to take $\mu$ sufficiently large in order to have a large enough annulus of quantitative regularity, and hence a large $r_+$ below in the application of the first Carleman inequality, Proposition \ref{prop.firstcarl}. 
To fix the ideas, we take $\mu=120$. Let $T_1$ and $T_2$ such that
\begin{equation}\label{e.sepT2ts}
T_{2}:=-t_{(k)}\quad\mbox{and}\quad T_1:=\frac{T_2}{4(\Mp)^{201}}.
\end{equation}
Let 
\begin{equation}\label{e.restrR2ts}
R_2:=K^\flat(T_2)^\frac12
\end{equation}
for a universal constant $K^\flat\geq 1$ to be chosen sufficiently large below. In particular it is chosed such that \eqref{e.smallKflat} holds. 
By Corollary \ref{cor.rescaleannulusts} applied on the epoch $(t_{(k)},0)$, for $M\geq M_{2}(120)$ there exists a scale
\begin{equation}\label{e.bdR2primets}
2R_2\leq R_2'\leq 2R_2\exp(C(120)(\Mp)^{1220})
\end{equation}
and a good cylindrical annulus 
\begin{equation}\label{e.defann2ts}
\mathcal A_2:=\{R_2'<|x|<(\Mp)^{1200}R_2'\}\times\Big(-\frac{T_2}{32},0\Big)
\end{equation} 
such that for $j=0,1$,
\begin{align}\label{e.linftyderivstep2ts}
\begin{split}
\|\nabla^j u\|_{L^{\infty}(\mathcal A_2)}
\leq\ & 2^{j+1}\bar{C}_{j} (\Mp)^{-{360}
}T_2^{-\frac{j+1}{2}},\\
\|\nabla \omega\|_{L^{\infty}(\mathcal A_2)}
\leq\ & 2^\frac32\bar{C}_{2} (\Mp)^{-{360}
}T_2^{-\frac32}.
\end{split}
\end{align}
We apply now the quantitative backward uniqueness, Proposition \ref{prop.firstcarl} to the function $w:\, \R^3\times [0,\frac{T_2}{(\Mp)^{201}}]\rightarrow\R^3$ defined by for all $(x,t)\in\R^3\times [0,\frac{T_2}{(\Mp)^{201}}]$,
\begin{equation*}
w(x,t)=\omega(x,-t).
\end{equation*}
An important remark is that although we have a large cylindrical annulus of quantitative  regularity $\mathcal A_2$, we apply the Carleman estimate on a much smaller annulus, namely
\begin{equation}\label{e.defann2bists}
\widetilde{\mathcal A}_2:=\Big\{4R_2'<|x|<\frac{(\Mp)^{1200}}{4}R_2'\Big\}\times\Big(-\frac{T_2}{(\Mp)^{201}},0\Big).
\end{equation}
The reason  for this is to ensure we can apply Step 1 to get a lower bound \eqref{e.lowerZ2ts} for $Z_{2}$.

Choosing $M$ sufficiently large such that $2\bar{C}_{j} (\Mp)^{-{360}}\leq 1$ and $2^\frac32\bar{C}_{2} (\Mp)^{-{360}}\leq 1$, we see that the bounds \eqref{e.linftyderivstep2ts} imply that the differential inequality \eqref{e.diffineq} is satisfied with $S=S_2:=\frac{T_2}{(\Mp)^{201}}$ and 
$C_{Carl}=(\Mp)^{201}$. Take 
\begin{equation*}
r_-=4R_2',\qquad r_+=\tfrac{1}{4}{(\Mp)^{1200}}R_2'.
\end{equation*}
Then, 
\begin{equation*}
B_0(160R_2')\setminus B_0(40R_2')=B_0(40r_-)\setminus B_0(10r_-)\subset\left\{40R_2'<|x|<\tfrac{1}{8}{(\Mp)^{1200}}R_2'\right\}
\end{equation*}
on condition that $M$ is sufficiently large: one needs $(\Mp)^{1200}>1280$. By \eqref{e.conclcarlone}, we get
\begin{equation}\label{e.conclcarlonebists}
Z_2\lesssim e^{-\frac{O(1)(\Mp)^{1200}(R_2')^2}{T_2}}\big(X_2+e^{\frac{O(1)(\Mp)^{2400}(R_2')^2}{T_2}}Y_2\big),
\end{equation}
where 
\begin{align*}
X_2:=\ &\int\limits_{-\frac{T_{2}}{{(\Mp)}^{201}}}\int\limits_{r_{-}\leq |x|\leq r_{+}}e^{\frac{4|x|^2}{T_2}}((\Mp)^{201}T_2^{-1}|\omega|^2+|\nabla \omega|^2)\, dxdt,\\
Y_2:=\ &\int\limits_{r_-\leq |x|\leq r_+}|\omega(x,0)|^2\, dx,\\
Z_2:=\ &\int\limits_{-\frac{T_2}{4(\Mp)^{201}}}^{0}\int\limits_{10r_-\leq|x|\leq \frac{r_+}2}((\Mp)^{201}T_2^{-1}|\omega|^2+|\nabla \omega|^2)\, dxdt.
\end{align*}
For $M$ large enough \eqref{e.restrR2ts} implies
\begin{equation*}
20r_{-}\geq 10R_2'\geq 20R_2=20K^\flat(T_2)^\frac12\geq (\Mp)^{100}\Big(\frac{T_2}{8(\Mp)^{201}}\Big)^\frac12=(\Mp)^{100}\Big(\frac{T_1}{2}\Big)^\frac12.
\end{equation*} 
Hence, we can apply the concentration result of Step 1, taking $T_1=\frac{T_2}{4 (\Mp)^{201}}=\frac{-t_{(k)}}{4(\Mp)^{201}}=\frac {S_2}4$ and $R=20r_-$. By \eqref{e.claimstep1ts} we have that
\begin{equation}\label{e.lowerZ2ts}
Z_2\gtrsim (\Mp)^{201}\left(\frac{T_2}{4(\Mp)^{201}}\right)^\frac12e^{-\frac{O(1)(\Mp)^{1166}(R_2')^2}{T_2}}T_2^{-1}\gtrsim T_2^{-\frac12}e^{-\frac{O(1)(\Mp)^{1166}(R_2')^2}{T_2}}.
\end{equation}
Therefore, one of the following two lower bounds holds
\begin{align}
T_2^{-\frac12}\exp\Big(\frac{O(1)M^{1200}(R_2')^2}{T_2}\Big)\lesssim X_2,\label{e.step2lowerbd1ts}\\
T_2^{-\frac12}\exp(-\exp({(\Mp)^{1221}})
)
\lesssim e^{-\frac{O(1)(\Mp)^{2400}(R_2')^2}{T_2}}T_{2}^{-\frac12}\lesssim Y_2,\label{e.step2lowerbd2ts}
\end{align}
where we used the upper bound \eqref{e.bdR2primets} for \eqref{e.step2lowerbd2ts}. 
The bound \eqref{e.step2lowerbd2ts} can be used directly in Step 4 below. On the contrary, if \eqref{e.step2lowerbd1ts} holds more work needs to be done to transfer the lower bound on the enstrophy at time $0$. This is the objective of Step 3 below.

\noindent{\bf Step 3: a final application of quantitative unique continuation.} Assume that the bound \eqref{e.step2lowerbd1ts} holds. We will apply the pigeonhole principle three times successively in order to end up in a situation where we can rely on the quantitative unique continuation to get a lower bound at time $0$. We first remark that this with the definition \eqref{e.defann2bists} of the annulus $\widetilde{\mathcal A}_2$ implies the following lower bound
\begin{multline*}
T_2^{-\frac12}\exp\Big(\frac{O(1)(\Mp)^{1200}(R_2')^2}{T_2}\Big)\\
\lesssim \int\limits_{-\frac{T_2}{(\Mp)^{201}}}^{0}\int\limits_{4R_2'\leq|x|\leq \frac{(\Mp)^{1200}}{4}R_2'}e^{\frac{4|x|^2}{T_2}}((\Mp)^{201}T_2^{-1}|\omega|^2+|\nabla \omega|^2)\, dxdt.
\end{multline*}
By the pigeonhole principle, there exists 
\begin{equation}\label{e.condR3ts}
8R_2'\leq R_3\leq \tfrac12{(\Mp)^{1200}}R_2'
\end{equation}
such that 
\begin{equation*}
T_2^{-\frac12}\exp\Big(-\frac{4R_3^2}{T_2}\Big)
\lesssim \int\limits_{-\frac{T_2}{(\Mp)^{201}}}^0\int\limits_{B_0(R_3)\setminus B_0(\frac{R_3}2)}(T_2^{-1}|\omega|^2+|\nabla \omega|^2)\, dxdt.
\end{equation*}
Using the bounds \eqref{e.linftyderivstep2ts}, we have that 
\begin{equation*}
T_2^{-\frac12}\exp\Big(-\frac{4R_3^2}{T_2}\Big)\lesssim \int\limits_{-\frac{T_2}{(\Mp)^{201}}}^{-\exp(-\frac{8R_3^2}{T_2})T_2}\int\limits_{B_0(R_3)\setminus B_0(\frac{R_3}2)}(T_2^{-1}|\omega|^2+|\nabla \omega|^2)\, dxdt.
\end{equation*}
By the pigeonhole principle, there exists 
\begin{equation}\label{e.condt3ts}
\frac12\exp\Big(-\frac{8R_3^2}{T_2}\Big)T_2\leq -t_3\leq\frac{T_2}{(\Mp)^{201}}
\end{equation}
such that
\begin{equation*}
T_2^{-\frac12}\exp\Big(-\frac{5R_3^2}{T_2}\Big)\lesssim \int\limits_{2t_3}^{t_3}\int\limits_{B_0(R_3)\setminus B_0(\frac{R_3}2)}(T_2^{-1}|\omega|^2+|\nabla \omega|^2)\, dxdt.
\end{equation*}
We finally cover the annulus $B_0(R_3)\setminus B_0(\frac{R_3}2)$ with
\begin{equation*}
O(1)\frac{R_3^3}{(-t_3)^\frac32}\lesssim \frac{R_3^3}{T_2^\frac32}\exp\Big(\frac{12R_3^2}{T_2}\Big)\lesssim \exp\Big(\frac{13R_3^2}{T_2}\Big)
\end{equation*}
balls of radius $(-t_3)^\frac12$, and apply the pigeonhole principle a third time to find that there exists $x_3\in B_0(R_3)\setminus B_0(\frac{R_3}2)$ such that 
\begin{equation}\label{e.conct3ts}
T_2^{-\frac12}\exp\Big(-\frac{18R_3^2}{T_2}\Big)\lesssim \int\limits_{2t_3}^{t_3}\int\limits_{B_{x_3}((-t_3)^\frac12)}(T_2^{-1}|\omega|^2+|\nabla \omega|^2)\, dxdt.
\end{equation}
We apply now the second Carleman inequality, Proposition \ref{prop.sndcarl}, to the function $w:\, \R^3\times [0,-20000t_{3}]\rightarrow\R^3$ defined by for all $(x,t)\in\R^3\times [0,-20000t_{3}]$,
\begin{equation*}
w(x,t)=\omega(x+x_3,-t).
\end{equation*}
Let $S_3:=-20000t_{3}$. We take\footnote{As in the proof of Proposition \ref{prop.main} above, we follow here again Tao's idea; see footnotes \ref{foot.tao} and \ref{foot.taobis}.}
\begin{equation}
r_3:=1000R_{3}\Big(-\frac{t_{3}}{T_{2}}\Big)^{\frac{1}{2}},\quad \hat s_3=\check s_3=-t_3.
\end{equation}
Notice that due to \eqref{e.restrR2ts}-\eqref{e.bdR2primets} and \eqref{e.condR3ts}, we have that
\begin{align}\label{e.condx3ts}
r_3^2=\ &{10^6 R_{3}^2\Big(-\frac{t_{3}}{T_{2}}\Big)}\geq (2.56\times 10^8)(K^\flat)^2(-t_{3})\geq 4000S_{3}=(8\times 10^7)(-t_{3}),\\
\frac{r_{3}}{2}\geq\ &8000R_{2}\Big(-\frac{t_{3}}{T_{2}}\Big)^{\frac{1}{2}}=8000K^\flat (-t_3)^\frac12> (-t_3)^\frac12,
\end{align}
so that \eqref{e.lowerr} is satisfied. 
Furthermore, from \eqref{e.condt3ts} we have
$$\frac{|x_{3}|}{2}\geq \frac{R_{3}}{4}\geq 1000R_{3}\Big(\frac{1}{(\Mp)^{201}}\Big)^{\frac{1}{2}}\geq r_{3}. $$
Thus
\begin{multline}\label{domaininclusts}
B_{x_3}((-t_3)^\frac12)\subset B_{x_3}(\tfrac{r_3}2)\subset B_{x_{3}}(r_{3}) \subset B_{x_3}\Big(\frac{|x_{3}|}{2}\Big)\\
\subset\{\tfrac{R_3}{4}<|y|<\tfrac{3}{2}R_3\}\subset\Big\{2R_2'<|y|<\frac{3(\Mp)^{1200}}{4}R_2'\Big\}. 
\end{multline}
Moreover, 
\begin{equation*}
0\leq\hat s_3=\check s_3=-t_3\leq -2t_{3} =\frac{S_3}{10^4}.
\end{equation*} 
By \eqref{e.condt3ts}, we see that for $M$ large enough $S_{3}\leq \frac{T_{2}}{32}$, hence the bounds \eqref{e.linftyderivstep2ts} imply that the differential inequality \eqref{e.diffineq} is satisfied with $S=S_3$ and $C_{Carl}=1$.
Therefore, by \eqref{e.conclcarltwo} we have
\begin{equation}\label{e.conclcarltwobists}
Z_3\leq C_{univ}e^{\frac{r_3^2}{500t_3}}X_3+C_{univ}(-t_3)^\frac32e^{-\frac{O(1)r_3^2}{t_3}}Y_3,
\end{equation}
where 
\begin{align*}
X_3:=\ &\int\limits_{-S_{3}}^0\int\limits_{B_{x_3}(r_3)}(S_{3}^{-1}|\omega|^2+|\nabla\omega|^2)\, dxdt,\qquad Y_3:=\int\limits_{B_{x_3}(r_3)}|\omega(x,0)|^2(-t_3)^{-\frac32}e^{\frac{|x-x_3|^2}{4t_3}}\, dx,\\
Z_3:=\ &\int\limits_{2t_3}^{t_3}\int\limits_{B_{x_3}(\frac{r_3}2)}(S_{3}^{-1}|\omega|^2+|\nabla \omega|^2)e^{\frac{|x-x_3|^2}{4t}}\, dxdt.
\end{align*}
Using \eqref{e.conct3ts} and $T_{2}^{-1}\leq S_{3}^{-1}$ we have 
\begin{equation}\label{Z3lowerboundts}
  T_2^{-\frac12}\exp\Big(-\frac{18R_3^2}{T_2}\Big)\lesssim \int\limits_{2t_3}^{t_3}\int\limits_{B_{x_3}((-t_{3})^{\frac{1}{2}})}(T_2^{-1}|\omega|^2+|\nabla \omega|^2)e^{\frac{|x-x_3|^2}{4t}}\, dxdt\leq Z_{3}
\end{equation} 
Using the bounds \eqref{e.linftyderivstep2ts} along with \eqref{e.condt3ts}, we find that as in \eqref{e.controlX3},
\begin{equation}\label{e.controlX3ts}
C_{univ}e^{\frac{r_3^2}{500t_3}}X_3\lesssim T_2^{-\frac12}e^{-\frac{996 R_3^2}{T_2}}\leq C_{univ}'e^{-\frac{18R_3^2}{T_2}}e^{-978\cdot 256(K^\flat)^2}.
\end{equation}
We choose $K^\flat$ sufficiently large such that 
\begin{equation}\label{e.smallKflat}
C_{univ}'e^{-978\cdot 256(K^\flat)^2}\leq\frac12,
\end{equation}
where $C_{univ}'\in(0,\infty)$ is the constant appearing in the last inequality of \eqref{e.controlX3ts}. 
Combining now \eqref{e.conclcarltwobists} with the lower bound \eqref{Z3lowerboundts}, we obtain
\begin{align*}
T_2^{-\frac12}\exp\Big(-\frac{18R_3^2}{T_2}\Big)\lesssim\ 
 & \exp\Big(-{\frac{O(1)r_3^2}{t_3}}\Big)\int\limits_{B_{x_3}(r_3)}|\omega(x,0)|^2\, dx\\
\lesssim\ & \exp\Big(O(1){\frac{R_3^2}{T_2}}\Big)\int\limits_{B_{x_3}(r_3)}|\omega(x,0)|^2\, dx.
\end{align*}
Hence,
\begin{equation*}
T_2^{-\frac12}\exp\Big(-O(1){\frac{R_3^2}{T_2}}\Big)\lesssim \int\limits_{B_{x_3}(r_3)}|\omega(x,0)|^2\, dx.
\end{equation*}
Using \eqref{e.restrR2ts}, \eqref{domaininclusts} and the upper bound 
\begin{equation*}
R_3\leq \tfrac{1}2(\Mp)^{1200}R_2'\leq (\Mp)^{1200}\exp(C(120)(\Mp)^{1220})R_2,
\end{equation*}
it follows that 
\begin{align}\label{e.lastbddstep3ts}
\begin{split}
T_2^{-\frac12}\exp(-\exp((\Mp)^{1221}))
\lesssim\ &\int\limits_{B_0\big(\frac{3(\Mp)^{1200}}{4}R'_{2}\big)\setminus B_{0}(2R'_{2})}|\omega(x,0)|^2\, dx.
\end{split}
\end{align}
This together with the bounds \eqref{e.bdR2primets} and \eqref{e.condR3ts} for $R_3$ proves the claim \eqref{e.claimstep2ts}.

\noindent{\bf Step 4, conclusion: summing the scales and lower bound for the global $L^3$ norm.} The key estimate is \eqref{e.claimstep2ts}. From \eqref{e.restrR2ts}-\eqref{e.bdR2primets}, we see that the volume of the annulus ${B_0\big(\frac{3(\Mp)^{1200}}{4}R_{2}'\big)\setminus B_0(2R_{2}')}$ is less than or equal to $T_{2}^{\frac{3}{2}}\exp((\Mp)^{1221})$. 
By the pigeonhole principle, there exists $i\in \{1,2,3\}$ and $$x_4\in B_0\Big(\frac{3(\Mp)^{1200}}{4}R_{2}'\Big)\setminus B_0\big(2R_{2}'\big)\,\,\,\,\textrm{such\,\,that}\,\,\,\,|\omega_{i}(x_4,0)|\geq 2T_2^{-1}\exp(-\exp((\Mp)^{1222})).
$$
Let $r_{4}:= T_{2}^{\frac{1}{2}} \exp(-\exp((\Mp)^{1222})).$ Using \eqref{e.restrR2ts}-\eqref{e.defann2ts}, we see that  $B_{r_{4}}(x_{4})\times \{0\}\subset \mathcal{A}_{2}.$
Thus the quantitative estimate \eqref{e.linftyderivstep2ts} gives that
$$|\omega_{i}(x,0)|\geq T_{2}^{-1}\exp(-\exp((\Mp)^{1222}))\,\,\textrm{in}\,\, B_{r_{4}}(x_{4})$$
and that $\omega_{i}(x,0)$ has constant sign in $B_{r_{4}}(x_{4})$. This along with H\"older's inequality yields that
\begin{align*}
T_2^{-1}\exp(-\exp((\Mp)^{1222}))\leq\ &\Big|\int\limits_{B_0(1)}\omega_{i}(x_4-r_4z,0)\varphi(z)\, dz\Big|\\
\leq\ &r_4^{-1}\Big|\int\limits_{B_0(1)}u(x_4-r_{4}z,0)\nabla\times\varphi(z)\, dz\Big|\\
\leq\ &r_4^{-2}\|u\|_{L^3(B_{0}((\Mp)^{1200}R_{2}')\setminus B_{0}(R_{2}'))}\|\nabla\times\varphi\|_{L^{\frac32}(B_{0}(1))}
\end{align*}
for a fixed non-negative $\varphi\in C^\infty_c(B_0(1))$. Recalling \eqref{e.sepT2ts}-\eqref{e.bdR2primets} we conclude that, 
\begin{equation}\label{e.lowerL3ts}
\int\limits_{B_0\big(\exp((\Mp)^{1223})(-t_{(k)})^\frac12\big)\setminus B_0\big((-t_{(k)})^\frac12\big)}|u(x,0)|^3\, dx\geq \exp\big(-\exp((\Mp)^{1223})\big),
\end{equation}
for all $k\in \{1,\ldots j\}$. Note that \eqref{wellsepproptsrecap} implies that for distinct $k$ the spatial annuli in \eqref{e.lowerL3ts} are disjoint.
Summing \eqref{e.lowerL3ts} over such $k$
we obtain that 
\begin{align*}
&\exp\big(-\exp({(\Mp)^{1223}})\big)j\\
\leq\ &\int\limits_{B_0\big(\exp((\Mp)^{1223})(-t_{1})^{\frac{1}{2}}\big)\setminus B_0\big((-t_{(j)})^\frac12\big)}|u(x,0)|^3\, dx\\
\leq\ &\int\limits_{\R^3}|u(x,0)|^3\, dx.
\end{align*}
This gives $$j\leq \exp\big(\exp({(\Mp)^{1223}})\big) \int\limits_{\mathbb{R}^3}|u(x,0)|^3\, dx\leq \exp\big(\exp({(\Mp)^{1224}})\big). $$
This concludes the proof of Proposition \ref{prop.maints}.

\section{Further applications}
\label{sec.further}

\subsection{Effective regularity criteria based on the local smallness of the $L^{3,\infty}$ at blow-up time}

\begin{proposition}\label{prop.type1limsup}
For all $M\in[1,\infty)$ sufficiently large 
the following result holds true. 
Consider a suitable finite-energy solutions $(u,p)$ to the Navier-Stokes equations on $\R^3\times[-1,0]$ 
that satisfies the following Type I bound
\begin{equation*}
\|u\|_{L^\infty_tL^{3,\infty}_x(\R^3\times(-1,0))}\leq M.
\end{equation*}
Assume 
\begin{equation}\label{e.conctype1limsup}
\limsup_{r\rightarrow 0}\|u(\cdot,T^*)\|_{L^{3,\infty}(B_0(r))}\leq\exp(-\exp(M^{1023})).
\end{equation}
Then, $(0,T^*)$ is a regular point.
\end{proposition}

\begin{proof}[Proof of Proposition \ref{prop.type1limsup}]
We argue by contradiction and assume $(0,T^*)$ is a singular point. The proof relies on two ingredients: (i) the concentration of the enstrophy norm near a Type I singularity, see Remark \ref{rem.concenstro}, (ii) the transfer of concentration at backward times to a lower bound at final time in Section \ref{subsec.quantest}. Contrary to the proof of Proposition \ref{prop.main} no summing of scales argument is required. 

Without loss of generality, we assume that $u$ solves Navier-Stokes on $\R^3\times(-1,0)$, that $(0,0)$ is a singular point of $u$
 and that it satisfies the Type I bound $\|u\|_{L^\infty_tL^{3,\infty}_x(\R^3\times(-1,0))}\leq M$. First note that by Lebesgue interpolation (see Lemma 2.2 in \cite{mccormick2013generalised} for example) we have that any suitable finite-energy solution with Type I bound is a mild solution on $\mathbb{R}^3\times [-1,0]$ with 
\begin{equation}\label{quantsingpointsL4} 
u\in L^{4}_{x,t}(\mathbb{R}^3\times (-1,0)).
\end{equation}
 
   By Remark \ref{rem.concenstro} and following Step 1-3 in Section \ref{subsec.quantest}, see in particular footnote \ref{foot.ae}, we can prove that  
\begin{align*}
T_2^{-\frac12}\exp\big(-\exp({M^{1021}})\big)\lesssim\ &\int\limits_{B_0(\exp(M^{1021})(T_2)^\frac12)\setminus B_0((T_2)^\frac12)}|\omega(x,0)|^2\, dx,
\end{align*}
for all $0< T_2\leq 1$ and $M$ sufficiently large. Here we used that $u\in L^{4}_{x,t}(\mathbb{R}^3\times (-1,0)\cap L^{\infty}_{t}L^{3,\infty}(\mathbb{R}^3\times (-1,0))$, which allows an application of Corollary \ref{cor.12} and Lemma \ref{epochTypeI} in the course of following Steps 1-3.

Let $r\in (0,1]$. Define $T_2:=r^2\exp(-2M^{1023})$. Following Step 4 of Section \ref{subsec.quantest} and using Hunt's inequality in Proposition \ref{hunt} instead of H\"older's inequality, we then obtain that 
\begin{equation}\label{e.lowerL3infty}
\|u(\cdot,0)\|_{L^{3,\infty}(B_0(r))}\geq\|u(\cdot,0)\|_{L^{3,\infty}(B_0(\exp(M^{1023})(T_2)^\frac12)\setminus B_0(T_2^\frac12))}\geq 2\exp(-\exp(M^{1023})).
\end{equation}
This contradicts \eqref{e.conctype1limsup}.
\end{proof}

\subsection{Estimate for the number of singular points in a Type I scenario}

The technology developed in the present paper also enables us to give an effective bound for the number singularities in a Type I scenario. The following proposition and its corollary are effective versions of the results by Choe, Wolf and Yang \cite{CWY19} and Seregin \cite{seregin2019note}.

\begin{proposition}\label{prop.cwycrit}
Let $M\in[1,\infty)$ be sufficiently large and define
\begin{equation}\label{e.formepM}
\ep(M):=\exp(-4\exp(M^{1023})).
\end{equation}
For all suitable finite-energy solutions\footnote{For a definition of \textit{suitable finite-energy solutions} we refer to Section \ref{subsec.not} `Notations'.} $(u,p)$ to the Navier-Stokes equations on $\R^3\times[-1,0]$ 
that satisfy the following Type I bound
\begin{equation*}
\|u\|_{L^\infty_tL^{3,\infty}_x(\R^3\times(-1,0))}\leq M,
\end{equation*}
the following result holds. 

Let $x_0\in\R^3$. Assume that there exists $r\in(0,\exp(M^{1021}))$,
\begin{equation}\label{profileassumption}
\frac{1}{|B_{x_0}(r)|}\left|\left\{x\in B_{x_0}(r):\, |u(x,0)|\geq\frac{\ep(M)}{r}\right\}\right|\leq\ep(M).
\end{equation}
Then $(x_0,0)$ is a regular space-time point.
\end{proposition}

This result is a variant of Theorem 1 in \cite{CWY19} and Proposition 1.3 in \cite{seregin2019note}. 
Our contribution is to provide the explicit formula \eqref{e.formepM} for $\ep(M)$ in terms of $M$. 

\begin{corollary}\label{cor.number}
Let $T^*\in(0,\infty)$ and $M\in[1,\infty)$ be sufficiently large. Assume that $(u,p)$ is a suitable finite-energy solution to the Navier-Stokes equations on $\R^3\times[0,T^*]$ that satisfies the following Type I bound
\begin{equation*}
\|u\|_{L^\infty_tL^{3,\infty}_x(\R^3\times(0,T^*))}\leq M.
\end{equation*}
Then $u$ has at most $\exp(\exp(M^{1024}))$ blow-up points at time $T^*$.
\end{corollary}

\begin{proof}[Proof of Corollary \ref{cor.number}]
We follow here the argument of \cite{seregin2019note}. Without loss of generality we can assume that $u$ is defined on $[-1,0]$ rather than $[0,T^*]$. Let $\sigma$ denote the set of all singular points at time $0$. We take a finite collection of $p$ points
\begin{equation}
x_1,\ldots\, x_p\in\sigma.
\end{equation}
There exists $r\in(0,\exp(M^{1021}))$ such that $B_{x_i}(r)\cap B_{x_j}(r)=\emptyset$ for all $i\neq j$. Then, Proposition \ref{prop.cwycrit} implies that 
\begin{align*}
|B_0(1)|\ep(M)^4p<\ &\sum_{i=1}^p\Big(\frac{\ep(M)}{r}\Big)^3\left|\left\{x\in B_{x_i}(r):\, |u(x,0)|\geq\frac{\ep(M)}{r}\right\}\right|\\
\leq\ & \|u\|_{L^\infty_tL^{3,\infty}_x(\R^3\times(-1,0))}^3\leq M^3.
\end{align*}
This yields the result.
\end{proof}

\begin{proof}[Proof of Proposition \ref{prop.cwycrit}]
Without loss of generality we assume that $x_0=0$. As in the proof of Proposition \ref{prop.type1limsup}, we assume for contradiction that $(0,0)$ is a singular point. 
 Using verbatim reasoning as in the proof of Proposition \ref{prop.type1limsup}, we see that the outcome of Step 1-3 in Section \ref{subsec.quantest} holds, in particular estimate \eqref{e.claimstep2}, which holds for all $0<T_2\leq 1$.

Arguing as in Step 4, and using the same notation, we get that there exists $$x_4\in B_0(\frac34C(100)M^{1000}R_2')\setminus B_0(2R_2')$$ such that for $r_4:=T_2^\frac12\exp(-\exp(M^{1022}))$,
\begin{align*}
\exp(-\exp(M^{1023}))\leq\ &\int\limits_{B_{x_4}(r_4)}|u(x,0)|^3\, dx\\
\leq\ &T_2^\frac32\exp(-3\exp(M^{1022}))\sup_{B_{x_4}(r_4)}|u(x,0)|^3.
\end{align*}
Hence, there exists $x_5\in B_{x_4}(r_4)$ such that 
\begin{equation*}
|u(x_5,0)|\geq 2T_2^{-\frac12}\exp(-\tfrac13\exp(M^{1023})).
\end{equation*}
By estimate \eqref{e.linftyderivstep2} and the choice of $M$ sufficiently large, we have $\|\nabla u\|_{L^\infty(\mathcal A_2)}\leq 1$ in the good annulus. Hence, for $r_5:=T_2^\frac12\exp(-\exp(M^{1023}))$, the ball $B_{x_5}(r_5)$ is contained in $\mathcal A_2$ and 
\begin{equation}
|u(x,0)|\geq T_{2}^{-\frac12}\exp(-\exp(M^{1023}))\quad\mbox{in}\quad B_{x_5}(r_5),
\end{equation}
for all $0<T_2\leq 1$. For $r:=T_2^\frac12\exp(M^{1021})$, we have $B_{x_5}(r_5)\subset\mathcal A_2\subset B_0(r)$ and
\begin{equation}
|u(x,0)|\geq \frac{\exp(-\exp(M^{1023}))}{r}.
\end{equation}
Subsequently,
\begin{multline*}
\frac{1}{|B_{0}(r)|}\left|\left\{x\in B_{0}(r):\, |u(x,0)|\geq\frac{\exp(-\exp(M^{1023}))}{r}\right\}\right|\\
\geq\Big(\frac{r_5}{r}\Big)^3>\exp(-4\exp(M^{1023})).
\end{multline*}
This holds for $r=T_{2}^{\frac{1}{2}}\exp(M^{1021})$ and every $0<T_{2}\leq 1$, which contradicts our assumption \eqref{profileassumption} on $u(\cdot,0)$.
\end{proof}

\subsection{Effective regularity criteria based on the relative smallness of the $L^3$ norm at the final moment in time}

Here we prove an effective regularity criteria for $(u,p)$ 
a solution to the Navier-Stokes equations on $\R^3\times[-1,0]$ based on the relative smallness of $\|u(\cdot,0)\|_{L^3}$ vs. $\|u(\cdot,-1)\|_{L^3}$. A \textit{non-effective} version of this result (without explicit quantitative bounds) is in \cite[Theorem 4.1 (i)]{AlbrittonBarkerBesov2018}. 

\begin{proposition}\label{prop.relative} 
For all sufficiently large $M\in[1,\infty)$, we define 
 $\Mp$ by \eqref{e.defM'theo}. 
Let $(u,p)$  
be a suitable finite-energy 
 solution to the Navier-Stokes equations \eqref{e.nse} on $\R^3\times[-1,0]$. 
Assume  that 
\begin{equation*}
\|u(\cdot,-1)\|_{L^{3}(\R^3)}\leq M.
\end{equation*}
If 
\begin{equation*}
\|u(\cdot,0)\|_{L^3(B_{0}(\exp((M^{b})^{1221}))\setminus B_{0}(1))}\leq \exp(-\exp((M^{b})^{1223})) ,
\end{equation*}
then $(0,0)$ is a regular point.
\end{proposition}

\begin{proof}
Assume for contradiction that $(0,0)$ is a singular point.
 Since $(u,p)$ is a suitable finite-energy solution, there exists $\Sigma\subset (-1,0)$ such that $|\Sigma|=1$ and 
 \begin{itemize}
 \item $\|\nabla u(\cdot,t')\|_{L^{2}(\mathbb{R}^3)}<\infty$ for all $t'\in\Sigma$,
 \item $u$ satisfies the energy inequality on $[t',0]$.
 \end{itemize}
  Then, arguing in a similar way as in the proof of Lemma \ref{lem.backconctimeslice}, we show that for  any $s_{0}\in [-1,-\frac{1}{8\alpha^{201}}]\cap\Sigma$ the vorticity concentrates in the following sense,
\begin{equation*}
\int\limits_{B_0(4(-s_{0})^\frac12 (\Mp)^{106})}|\omega(x,s_{0})|^2\, dx> \frac{(M+1)^2}{(-s_{0})^{\frac12}(\Mp)^{106}}.
\end{equation*}
Using $|\Sigma|=1$ and then following Step 1-3 of Section \ref{subsec.carlts} with one time scale, we obtain
\begin{align}\label{e.claimrelatts}
\begin{split}
T_2^{-\frac12}\exp\big(-\exp({{(\Mp)^{1221}}})\big)\lesssim\ &\int\limits_{B_0\big(\frac{3(\Mp)^{1200}}{4}R'_{2}\big)\setminus B_0(2R'_{2})}|\omega(x,0)|^2\, dx,
\end{split}
\end{align}
for
\begin{equation*}
T_{2}=1,\quad R_2:=K^\flat(T_2)^\frac12,\quad 2R_2\leq R_2'\leq 2R_2\exp(C(120)(\Mp)^{1220})
\end{equation*}
for $K^\flat$ chosen such that \eqref{e.smallKflat} holds. Reasoning as in Step 4 of Section \ref{subsec.carlts}, we then obtain 
\begin{equation*}
\int\limits_{B_0(\exp((\Mp)^{1221}))\setminus B_{0}(1)}|u(x,0)|^3\, dx
\geq 2\exp\big(-\exp((\Mp)^{1223})\big).
\end{equation*}
This concludes the proof.
\end{proof}

\section{Main tool 1: local-in-space short-time smoothing}
\label{sec.4}
The role of the next result is central in our paper.

\begin{theorem}[local-in-space short-time smoothing]\label{theo.locshortime}
There exists three universal constants $C_*,\, M_5,\, N_1\in[1,\infty)$. For all $M\geq M_5$, $N\geq N_1$, there exists a time $S_*(M,N)\in(0,\frac14]$ such that the following holds. Consider an initial data $u_0$ satisfying the global control
\begin{align*}
\|u_0\|_{L^2_{uloc}(\R^3)}\leq M,\quad\|u_0\|_{L^2(B_{\bar x}(1))}\stackrel{|\bar x|\rightarrow\infty}{\longrightarrow}0,
\end{align*}
and, in addition, $u_0\in L^6(B_0(2))$ with
\begin{equation*}
\|u_0\|_{L^6(B_0(2))}\leq N.
\end{equation*}
Then, for any global-in-time added `global-in-time' local energy solution\footnote{We recall that the definition of a `local energy solution' is given in footnote \ref{footdefsol}.}
 $(u,p)$ to \eqref{e.nse}  with initial data $u_0$ 
we have the estimate 
\begin{align}\label{e.conclbddu}
\|u\|_{L^\infty(B_0(\frac12)\times(\frac34S_*,S_*))}\leq C_*M^{8}N^{19},\\
\|\nabla u\|_{L^\infty_tL^2_x(B_0(\frac14)\times(\frac{15}{16}S_*,S_*))}\leq C_*M^{40}N^{98}.\label{e.conclbddnablau}
\end{align}

Moreover, there is an explicit formula for $S_*$, see \eqref{e.defS*}, and $S_*(M,N)=O(1)M^{-30}N^{-70}$.
\end{theorem}
\begin{remark}\label{gradientboundongeneraltimeint}
As a conclusion to the hypothesis in the above Theorem, one can also obtain general version of \eqref{e.conclbddnablau}. Specifically, for a local energy solution  with $\beta\in (0,S_{*})$, we get
\begin{multline}\label{nablalocgeneral}
\|\nabla u\|_{L^\infty_tL^2_x(B_0(\frac16)\times(\frac{255}{256}\beta,\beta))}\\
\leq C_*\beta^{-\frac{3}{4}}\Big(MN^2+MN\beta^{-\frac{1}{4}}+M\beta^{-\frac{1}{2}}+M^2+(N^2+N\beta^{-\frac{1}{4}})^2\Big).
\end{multline}
We will require this more general estimate. The computations producing it as identical to those used to show Theorem \ref{theo.locshortime} and hence are omitted.
\end{remark}

\begin{corollary}\label{corL_inftyconc}
There exists three universal constants $C_{**},\, M_6\in[1,\infty)$. For all $M\geq M_6$ there exists a time $S_{**}(M)\in(0,\frac14]$ with $S_{**}(M)=O(1)M^{-100}$ (given explicitly by \eqref{e.defbarS*}) such that the following holds. 
Suppose $(u,p)$ 
 is a `smooth solution with sufficient decay'\footnote{See footnote \ref{footdefsol}.} on $\mathbb{R}^3\times [0,T']$ for any $T'\in (0,T)$ and satisfies
\begin{equation}\label{uTypeILinftyconc}
\|u\|_{L^{\infty}_{t}L^{3,\infty}_{x}(\mathbb{R}^3\times [0,T])}\leq M.
\end{equation}
Furthermore, suppose there exists $t\in (0,T)$ such that
\begin{equation}\label{Linftyconcconverse}
\|u(\cdot, t)\|_{L^{\infty}(B_{0}(2\sqrt{S_{**}}^{-1}(T-t)^{\frac{1}{2}})}\leq \frac{M\sqrt{S_{**}}}{(T-t)^{\frac{1}{2}}}.
\end{equation}
Then we conclude that
\begin{equation}\label{localsmoothingLinfinity}
\|u\|_{L^{\infty}(B_{0}(\frac{1}{2}\sqrt{S_{**}}^{-1}(T-t)^{\frac{1}{2}})\times (t+\frac{3}{4}(T-t),T))}\leq \frac{C_{**}M^{27}\sqrt{S_{**}}}{(T-t)^{\frac{1}{2}}}.
\end{equation}
\end{corollary}
\begin{proof}
We define $S_{**}\in (0,\frac14]$ in the following way:
\begin{equation}\label{e.defbarS*}
S_{**}=S_{**}(M):=S_*\big(C_{weak}M,|B_{0}(2)|^{\frac{1}{6}}M\big),
\end{equation}
where $S_*$ is the constant defined in Theorem \ref{theo.locshortime} (see also the formula \eqref{e.defS*}).
Define $$r:= \sqrt{S_{**}}^{-1}(T-t)^{\frac{1}{2}}$$
and rescale 
\begin{equation}\label{rescaleLinfinityconc}
U(y,s):= ru(ry, r^2 s+t)\,\,\,\textrm{for}\,\,\,\,(y,s)\in\mathbb{R}^3\times (0,S_{**}).
\end{equation}
Then assumptions \eqref{uTypeILinftyconc}-\eqref{Linftyconcconverse} imply that
$$
\|U(\cdot,0)\|_{L^{\infty}(B_{0}(2))}\leq M\quad\mbox{and}\quad
\|U\|_{L^{\infty}_{t}L^{3,\infty}_{x}(\mathbb{R}^3\times (0,S_{**})}\leq M.$$
Hence we have,
\begin{equation}\label{rescaleLinfinityininitial}
\|U(\cdot,0)\|_{L^{6}(B_{0}(2))}\leq M|B_{0}(2)|^{\frac{1}{6}}
\end{equation}
and
\begin{equation}\label{rescaleTypeI}
\|U(\cdot,0)\|_{L^{2}_{uloc}(\mathbb{R}^3)}\leq C_{weak}M.
\end{equation}
Here, $C_{weak}\in[1,\infty)$ is a universal constant from the embedding $L^{3,\infty}(\mathbb{R}^3)\subset L^{2}_{uloc}(\mathbb{R}^3)$. 
We then apply Theorem \ref{theo.locshortime} to $U$ and then rescale according to \eqref{rescaleLinfinityconc}. This gives \eqref{localsmoothingLinfinity} as desired.
\end{proof}
Theorem \ref{theo.locshortime} is proven in Section \ref{sec.theo.locshortime} below. It relies on an $\ep$-regularity result for suitable weak solutions\footnote{For a definition of suitable weak solutions for \eqref{e.pertnse}, we refer to \cite[Definition 1]{BP18}.} to the perturbed Navier-Stokes equations
\begin{equation}\label{e.pertnse}
\partial_tv-\Delta v+\nabla q=-v\cdot\nabla v-a\cdot\nabla v-\nabla\cdot(a\otimes v),\quad\nabla\cdot v=0,\quad \nabla\cdot a=0
\end{equation} 
around a subcritical drift $a\in L^m(Q_{(0,0)}(1))$, $m>5$. We recover the result of Jia and \v{S}ver\'{a}k {\cite[Theorem 2.2]{JS14}} by a Caffarelli, Kohn and Nirenberg scheme \cite{CKN82} already used in \cite{BP18} for critical drifts. We also point out here that local-in-space short-time regularity estimates near locally critical initial data were recently proved in \cite{KMT20} using compactness arguments. Contrary to the critical case, here we can prove boundedness directly.

\begin{theorem}[epsilon-regularity around a subcritical drift]\label{theo.epreg}
There exists $C_{***}\in(0,\infty)$, 
for all $m\in(5,\infty]$, there exists $\ep_*(m)\in(0,\infty)$ such that the following holds for all $\ep\in(0,\ep_*(m))$. Take any $a\in L^m(Q_{(0,0)}(1))$ and any suitable weak solution $(v,q)$ to \eqref{e.pertnse} satisfying 
\begin{equation}\label{e.locenapriori}
\sup_{-1<s<0}\int\limits_{B_0(1)}|v(x,s)|^2dx+\int\limits_{Q_{(0,0)}(1)}|\nabla v|^2\, dxds\leq \ep^\frac59.
\end{equation}
Assume that 
\begin{align}
\|a\|_{L^m(Q_{(0,0)}(1))}\leq\ &\ep^\frac19,\label{e.smalldrift}\\
\int\limits_{Q_{(0,0)}(1)}|v|^3+|q|^\frac32\, dxds\leq\ &\ep.\label{e.hypepreg}
\end{align}
Then,
\begin{equation}\label{e.linftyepreg}
\sup_{(\bar x,t)\in\overline{Q_{(0,0)}(\frac12)}}\sup_{r\in (0,\frac14]}\intbar_{Q_{(\bar x,t)}(r)}|v|^3\, dxds\leq C_{***}\ep^\frac23.
\end{equation}
\end{theorem}

This theorem is proved in Section \ref{sec.theo.epreg} below. 
Notice that the smallness on the large-scale quantity \eqref{e.hypepreg} in $Q_{(0,0)}(1)$ is transferred to the $L^\infty$ bound \eqref{e.linftyepregbis}. In the following statement, we remove the smallness assumption \eqref{e.smalldrift} on the drift.

\begin{corollary}\label{cor.epreg}
Let $m\in (5,\infty]$. Let $C_{***}$ and $\ep_*$ be given by Theorem \ref{theo.epreg}. For all $\ep\in(0,\min(\ep_*,2^{-9}))$, for all $N\in[1,\infty)$, for all $a\in L^m(Q_{(0,0)}(1))$ and any suitable weak solution $(v,q)$ to \eqref{e.pertnse} satisfying 
\begin{equation}
\sup_{-1<s<0}\int\limits_{B_0(1)}|v(x,s)|^2dx+\int\limits_{Q_{(0,0)}(1)}|\nabla v|^2\, dxds\leq N^{-\frac1{1-\frac5m}} \ep^{{\frac{1}9}\cdot\frac{6-\frac5m}{1-\frac5m}}.\label{e.locenboundbis}
\end{equation}
Assume that 
\begin{align}
\|a\|_{L^m(Q_{(0,0)}(1))}\leq\ &N,\label{e.smalldriftbis}\\
\int\limits_{Q_{(0,0)}(1)}|v|^3+|q|^\frac32\, dxds\leq\ &N^{-\frac2{1-\frac5m}} \ep^{{\frac{1}9}\cdot\frac{11-\frac{45}m}{1-\frac5m}}.\label{e.hypepregbis}
\end{align}
Then,
\begin{equation}\label{e.linftyepregbis}
\|v\|_{L^\infty(Q_{(0,0)}(\frac12))}\leq C_{***}^\frac13N^{\frac1{1-\frac5m}}\ep^{{\frac{1}9}\cdot\frac{1-\frac{10}m}{1-\frac5m}}.
\end{equation}
\end{corollary}

\begin{proof}[Proof of Corollary \ref{cor.epreg}]
We use a scaling argument as in \cite[Theorem 2.2]{JS14}. Let $(x_0,t_0)\in Q_{(0,0)}(\frac12)$ and define $R_0\in(0,\infty)$ as
\begin{equation}\label{e.choiceR0}
R_0=N^{-\frac1{1-\frac5m}}\ep^{\frac{1}9\cdot\frac{1}{1-\frac5m}}.
\end{equation}
Notice that due to $0<\ep<\min(\ep_*,2^{-9})$, we have $R_0<\frac12$, so that the following rescaling is well defined: for all $(y,s)\in Q_{(0,0)}(1)$,
\begin{equation*}
V(y,s):=R_0 v(x_0+R_0y,t_0+R_0^2s),\qquad Q(y,s):=R_0^2 q(x_0+R_0y,t_0+R_0^2s).
\end{equation*}
Then $(V,Q)$ is a suitable weak solution to \eqref{e.pertnse} with a drift $b$ defined by
\begin{equation*}
b(y,s):=R_0 a(x_0+R_0y,t_0+R_0^2s).
\end{equation*}
We have by our choice of $R_0$ in \eqref{e.choiceR0}
\begin{align*}
&\|b\|_{L^m(Q_{(0,0)}(1))}\leq R_0^{1-\frac5m}\|a\|_{L^m(Q_{(x_0,t_0)}(R_0))}\leq R_0^{1-\frac5m}\|a\|_{L^m(Q_{(0,0)}(1))}\leq R_0^{1-\frac5m}N\leq\ep^\frac{1}9
\end{align*}
for the drift,
\begin{align*}
&\sup_{-1<s<0}\int\limits_{B_0(1)}|V(y,s)|^2\, dy+\int\limits_{Q_{(0,0)}(1)}|\nabla V|^2\, dyds\\
&\leq R_0^{-1}\left(\sup_{t_0-R_0^2<s<t_0}\int\limits_{B_{x_0}(R_0)}|v(x,s)|^2\, dx+\int\limits_{Q_{(x_0,t_0)}(R_0)}|\nabla v|^2\, dxds\right)\\
&\leq R_0^{-1}\left(\sup_{-1<s<0}\int\limits_{B_{0}(1)}|v(x,s)|^2\, dx+\int\limits_{Q_{(0,0)}(1)}|\nabla v|^2\, dxds\right)\\
&\leq R_0^{-1}N^{-\frac1{1-\frac5m}} \ep^{{\frac19}\cdot\frac{6-\frac5m}{1-\frac5m}}\leq \ep^\frac59
\end{align*}
for the local energy and finally
\begin{align*}
&\int\limits_{Q_{(0,0)}(1)}|V|^3+|Q|^\frac32\, dyds\\
&\leq R_0^{-2}\int\limits_{Q_{(x_0,t_0)}(R_0)}|v|^3+|q|^\frac32\, dxds\\
&\leq R_0^{-2}\int\limits_{Q_{(0,0)}(1)}|v|^3+|q|^\frac32\, dxds\\
&\leq R_0^{-2}N^{-\frac2{1-\frac5m}} \ep^{{\frac19}\cdot\frac{11-\frac{45}m}{1-\frac5m}}= \ep.
\end{align*}
Therefore, \eqref{e.locenapriori}, \eqref{e.smalldrift} and \eqref{e.hypepreg} are satisfied for $(V,Q)$, and hence,
\begin{equation*}
\sup_{(\bar x,t)\in\overline{Q_{(0,0)}(\frac12)}}\sup_{r\in (0,\frac14]}\intbar_{Q_{(\bar x,t)}(r)}|V|^3\, dyds\leq C_{***}\ep^\frac23.
\end{equation*}
Rescaling, this gives
\begin{equation*}
\sup_{(\bar x,t)\in\overline{Q_{(x_0,t_0)}(\frac{R_0}2)}}\sup_{r\in (0,\frac{R_0}4]}\intbar_{Q_{(\bar x,t)}(r)}|v|^3\, dxds\leq C_{***}\ep^\frac23R_0^{-3},
\end{equation*}
hence the bound \eqref{e.linftyepregbis} by taking the supremum over $(x_0,t_0)\in Q_{(0,0)}(\frac12)$. This concludes the proof.
\end{proof}

\subsection{Sketch of the proof of Theorem \ref{theo.epreg}}
\label{sec.theo.epreg}

The proof follows almost verbatim the one of Theorem 3 in \cite{BP18}, provided the following modifications are made. We propagate the following two bounds: for $r_k=2^{-k}$,
\begin{align*}
\tag{$A_k$} \frac{1}{r_k^2}\int\limits_{Q_{(\bar x,t)}(r_k)}|v(x,s)|^3\, dxds+\frac{1}{r_k^{\frac12+\kappa}}\int\limits_{Q_{(\bar x,t)}(r_k)}|q-(q)_{r_k}(s)|^\frac32\, dxds\leq\ & \ep^\frac23r_k^3,\\
\tag{$B_k$} \sup_{t-r_k<s<t}\int\limits_{B_{\bar x}(r_k)}|v(x,s)|^2\, dx+\int\limits_{Q_{(\bar x,t)}(r_k)}|\nabla v|^2\, dxds\leq\ & C_B\ep^\frac23r_k^3,
\end{align*}
for a universal constant $C_B\in (0,\infty)$ chosen sufficiently large and $\kappa(m)\in (0,\infty)$ such that
\begin{equation*}
0<\kappa<\min\Big(2,3-\frac{15}m\Big).
\end{equation*} 
One takes advantage of the subcriticality of the drift $a$ in the following way:
\begin{equation*}
\|a\|_{L^5(Q_{(\bar x,t)}(r_k))}\lesssim r_k^{1-\frac5m}\|a\|_{L^m(Q_{(\bar x,t)}(r_k))}\lesssim \ep^\frac19r_k^{1-\frac5m}.
\end{equation*}
This plays a key role in the estimate of $I_4$ and $I_5$ in Step 3, $J_2$ and $J_4$ in Step 4, using the same notations as in the proof of \cite[Theorem 3]{BP18}. The restriction $\kappa<2$ comes from handling $J_5$ and $J_6$, while the restriction $\kappa<3-\frac{15}m$ comes from bounding $J_2$ and $J_4$.

\subsection{Proof of Theorem \ref{theo.locshortime}}
\label{sec.theo.locshortime}

We fix $n=6$ and $m=\frac{5n}3=10$ in this proof. Let $C_{***}$ and $\ep_*=\ep_*(10)$ be given by Theorem \ref{theo.epreg}. Let also $k_0=k_0(6)$ and $K_0=K_0(6)$ be given by Proposition \ref{prop.mild}.

Let $M,\, N\in[1,\infty)$. 
Let $u_0\in L^2_{uloc}(\R^3)$ such that $\|u_0\|_{L^2(B_{\bar x}(1))}\stackrel{|\bar x|\rightarrow\infty}{\longrightarrow}0$. We assume in addition that $u_0\in L^6(B_0(2))$. Moreover, 
\begin{align*}
\|u_0\|_{L^2_{uloc}(\R^3)}\leq M,\quad \|u_0\|_{L^6(B_0(2))}\leq N.
\end{align*}
Let $u$ be any local energy solution to \eqref{e.nse} with such a data $u_0$. The goal is to prove the local-in-space short-time smoothing for $u$ stated in Theorem \ref{theo.locshortime}.

\noindent{\bf Step 1: decomposition of the initial data.}

\begin{lemma}\label{lem.bogo}
Let $u_0\in L^2_{uloc}(\R^3)$ with, in addition, $u_0|_{B_0(2)}\in L^6(\R^3)$. Then, there exists a universal constant $K_2\in[1,\infty)$, there exists $u_{0,a}\in L^6_\sigma(\R^3)\cap L^2_\sigma(\R^3)$, $\supp(u_{0,a})\subset B_0(2)$, and $u_{0,b}\in L^2_{uloc}(\R^3)$ such that the following holds:
\begin{align*}
u_0=u_{0,a}+u_{0,b},\quad u_{0,a}=u_0\ \mbox{on}\ B_0(\tfrac32),\quad \|u_{0,a}\|_{L^6}\leq K_2\|u_0\|_{L^6(B_0(2))},\\
\|u_{0,a}\|_{L^2}\leq K_2\|u_0\|_{L^2(B_0(2))}\quad\mbox{and}\quad \|u_{0,b}\|_{L^2_{uloc}}\leq K_2\|u_0\|_{L^2_{uloc}}.
\end{align*}
\end{lemma}

\begin{proof}
The proof is standard using Bogovskii's operator \cite[Chapter III.3]{Galdibook}. We refer to \cite{BP18} for a detailed proof.
\end{proof}

\noindent{\bf Step 2: control of the local energy of the perturbation.} We use the decomposition given by Lemma \ref{lem.bogo} for $u_0$ as above. Let $a$ be the mild solution given by Proposition \ref{prop.mild} associated to the data $u_{0,a}\in L^6(\R^3)$. The mild solution $a$ exists at least on the time interval $(0,S_{mild}^a)$, where 
\begin{equation*}
S_{mild}^a:=k_0N^{-4}.
\end{equation*}
Moreover since $u_0\in L^2_\sigma(\R^3)$, the mild solution $a$ can be constructed to be a weak Leray-Hopf solution on $\R^3\times(0,S_{mild}^a)$ and we have the global energy control\footnote{This can be inferred from arguments similar to \cite[Section 3.1]{Barker18}.}
\begin{equation}\label{e.globalenmild}
\sup_{s\in(0,S_{mild}^a)}\int\limits_{\R^3}\frac{|a(x,s)|^2}2\, dx+\int\limits_0^{S_{mild}^a}\int\limits_{\R^3}|\nabla a|^2\, dxds\leq K_0'N^2,
\end{equation}
with $K_0'\in[1,\infty)$ a universal constant. This and Calder\'{o}n-Zygmund theory implies
\begin{equation}\label{qaest}
\|q_{a}\|_{L^{\frac 5 3}(\mathbb{R}^3\times (0, S^{a}_{mild}))}\leq K''_{0} N^{2},
\end{equation} where $q_{a}$ is the pressure associated to $a$.
Moreover, since $u$ is a local energy solution with the initial data $u_0$, Proposition \ref{prop.lews} implies the following control of the local energy
\begin{equation}\label{e.contlocenu}
\sup_{s\in(0,S_{locen}^u)}\sup_{\bar x\in\R^3}\int\limits_{B_{\bar x}(1)}\frac{|u(x,s)|^2}{2}\, dx+\sup_{\bar x\in\R^3}\int\limits_0^{S_{locen}^u}\int\limits_{B_{\bar x}(1)}|\nabla u(x,s)|^2\, dx\, ds\leq K_1M^2,
\end{equation}
where $S_{locen}^{u}(N):=k_1\min(M^{-4},1)$. As a consequence, the perturbation $v=u-a$ is a local energy solution to \eqref{e.pertnse} 
\begin{align}\label{e.contlocen}
\begin{split}
&\sup_{s\in(0,S^v)}\sup_{\bar x\in\R^3}\int\limits_{B_{\bar x}(1)}\frac{|v(x,s)|^2}{2}\, dx+\sup_{\bar x\in\R^3}\int\limits_0^{S^v}\int\limits_{B_{\bar x}(1)}|\nabla v(x,s)|^2\, dx\, ds\\
&\leq K_1'
(N^2+M^2),
\end{split}
\end{align}
where $K_1'\in[1,\infty)$ is a universal constant and
\begin{align}
\begin{split}\label{e.defSv}
S^v=\ &S^v(M,N):=\min(\tfrac14,S^a_{mild},S_{locen}^u)\\
=\ &\min\big(\tfrac14,k_0N^{-4},
k_1M^{-4},k_1\big).
\end{split}
\end{align}
Moreover, we have the following pressure estimate
\begin{equation}\label{e.contpress}
\|q-C_0(t)\|_{L^\frac53(B_0(\frac32)\times(0,S^v))}
\leq K_1'(M^2+N^2),
\end{equation}
with a universal constant $K_1'\in[1,\infty)$. 
This bound follows from \eqref{qaest}, \eqref{e.pressuredec} and \eqref{e.pressurest}.

\noindent{\bf Step 3: smallness of the local energy in short time.} Let $\phi\in C^\infty_c(\R^3)$ be a cut-off function such that 
\begin{align}\label{e.K3}
0\leq\phi\leq 1,\quad\supp\phi\subset B_{0}(\tfrac32),\quad \phi=1\ \mbox{on}\ B_0(1)\quad\mbox{and}\quad 
|\nabla(\phi^2)|+|\Delta(\phi^2)|\leq K_3, 
\end{align}
where $K_3\in[1,\infty)$. We estimate the local energy 
\begin{equation*}
E(t):=\sup_{s\in(0,t)}\int\limits_{\R^3}|v(x,t)|^2\phi^2\, dx+2\int\limits_0^t\int\limits_{\R^3}|\nabla\phi|^2\phi^2\, dxds
\end{equation*}
for all $t\in(0,S^v)$. 
The local energy inequality gives
\begin{equation*}
E(t)\leq I_1+\ldots I_6,
\end{equation*}
with 
\begin{align*}
&I_1=\int\limits_0^t\int\limits_{\R^3}|v|^2\Delta(\phi^2)\, dxds,\quad I_2=\int\limits_0^t\int\limits_{\R^3}|v|^2v\cdot\nabla(\phi^2)\, dxds,\\
&I_3=2\int\limits_0^t\int\limits_{\R^3}(q-C_0(t))v\cdot\nabla(\phi^2)\, dxds,\quad
I_4=-2\int\limits_0^t\int\limits_{\R^3}(a\cdot\nabla v)\cdot v\phi^2\, dxds,\\
&I_5=2\int\limits_0^t\int\limits_{\R^3}(a\otimes v):\nabla v\phi^2\, dxds\quad\mbox{and}\quad I_6=2\int\limits_0^t\int\limits_{\R^3}(a\otimes v):v\otimes\nabla(\phi^2)\, dxds.
\end{align*}
Let $t\in(0,S^v)$. Let us estimate each term in the right hand side. For that purpose, we rely on the bounds \eqref{e.contlocen} for the local energy and \eqref{e.contpress} for the pressure. For the terms involving only $v$, we have using that $|B_0(\frac32)|=3^3$,
\begin{align*}
|I_1|
\leq  3^3K_3\cdot2K_1'(M^2+N^2)t,
\end{align*}
and
\begin{align*}
|I_2|
\leq\ & 3^{3+\frac3{10}}K_3{(3^3\cdot 2K_1')}^\frac32(M^2+N^2)^\frac32t^\frac1{10}\\
=\ & 2^\frac32\cdot 3^\frac{39}5K_3{K_1'}^\frac32(M^2+N^2)^\frac32t^\frac1{10}.
\end{align*}
For the terms involving $a$ and $v$ we use \eqref{e.propmild} in Proposition \ref{prop.mild}, more precisely the bound $\|a\|_{L^{10}(\R^3\times(0,S_{mild}^a))}\leq K_0N$. This in turn implies the controls
\begin{align*}
|I_4+I_5+I_6|
\leq\ &3\|a\|_{L^{10}}3^3E(t)(3^3t)^{\frac15-\frac1{10}}\\
\leq\ &2\cdot 3^{\frac{43}{10}}K_0K_{3}K_1'N(M^2+N^2)t^{\frac1{10}}.
\end{align*}
For $I_{6}$, we used $t\in (0,S^{v})\subset(0,\frac{1}{4})$.

Finally, we estimate the term involving the pressure
\begin{align*}
|I_3|
\leq\ &2K_3\|q-C_0(t)\|_{L^\frac53(B_0(\frac32))}\|v\|_{L^\frac{10}3(B_0(\frac32))}(3^3t)^\frac1{10}\\
\leq\ &2\cdot 3^{\frac95}K_3{K_1'}^\frac32(M^2+N^2)^\frac32t^\frac1{10}.
\end{align*}
Finally, we get the following estimate: there exists a universal constant $K_*\in[1,\infty)$ such that for all $t\in(0,1]$, 
\begin{equation}\label{e.estEt}
E(t)\leq K_*(M^2+N^2)^\frac32t^\frac1{10}.
\end{equation}
Notice that $K_*$ can be taken as
\begin{equation}\label{e.K*}
K_*:=4\max\big(2\cdot3^3K_3K_1',2^\frac32\cdot 3^\frac{39}5K_3{K_1'}^\frac32,2\cdot 3^{\frac{43}{10}}K_0K_1'\big),
\end{equation}
where $K_0$ is defined in Proposition \ref{prop.mild}, $K_1'$ in \eqref{e.contlocen} and \eqref{e.contpress}, and $K_3$ is the constant in \eqref{e.K3}.

\noindent{\bf Step 4: boundedness of the perturbation.} Let $\ep\in(0,\min(\ep_*,2^{-9}))$. Our objective is now to apply the $\ep$-regularity result Corollary \ref{cor.epreg} in order to get the boundedness of the perturbation. As in \cite{JS14} and \cite{BP18}, we extend $v$ by $0$ in the past. The extension $v$ is a suitable weak solution on $B_0(1)\times(-\infty,S^v)$ to the Navier-Stokes equations \eqref{e.pertnse} with a drift $a$ defined to be the zero extension of $a$ to $\R^3\times(-\infty,0)$. The bound on the local energy \eqref{e.estEt} is crucial here, as is emphasized in \cite{JS14}. Notice that the extended $a$ is not a mild solution to the Navier-Stokes system \eqref{e.nse} on $\R^3\times(-\infty,S^a_{mild})$ but in Corollary \ref{cor.epreg} this fact is not required. We have the bound
\begin{equation*}
\|a\|_{L^{10}(\R^3\times(-\infty,S^a_{mild}))}\leq K_0N.
\end{equation*}
By the control \eqref{e.contlocen} on the local energy and \eqref{e.contpress} on the pressure, we have 
\begin{align}\label{e.contepreg}
\begin{split}
\int\limits_{t-1}^t\int\limits_{B_0(1)}|v|^3+|q-C_0(t)|^\frac32\, dxds=\ &\int\limits_{0}^t\int\limits_{B_0(1)}|v|^3+|q-C_0(t)|^\frac32\, dxds\\
\leq\ &2{K_1'}^\frac32(M^2+N^2)^\frac32(2^3t)^\frac1{10}.
\end{split}
\end{align}
Therefore, in order to apply Corollary \ref{cor.epreg}, we choose $S_*=S_*(M,N)\in(0,S^v)$ sufficiently small such that
\begin{align}\label{e.condS*}
\begin{split}
K_*(M^2+N^2)^\frac32S_*^\frac1{10}\leq\ &(K_0N)^{-2}\ep^\frac{11}{9},\\
2^{\frac{13}{10}}{K_1'}^\frac32(M^2+N^2)^\frac32 S_*^\frac1{10}\leq\ &(K_0N)^{-4}\ep^\frac{13}9.
\end{split}
\end{align}
Conditions \eqref{e.condS*} imply that \eqref{e.locenboundbis} and \eqref{e.hypepregbis} are satisfied on $B_0(1)\times(S_*-1,S_*)$. According to \eqref{e.condS*}, we define $S_*=S_*(M,N)$ in the following way 
\begin{equation}\label{e.defS*bis}
S_*:=\min\left(S^v,
\frac{2^{-\frac{3}{2}}\ep^\frac{130}9}{K_0^{40}M^{30}N^{70}}\min\Big(\frac1{2^{13}{K_1'}^{15}},\frac1{K_*^{10}}\Big)\right)
\end{equation}
Notice that for $M,\, N$ sufficiently large we have 
\begin{align}\label{e.defS*}
S_*=\ &\frac{2^{-\frac{3}{2}}\ep^\frac{130}9}{K_0^{40}M^{30}N^{70}}\min\Big(\frac1{2^{13}{K_1'}^{15}},\frac1{K_*^{10}}\Big)\\
=\ &O(M^{-30}N^{-70}).\nonumber
\end{align}
For the rest of the proof we take this definition of $S_*$. 
It follows from \eqref{e.linftyepreg} that
\begin{equation}\label{e.conclbdd}
\|u-a\|_{L^\infty(Q_{(0,S_*)}(\frac12))}=\|v\|_{L^\infty(Q_{(0,S_*)}(\frac12))}\leq C_{***}^\frac13N^2.
\end{equation}
Combining this estimate with \eqref{e.propmild} enables to obtain for all $\beta\in(0,S_*)$
\begin{equation}\label{e.conclbddubis}
\|u\|_{L^\infty(B_0(\frac12)\times(\beta,S_*))}\leq C_{***}^\frac13N^2+\frac{K_0N}{\beta^{\frac14}}.
\end{equation}
which implies estimate \eqref{e.conclbddu} as a particular case. 

\noindent{\bf Step 5: estimates of the gradient of the perturbation.} In this step, we take $\beta=\frac{3S_*}4$. Our goal is to prove the following claim 
\begin{equation}\label{e.estsupbarxLinftyL2}
\sup_{\bar x\in B_0(\frac14)}\sup_{t\in(\frac{15}{16}S_*,S_*)}\int\limits_{B_{\bar x}(\frac{\sqrt{S_*}}4)}|\nabla u(x,t)|^2\, dx\leq CM^{34}N^{80},
\end{equation}
for a universal constant $C\in[1,\infty)$. Estimate \eqref{e.estsupbarxLinftyL2} implies estimate \eqref{e.conclbddnablau} in the theorem by a covering argument. 
Let $\bar x\in B_0(\frac14)$. Notice that $B_{\bar x}(\frac{\sqrt{S_*}}2)\subset B_0(\frac13)$ for $M,\, N$ sufficiently large. Without loss of generality, we assume that $\bar x=0$. We bootstrap the regularity of $u$ in the parabolic cylinder
\begin{equation*}
Q_{(0,S_*)}\big(\tfrac{\sqrt{S_*}}2\big)=B_0\big(\tfrac{\sqrt{S_*}}2\big)\times\big(\tfrac{3}{4}S_*,S_*\big)
\end{equation*}
using the local maximal regularity for the non-stationary Stokes system. 
We first zoom in on the parabolic cylinder $Q_{(0,S_*)}\big(\frac{\sqrt{S_*}}2\big)$ and define 
\begin{align*}
U(y,s):=ru(ry,r^2s+S_*),\quad P(y,s):=r^2\big(p(ry,r^2s+S_*)-C_0(r^2s+S_*)\big),
\end{align*}
for all $(y,s)\in Q_{(0,0)}(1)$, where $r:=\frac{\sqrt{S_*}}2$. By the estimate \eqref{e.conclbddu}, we have
\begin{align}\label{e.estULinfty}
\begin{split}
\|U\|_{L^\infty(Q_{(0,0)}(1))}\leq\ &\tfrac{C_*}2M^8N^{19}S_*^\frac12\\
\leq\ &CM^{-7}N^{-16}.
\end{split}
\end{align}
Moreover, the local energy estimate \eqref{e.contlocenu} implies that
\begin{align*}
\|\nabla U\|_{L^2(Q_{(0,0)}(1))}=\ &r^{-\frac12}\|\nabla u\|_{L^2(Q_{(0,S_*)}(r))}\leq 2^\frac12K_1^\frac12M S_*^{-\frac14}\leq CM^9N^{18},
\end{align*}
For the pressure, we decompose $p-C_0(t)$ according to \eqref{e.pressuredec}. We have 
according to the estimates in Proposition \ref{prop.lews} 
\begin{align*}
\|p_{nonloc}\|_{L^2(Q_{(0,S_*)}(r))}\leq\ & \|p_{nonloc}\|_{L^2(B_0(\frac12)\times (\frac34S_*,S_*))}\\
\leq\ &S_*^\frac12\|p_{nonloc}\|_{L^\infty(B_0(\frac12)\times (\frac34S_*,S_*))}\\
\leq\ &S_*^\frac12K_1M^2,
\end{align*}
on the one hand, and 
\begin{align*}
\|p_{loc}\|_{L^2(Q_{(0,S_*)}(r))}\leq\ &\|p_{loc}\|_{L^2(B_0(\frac12)\times (\frac34S_*,S_*))}\\
\leq\ &C\|u\|_{L^4(B_0(\frac12)\times (\frac34S_*,S_*))}^2\\
\leq\ &CS_*^\frac12\|u\|_{L^\infty(B_0(\frac12)\times (\frac34S_*,S_*))}^2\\
\leq\ &CS_*^\frac12M^{16}N^{39}
\end{align*}
using \eqref{plocdef} and Calder\'on-Zygmund on the other hand.
Hence,
\begin{align*}
\|P\|_{L^2(Q_{(0,0)}(1))}=\ &r^{-\frac12}\|p-C(t)\|_{L^2(Q_{(0,S_*)}(r))}\leq 2^\frac12S_*^{-\frac14}(S_*^\frac12K_1M^2+CS_*^\frac12M^{16}N^{39})\\
\leq\ &CM^9N^{22}.
\end{align*}
Notice that these are rough bounds, but they are enough for our purposes. 
Therefore,
\begin{equation}\label{e.nablaUP}
\|\nabla U\|_{L^2(Q_{(0,0)}(1))}+\|P\|_{L^2(Q_{(0,0)}(1))}\leq CM^{9}N^{22}.
\end{equation}
Using the local maximal regularity \cite[Proposition 2.4]{Ser06} leads to
\begin{align*}
&\|\partial_tU\|_{L^2(Q_{(0,0)}(\frac34))}+\|\nabla^2U\|_{L^2(Q_{(0,0)}(\frac34))}+\|\nabla P\|_{L^2(Q_{(0,0)}(\frac34))}\\
\leq\ & C\left(\|U\|_{L^\infty(Q_{(0,0)}(1))}\|\nabla U\|_{L^2(Q_{(0,0)}(1))}\right.\\
&\left.+\|U\|_{L^2(Q_{(0,0)}(1))}+\|\nabla U\|_{L^2(Q_{(0,0)}(1))}+\|P\|_{L^2(Q_{(0,0)}(1))}\right)\\
\leq\ & CM^{9}N^{22},
\end{align*}
where we used the bounds \eqref{e.estULinfty} and \eqref{e.nablaUP}. A simple local energy estimate for $\nabla U$ then leads to the bound
\begin{equation*}
\|\nabla U\|_{L^\infty_tL^2_x(Q_{(0,0)}(\frac12))}\leq CM^{9}N^{22}.
\end{equation*}
Scaling back to the original variables leads to \eqref{e.estsupbarxLinftyL2} and concludes the proof.

\section{Main tool 2: quantitative annuli of regularity}
\label{sec.quantannulus}

In this section we prove that a solution $u$, satisfying the hypothesis of Propositions \ref{prop.main}-\ref{prop.maints}, enjoys good quantitative bounds in certain spatial annuli. This was crucially used in the aforementioned propositions in two places. Namely for applying the Carleman inequalities (Propositions \ref{prop.firstcarl}-\ref{prop.sndcarl}), as well as in \textbf{`Step 4'} for transferring the lower bound of the vorticity to a lower bound on the localized $L^{3}$ norm.

In the context of classical solutions to the Navier-Stokes equations in $L^{\infty}_{t}L^{3}_{x}(\mathbb{R}^3\times (t_{0}-T,t_{0}))$, a related version was proven by  Tao in \cite{Tao19} using a delicate analysis of local enstrophies from \cite{tao2013localisation}. Our proof is somewhat different and elementary (though we use the `pidgeonhole principle' as in \cite{Tao19}), instead we utilize known $\varepsilon$-regularity criteria.

\begin{proposition}[\cite{CKN82} and {\cite[Theorem 30.1]{LR02}}]\label{CKN}
There exists absolute constants $\ep_{0}^{*}>0$ and $C_{CKN}\in(0,\infty)$ such that if $(u,p)$ is a suitable weak solution to the Navier-Stokes equations on $Q_{(0,0)}(1)$ and for some $\ep_{0}\leq \ep_{0}^{*}$ 
\begin{equation}\label{CKNsmallness}
\int\limits_{Q_{(0,0)}(1)} |u|^3+|p|^{\frac{3}{2}} dxdt\leq \ep_{0}
\end{equation}
then  one concludes that
\begin{equation}\label{CKNboundedquant}
u\in L^{\infty}(Q_{(0,0)}({1}/{2}))\,\,\,\,\textrm{with}\,\,\,\,\|u\|_{L^{\infty}(Q_{(0,0)}(\frac{1}{2}))}\leq C_{CKN} \ep_{0}^{\frac{1}{3}}.
\end{equation}
\end{proposition}

We require the following proposition, which is a quantitative version of  Serrin's result \cite{Serrin1962}. Since the procedure is the same as that described in \cite{Serrin1962}, we omit the proof.

\begin{proposition}\label{quantitativeSerrin}
Suppose $u\in L^{\infty}(Q_{(0,0)}(1/2))$ and $\omega:=\nabla\times u\in L^{2}(Q_{(0,0)}(1/2))$ is such that $(u,p)$ is a distributional solution to the Navier-Stokes equations in $Q_{(0,0)}(1/2)$.
Furthermore, suppose
\begin{equation}\label{serrinsmallnessu}
\|u\|_{L^{\infty}(Q_{(0,0)}(1/2))}<1
\end{equation}
and
\begin{equation}\label{serrinsmallnessvort}
\|\omega\|_{L^{2}(Q_{(0,0)}(1/2))}<1.
\end{equation}
There exists universal constants $C_{k}'\in(0,\infty)$ with $k=0,1$, such that the above assumptions imply that for $k=0,1$
\begin{equation}\label{quantserrinest}
\|\nabla^{k}\omega\|_{L^{\infty}(Q_{(0,0)}(1/3))}\leq C_{k}'(\|u\|_{L^{\infty}(Q_{(0,0)}(1/2))}+\|\omega\|_{L^{2}(Q_{(0,0)}(1/2))}).
\end{equation}

\end{proposition}

\begin{remark}\label{vortimprovesws}
If we instead use the framework of \textit{suitable weak solutions} in the above proposition, we can use the time integrability of the pressure to gain space-time H\"{o}lder continuity of all spatial derivatives of $u$ in $Q_{(0,0)}(1/3)$. The vorticity equation then implies   $w,\ \partial_tw,\ \nabla w$ and $\nabla^2w$ are continuous in space and time in $Q_{(0,0)}(1/3)$. 
\end{remark}

\begin{proposition}\label{CKNquanthigher}
There exists absolute constants $\ep_{1}^{*}\in (0,1)$ and $C_{k}''\in(0,\infty)$, $k=0, 1, 2$, such that if $(u,p)$ is a suitable weak solution to the Navier-Stokes equations on $Q_{(0,0)}(1)$ and for some $\ep_{1}\leq \ep_{1}^{*}$ 
\begin{equation}\label{CKNsmallnessrepeat}
\int\limits_{Q_{(0,0)}(1)} |u|^3+|p|^{\frac{3}{2}} dxdt\leq \ep_{1}
\end{equation}
then  one concludes that for $j=0,1$
\begin{equation}\label{CKNboundedquanthigher}
\nabla^{j}u\in L^{\infty}(Q_{(0,0)}({1}/{4}))\,\,\,\,\textrm{with}\,\,\,\,\|\nabla^j u\|_{L^{\infty}(Q_{(0,0)}(\frac{1}{4}))}\leq C_{j}'' \ep_{1}^{\frac{1}{3}}
\end{equation}
and 
\begin{equation}\label{CKNgradboundedvortquant}
\nabla\omega\in L^{\infty}(Q_{(0,0)}({1}/{4}))\,\,\,\,\textrm{with}\,\,\,\,\|\nabla \omega\|_{L^{\infty}(Q_{(0,0)}(\frac{1}{4}))}\leq C_{2}'' \epsilon_{1}^{\frac{1}{3}}.
\end{equation}
\end{proposition}
\begin{proof}
Let $\epsilon_{1}^{*}\in (0, \min(1, \epsilon_{0}^{*}))$. Notice that $\epsilon^{*}_{1}$ still to be determined. Then for $\ep_{1}\leq \ep_{1}^{*}$ we can apply Proposition \ref{CKN} to get
\begin{equation}\label{CKNconsequence}
\|u\|_{L^{\infty}(Q_{(0,0)}(\frac{1}{2}))}\leq C_{CKN} \ep_{1}^{\frac{1}{3}}\leq C_{CKN} (\ep_{1}^{*})^{\frac{1}{3}}
\end{equation}
The local energy inequality for suitable weak solutions to the Navier-Stokes equations implies that
$$\int\limits_{Q_{(0,0)}(\frac{1}{2})} |\nabla u|^2 dxdt\leq C\Big(\int\limits_{Q(1)} |u|^3 dxdt\Big)^{\frac{2}{3}}+ C\int\limits_{Q(1)} |u|^3+|p|^{\frac{3}{2}} dxdt.$$
Thus using \eqref{CKNsmallnessrepeat} and the fact that $\ep_{1}^{*}<1$, we get that
\begin{equation}\label{CKNvortconsequence}
\int\limits_{Q_{(0,0)}(\frac{1}{2})} |\omega|^2 dxdt \leq C_{univ}\ep_{1}^{\frac{2}{3}}\leq  C_{univ}(\ep_{1}^{*})^{\frac{2}{3}} .
\end{equation}
So defining
\begin{equation}\label{epsilon0**def}
\ep_{1}^{*}:=\min\Big(\epsilon_{0}^{*},1,(2C_{univ})^{-\frac{3}{2}}, (2C_{CKN})^{-3}\Big),
\end{equation}
we get that \eqref{CKNconsequence}-\eqref{CKNvortconsequence} imply
\begin{equation}\label{uandvortlessthan1}
\|u\|_{L^{\infty}(Q_{(0,0)}(1/2))}<1\,\,\,\,\,\textrm{and}\,\,\,\,\,\,\|\omega\|_{L^{2}(Q_{(0,0)}(1/2))}<1.
\end{equation}
Applying Proposition \ref{quantitativeSerrin}, together with \eqref{CKNconsequence}-\eqref{CKNvortconsequence}, gives that for $k=0,1$
\begin{equation}\label{vorthighersmall}
\|\nabla^k\omega\|_{L^{\infty}(Q(\frac{1}{3}))}\leq C_{k}'(\|u\|_{L^{\infty}(Q_{(0,0)}(\frac{1}{2})}+\|\omega\|_{L^{2}(Q_{(0,0)}(\frac{1}{2})})\leq {C}_{k}''\ep_1^{\frac{1}{3}}.
\end{equation}
Using $-\Delta u= \nabla\times \omega$, known elliptic theory, \eqref{CKNconsequence} and \eqref{vorthighersmall}   gives 
\begin{align*}
\|\nabla u\|_{L^{\infty}(Q_{(0,0)}(\frac{1}{4}))}\leq\ & C(\|u\|_{L^{\infty}(Q_{(0,0)}(\frac{1}{3}))}+\|\omega\|_{L^{\infty}(Q_{(0,0)}(\frac{1}{3}))}+\|\nabla\omega\|_{L^{\infty}(Q_{(0,0)}(\frac{1}{3}))})\\
\leq\ & C_{1}''\ep_{1}^{\frac{1}{3}}. \qedhere
\end{align*}
\end{proof}
\begin{proposition}[annulus of regularity, general form]\label{goodannulusgneral}
For all $\mu>0$ there exists $\lambda_{0}(\mu)>1$ such that the following holds true. For all $\lambda\geq \lambda_{0}(\mu)$, $R\geq 1$ and for every solution $(u,p)$ to the Navier-Stokes equations on $\mathbb{R}^3\times [-1,0]$ that is a 
suitable weak solution on $Q_{(x_*,0)}(1)$ for any $x_*\in\R^3$ and satisfies
\begin{equation}\label{globalbound}
\int\limits_{-1}^{0}\int\limits_{\mathbb{R}^3} |u|^{\frac{10}{3}}+|p|^{\frac{5}{3}} dxdt\leq\lambda<\infty,
\end{equation}
there exists $R''(u,p,\lambda,\mu,R)$ with
\begin{equation}\label{annulusradiusbound}
2R\leq R''\leq 2R\exp{(2\mu \lambda^{\mu+2})}
\end{equation}
and universal constants $\bar{C}_{j}\in(0,\infty)$ for $j=0, 1, 2$ such that for $j=0,1$
\begin{equation}\label{nablagoodannulusbound}
\|\nabla^j u\|_{L^{\infty}(\{R''<|x|<\frac{\lambda^{\mu}}{4}R''\}\times (-\frac{1}{16},0))}\leq \bar{C}_{j} \lambda^{-\frac{3\mu}{10}}
\end{equation}
and
\begin{equation}\label{vortgoodannulusbound}
\|\nabla \omega\|_{L^{\infty}(\{R''<|x|<\frac{\lambda^{\mu}}{4}R''\}\times (-\frac{1}{16},0))}\leq \bar{C}_{2} \lambda^{-\frac{3\mu}{10}}.
\end{equation}
\end{proposition}
\begin{proof}
Fix any $R\geq 1$ and $\mu>1$. With these choices, take $\lambda>\lambda_{0}(\mu)\geq 1$. Here, $\lambda_{0}$ is to be determined.
Then
$$\sum_{k=0}^{\infty} \int\limits_{-1}^{0}\int\limits_{(\lambda^\mu)^{k}R<|x|<(\lambda^\mu)^{k+1}R} |u|^{\frac{10}{3}}+|p|^{\frac{5}{3}} dxdt\leq \lambda. $$
by the pigeonhole principle, there exists $k_{0}\in \{0,1,\ldots\lceil{\lambda^{\mu+1}\rceil}\}$ such that
$$\int\limits_{-1}^{0}\int\limits_{(\lambda^\mu)^{k_{0}}R<|x|<(\lambda^\mu)^{k_{0}+1}R} |u|^{\frac{10}{3}}+|p|^{\frac{5}{3}} dxdt\leq \lambda^{-\mu}. $$
Define $R':= R\lambda^{\mu k_{0}}.$ Then
\begin{equation}\label{R'bound}
R\leq R'\leq R\exp(2\mu \lambda^{\mu+2})
\end{equation}
and
\begin{equation}\label{decayonannulus}
\int\limits_{-1}^{0}\int\limits_{R'<|x|<\lambda^\mu R'} |u|^{\frac{10}{3}}+|p|^{\frac{5}{3}} dxdt\leq \lambda^{-\mu}
\end{equation}
Impose the restriction $\lambda_{0}(\mu)>4^{\frac{1}{\mu}}$ and define
\begin{equation}\label{Aannulusdef}
A:=\{x: R'+1<|x|<\lambda^{\mu}R'-1\}.
\end{equation}
By H\"{o}lder's inequality we have
\begin{align*} 
\sup_{x_*\in A}\int\limits_{-1}^{0}\int\limits_{B(x_*,1)} |u|^{3}+|p|^{\frac{3}{2}} dxdt\leq\ & C_{univ}\Big(\sup_{x_*\in A}\int\limits_{-1}^{0}\int\limits_{B_{x_*}(1)} |u|^{\frac{10}{3}}+|p|^{\frac{5}{3}} dxdt\Big)^{\frac{9}{10}} 
\\
\leq\ & C_{univ}\lambda^{-\frac{9\mu}{10}}.
\end{align*}
Defining
\begin{equation}\label{lambda0def}
\lambda_{0}(\mu):=\max\Big(2\cdot4^{\frac{1}{\mu}}, \Big(\frac{2C_{univ}}{\ep_1^{*}}\Big)^{\frac{10}{9\mu}}\Big),
\end{equation}
the inequality $\lambda\geq \lambda_{0}(\mu)$ implies that
$$\sup_{x_*\in A}\int\limits_{-1}^{0}\int\limits_{B_{x_*}(1)} |u|^{3}+|p|^{\frac{3}{2}} dxdt\leq C_{univ}\lambda^{-\frac{9\mu}{10}}< \ep_1^{*}. $$
Thus, we can apply Proposition \ref{CKNquanthigher} to get that for $j=0,1$
$$\sup_{x_{*}\in A}\|\nabla^{j} u\|_{L^{\infty}Q_{(x_{*},0)}(\frac{1}{4})}\leq C_{j}''C_{univ}^{\frac{1}{3}} \lambda^{-\frac{3\mu}{10}}\,\,\,\,\textrm{and}\,\,\,\,\,\sup_{x_{*}\in A}\|\nabla \omega\|_{L^{\infty}Q_{(x_{*},0)}(\frac{1}{4})}\leq C_{2}''C_{univ}^{\frac{1}{3}} \lambda^{-\frac{3\mu}{10}}.$$
Hence, 
$$\|\nabla^{j} u\|_{L^{\infty}( A\times (-\frac{1}{16},0))}\leq C_{j}''C_{univ}^{\frac{1}{3}} \lambda^{-\frac{3\mu}{10}}\,\,\,\,\textrm{and}\,\,\,\,\,\|\nabla \omega\|_{L^{\infty}( A\times (-\frac{1}{16},0))}\leq C_{2}''C_{univ}^{\frac{1}{3}} \lambda^{-\frac{3\mu}{10}}.$$
Finally, note that \eqref{lambda0def} and $\lambda\geq \lambda_{0}$ imply that $\lambda^{\mu}R'-1>\frac{\lambda^{\mu}}{2} R'>2R'>R'+1.$
Defining $R'':= 2R'$ and using $\{x: R''<|x|< \frac{\lambda^{\mu}}{4} R''\}\subset A$, we see that the above estimates readily give the desired conclusion.
\end{proof}
Bearing in mind \eqref{semigroupL103}, the energy estimates \eqref{wkineticenergyest}, \eqref{wdissipationest} and Calder\'{o}n-Zygmund estimates for the pressure, the following Lemma is obtained as an immediate corollary to the above Proposition. We also use the known fact that mild solutions to the Navier-Stokes equations in $L^{4}_{x,t}(\mathbb{R}^3\times (0,T))$ are suitable weak solutions on $\mathbb{R}^3\times (0,T)$, which can be seen by using a mollification argument along with Calder\'{o}n-Zygmund estimates for the pressure.
 
\begin{lemma}[annulus of regularity, Type I]\label{lem.ga}
For all $\mu>0$ there exists $M_1(\mu)>1$ such that the following holds true. For all $M\geq M_1(\mu)$, $R\geq 1$ and for every 
mild solution $(u,p)$ of the Navier-Stokes equations on $\mathbb{R}^3\times [-2,0]$ satisfying
\begin{equation}\label{globalboundTypeI}
\|u\|_{L^{\infty}_{t}L^{3,\infty}_{x}(\mathbb{R}^3\times (-2,0))}\leq M
\end{equation}
and 
\begin{equation}\label{goodannulusL4}
u\in L^{4}_{x,t}(\mathbb{R}^3\times (-2,0)),
\end{equation}
there exists $R''(u,p,M,\mu,R)$ with
\begin{equation}\label{annulusradiusboundTypeI}
2R\leq R''\leq 2R\exp{(C(\mu) M^{10(\mu+2)})}
\end{equation}
and universal constants $\bar{C}_{j}\in(0,\infty)$ for $j=0, 1, 2$ such that for $j=0,1$
\begin{equation}\label{nablagoodannulusboundTypeI}
\|\nabla^j u\|_{L^{\infty}(\{R''<|x|<c(\mu){M^{10\mu}}R''\}\times (-\frac{1}{16},0))}\leq \bar{C}_{j}C(\mu) M^{-{3\mu}}
\end{equation}
and 
\begin{equation}\label{vortgoodannulusboundTypeI}
\|\nabla \omega\|_{L^{\infty}(\{R''<|x|<c(\mu){M^{10\mu}}R''\}\times (-\frac{1}{16},0))}\leq \bar{C}_{2}C(\mu) M^{-{3\mu}}.
\end{equation}
\end{lemma}

A simple rescaling  
gives the following corollary, which is directly used in the proof of Proposition \ref{prop.main}.

\begin{corollary}
\label{cor.12}
Let $S\in(0,\infty)$. For all $\mu>0$, let $M_1(\mu)>1$ be given by Lemma \ref{lem.ga}. For all $M\geq M_1(\mu)$, $R\geq S^\frac12$ and for every mild solution $(u,p)$ of the Navier-Stokes equations on $\mathbb{R}^3\times [-S,0]$ satisfying
\begin{equation}\label{globalboundTypeIbis}
\|u\|_{L^{\infty}_{t}L^{3,\infty}_{x}(\mathbb{R}^3\times (-S,0))}\leq M
\end{equation}
and
\begin{equation}\label{L4goodannulusrescaled}
u\in L^{4}_{x,t}(\mathbb{R}^3\times (-S,0)),
\end{equation}
there exists $R''(u,p,M,\mu,R)$ with  \eqref{annulusradiusboundTypeI} 
and universal constants $\bar{C}_{j}$ for $j=0,1,2$ such that for $j=0,1$
\begin{equation}\label{nablagoodannulusboundTypeIbis}
\|\nabla^j u\|_{L^{\infty}(\{R''<|x|<c(\mu){M^{10\mu}}R''\}\times (-\frac{S}{32},0))}\leq 2^{\frac{j+1}{2}}\bar{C}_{j}C(\mu) M^{-{3\mu}}S^{-\frac{j+1}{2}}
\end{equation}
and 
\begin{equation}\label{vortgoodannulusboundTypeIbis}
\|\nabla \omega\|_{L^{\infty}(\{R''<|x|<c(\mu){M^{10\mu}}R''\}\times(-\frac{S}{32},0))}\leq 2^{\frac23}\bar{C}_{2}C(\mu) M^{-{3\mu}}S^{-\frac{3}{2}}.
\end{equation}
\end{corollary}

Bearing in mind \eqref{semigroupL3L103}, the energy estimate in footnote \ref{footenergy} (see also Lemma \ref{lem.backconctimeslice}) and Calder\'{o}n-Zygmund estimates for the pressure, the following Lemma is obtained as an immediate corollary to Proposition \ref{goodannulusgneral}.

\begin{lemma}[annulus of regularity, time slices]\label{lem.annulusts}
For all $\mu>0$ there exists $M_2(\mu)>1$ such that the following holds true. For all $M\geq M_2(\mu)$, $R\geq 1$ and for every suitable finite-energy solution\footnote{See Section \ref{subsec.not} `Notations' for a definition of \textit{suitable finite-energy solutions.}} 
$(u,p)$ of the Navier-Stokes equations on $\mathbb{R}^3\times [-1,0]$ satisfying 
\begin{equation}\label{globalboundtimeslice}
\|u(\cdot,-1)\|_{L^{3}(\mathbb{R}^3)}\leq M
\end{equation}
and with $\Mp$ defined by \eqref{e.defM'theo}, 
there exists $R''(u,p,M,\mu,R)$ with 
\begin{equation}\label{annulusradiusboundtimeslice}
2R\leq R''\leq 2R\exp{(C(\mu)(\Mp)^{{10(\mu+2)}})}
\end{equation}
and universal constants $\bar{C}_{j}$ for $j=0,1,2$ such that for $j=0,1$
\begin{equation}\label{nablagoodannulusboundtimeslice}
\|\nabla^j u\|_{L^{\infty}(\{R''<|x|<(\Mp)^{{10\mu}}R''\}
\times(-\frac{1}{16},0))}\leq \bar{C}_{j} (\Mp)^{-3\mu}
\end{equation}
and
\begin{equation}\label{vortgoodannulusboundtimeslice}
\|\nabla \omega\|_{L^{\infty}(\{R''<|x|<(\Mp)^{10\mu}R''\}
\times(-\frac{1}{16},0))}\leq \bar{C}_{2} (\Mp)^{-3\mu}.
\end{equation}
\end{lemma}

A simple rescaling 
gives the following corollary, which is directly used in the proof of Theorem \ref{theo.mainbis}.

\begin{corollary}\label{cor.rescaleannulusts}
Let $S\in(0,\infty)$. For all $\mu>0$, let $M_2(\mu)>1$ and $\Mp$ be given by Lemma \ref{lem.annulusts}. For all $M\geq M_2(\mu)$, $R\geq S^\frac12$ and for every suitable  
finite-energy solution $(u,p)$ of the Navier-Stokes equations on $\mathbb{R}^3\times [-S,0]$ satisfying
\begin{equation}\label{globalboundts}
\|u(\cdot,-S)\|_{L^{3}(\mathbb{R}^3)}\leq M,
\end{equation}
there exists $R''(u,p,M,\mu,R)$ with
\begin{equation}\label{annulusradiusboundts}
2R\leq R''\leq 2R\exp{(C(\mu)(\Mp)^{10(\mu+2)})}
\end{equation}
and universal constants $\bar{C}_{j}$ for $j=0,1,2$ such that for $j=0,1$ 
\begin{equation}\label{nablagoodannulusboundts}
\|\nabla^j u\|_{L^{\infty}(\{R''<|x|<(\Mp)^{10\mu}R''\}\times (-\frac{S}{32},0))}
\leq 2^{\frac{j+1}{2}}\bar{C}_{j}(\Mp)^{-{3\mu}}S^{-\frac{j+1}{2}}
\end{equation} 
and 
\begin{equation}\label{vortgoodannulusboundts}
\|\nabla \omega\|_{L^{\infty}(\{R''<|x|<(\Mp)^{10\mu}R''\}\times(-\frac{S}{32},0))}\leq 2^\frac32\bar{C}_{2} (\Mp)^{-{3\mu}}S^{-\frac{3}{2}}.
\end{equation}
\end{corollary}

\begin{remark}\label{rem.contomega}
According to Remark \ref{vortimprovesws}, in Proposition \ref{goodannulusgneral}-Corollary \ref{cor.rescaleannulusts} we have that $w,\ \partial_tw,\ \nabla w$ and $\nabla^2w$ are continuous in space and time on the annuli considered. 
This remark is needed to apply the first Carleman inequality, Proposition \ref{prop.firstcarl}, in Section \ref{sec.3} and \ref{sec.further}.
\end{remark}

\section{Main tool 3: quantitative epochs of regularity}
\label{sec.6}

In this section, we prove that a solution satisfying the hypothesis of Propositions \ref{prop.main}-\ref{prop.maints} enjoys good quantitative estimates in certain time intervals. In the literature, these are commonly referred to as `epochs of regularity'. Such a property is crucially used in `Step 1' of the above propositions, when applying a quantitative Carleman inequality based on unique continuation (Proposition \ref{prop.sndcarl}).

To show the existence of such epochs of regularity, we follow Leray's approach in \cite{Leray}. In particular, we utilize arguments involving existence of mild solutions for subcritical data and weak-strong uniqueness. We provide full details for the reader's convenience.

We first recall a result known as `O'Neil's convolution inequality' (Theorem 2.6 of  O'Neil's paper \cite{O'Neil}).
\begin{proposition}\label{O'Neil}
Suppose $1< p_{1}, p_{2}, r<\infty$ and $1\leq q_{1}, q_{2}, s\leq\infty $ are such that
\begin{equation}\label{O'Neilindices1}
\frac{1}{r}+1=\frac{1}{p_1}+\frac{1}{p_{2}}
\end{equation}
and
\begin{equation}\label{O'Neilindices2}
\frac{1}{q_1}+\frac{1}{q_{2}}\geq \frac{1}{s}.
\end{equation}
Suppose that
\begin{equation}\label{fghypothesis}
f\in L^{p_1,q_1}(\mathbb{R}^{d})\,\,\textrm{and}\,\,g\in  L^{p_2,q_2}(\mathbb{R}^{d}).
\end{equation}
Then
\begin{equation}\label{fstargconclusion1}
f\ast g \in L^{r,s}(\mathbb{R}^d)\,\,\rm{with} 
\end{equation}
\begin{equation}\label{fstargconclusion2}
\|f\ast g \|_{L^{r,s}(\mathbb{R}^d)}\leq 3r \|f\|_{L^{p_1,q_1}(\mathbb{R}^d)} \|g\|_{L^{p_2,q_2}(\mathbb{R}^d)}. 
\end{equation}
\end{proposition}

We will also use an inequality that we will refer to as `Hunt's inequality'. The statement below and its proof can be found in Hunt's paper \cite{hunt} (Theorem 4.5, p.271 of \cite{hunt}).

\begin{proposition}\label{hunt}
Suppose that $1\leq p,q,r\leq\infty$ and $1\leq s_1,s_2\leq\infty$. Furthermore, suppose that $p$, $q$, $r$, $s_1$ and $s_2$ satisfy the following relations: $$\frac{1}{p}+\frac{1}{q}=\frac{1}{r}$$ and $$\frac{1}{s_1}+\frac{1}{s_2}=\frac{1}{s}.$$ 
Then the assumption that $f\in L^{p,s_1}(\Omega)$ and $g\in L^{q,s_2}(\Omega)$ implies  that $fg \in L^{r,s}(\Omega)$, with the estimate 
\begin{equation}\label{Holderverygeneral}
\|fg\|_{L^{r,s}(\Omega)}\leq C(p,q,s_1,s_2)\|f\|_{L^{p,s_1}(\Omega)}\|g\|_{L^{q,s_2}(\Omega)}.
\end{equation}
\end{proposition}

\begin{lemma}[epoch of regularity, Type I]\label{epochTypeI}
There exists a universal constant $C_3\in[1,\infty)$ 
such that the following holds. Suppose $u: \mathbb{R}^3\times [t_{0}-T, t_{0}]\rightarrow \mathbb{R}^3$ and $p: \mathbb{R}^3\times [t_{0}-T, t_{0}]\rightarrow \mathbb{R}$ is a mild solution\footnote{See Subsection \ref{subsec.def} of Section \ref{subsec.not} `Notations'.}
 of the Navier-Stokes equations. 
Furthermore, assume for some $M\geq 1$ 
that
\begin{equation}\label{typeIboundepoch}
\|u\|_{L^{\infty}_{t}L^{3,\infty}_{x}(\mathbb{R}^3\times (t_{0}-T, t_{0}))}\leq M
\end{equation}
and
\begin{equation}\label{L4assumption}
u\in L^{4}_{x,t}(\mathbb{R}^3\times (t_0-T,t_0)).
\end{equation}
Then for all intervals $I\subset [t_{0}-\frac{T}{2}, t_{0}]$ there exists a subinterval $I'\subset I$ 
 such that the following holds true.
Namely, 
\begin{equation}\label{epochTypeIsmoothing} 
\|\nabla^j u\|_{L^{\infty}_{t}L^{\infty}_{x}(\mathbb{R}^3\times I')}\leq C_3 M^{18} |I|^{\frac{-(j+1)}{2}}
\end{equation}
 for $j=0,1,2$ and 
\begin{equation}\label{epochTypeIintervallength}
|I'|\geq C_3^{-1} M^{-12} |I|.
\end{equation}
\end{lemma}
\begin{remark}[estimates for applying Carleman inequalities (Type I)]\label{rem.estcarlepo}
Let $I''\subset I'$ be such that $$|I''|= \frac{M^{-36}}{4 C_3^2}|I'|.$$ 
Then $$|I'|^{-1}= \frac{M^{-36}}{4 C_3^2}|I''|^{-1}.$$
Using \eqref{typeIboundepoch}-\eqref{epochTypeIsmoothing}, together with the fact that $C_3$ and $M\in [1,\infty)$, we see that
\begin{equation}\label{carlemantypeIepochsmoothing}
\|\nabla^j u\|_{L^{\infty}_{t}L^{\infty}_{x}(\mathbb{R}^3\times I'')}\leq \frac{1}{2^{j+1}} |I''|^{\frac{-(j+1)}{2}}
\end{equation}
for $j=0,1,2$ and
\begin{equation}\label{carlemantypeIepochbound}
|I''|\geq \frac{M^{-48}}{4C_3^3}{|I|}.
\end{equation} 
\end{remark}
\begin{proof}

The first part of the proof closely follows arguments in Tao's paper \cite{Tao19}. The only difference in the first part of the proof is that we exploit the above facts regarding Lorentz spaces.  For completeness, we give full details.

As observed by Tao in \cite{Tao19}, \eqref{typeIboundepoch}-\eqref{epochTypeIintervallength} are invariant with respect to the Navier-Stokes scaling and time translation. So we can assume without loss of generality
\begin{equation}\label{intervalWLOG}
I=[0,1]\subset [t_{0}-\tfrac{T}{2}, t_{0}]\Rightarrow [-1,1]\subset [t_{0}-T, t_{0}]
\end{equation}
\noindent\textbf{Step 1: a priori energy estimates.} Clearly we have from the standing assumptions that
\begin{equation}\label{TypeI-11}
\|u\|_{L^{\infty}_{t}L^{3,\infty}_{x}(\mathbb{R}^3\times (-1,1))}\leq M.
\end{equation}
On $\mathbb{R}^3\times (-1,1)$ we have
\begin{equation}\label{udecompEpoch}
u= e^{(t+1)\Delta}u(\cdot,-1)+w, \,\,\,\,\,\,w:=-\int\limits_{-1}^{t}e^{(t-s)\Delta}\mathbb{P}\nabla\cdot(u\otimes u)(\cdot,s) ds.
\end{equation}
It is known that $e^{t\Delta}\mathbb{P}\nabla\cdot$ has an associated convolution kernel $K$. 
Furthermore, from Solonnikov's paper \cite{solonnikov1964}, this satisfies the estimate
\begin{equation}\label{Oseenest}
|\partial_{t}^m\nabla^{j}K(x,t)|\leq \frac{C(m,j)}{(|x|^2+t)^{2+\frac{j}{2}+m}},\,\,\,\textrm{for}\,\,\,j,m=0,1\ldots
\end{equation}
Thus we may apply O'Neil's convolution inequality (Proposition \ref{O'Neil}) with $r=s=2$, $q_{1}=p_{1}=\frac{6}{5}$, $p_{2}=\frac{3}{2}$ and $q_{2}=\infty$. This and Hunt's inequality (Proposition \ref{hunt}) gives that for $t\in [-1,1]$
$$\|w(\cdot,t)\|_{L^{2}_{x}}\leq \int\limits_{-1}^{t} \frac{\|u\otimes u(\cdot,s)\|_{L^{\frac{3}{2},\infty}}}{(t-s)^{\frac{3}{4}}} ds\leq M^{2}(t+1)^{\frac{1}{4}}. $$
Thus,\begin{equation}\label{wkineticenergyest}
 \|w\|_{L^{\infty}_{t}L^{2}_{x}(\mathbb{R}^3\times (-1,1))}\leq CM^2.
\end{equation}
Using O'Neil's convolution inequality once more gives
\begin{equation}\label{semigroupLinfinity}
\|e^{(t+1)\Delta}u(\cdot,-1)\|_{L^{\infty}_{x}}\leq \frac{CM}{(t+1)^{\frac{1}{2}}},
\end{equation}
\begin{equation}\label{semigroupL103}
\|e^{(t+1)\Delta}u(\cdot,-1)\|_{L^{\frac{10}{3}}_{x}}\leq \frac{CM}{(t+1)^{\frac{1}{20}}},
\end{equation}
\begin{equation}\label{semigroupL4}
\|e^{(t+1)\Delta}u(\cdot,-1)\|_{L^{4}_{x}}\leq \frac{CM}{(t+1)^{\frac{1}{8}}},
\end{equation}
and
\begin{equation}\label{semigroupL62}
\|e^{(t+1)\Delta}u(\cdot,-1)\|_{L^{6,2}_{x}}\leq \frac{CM}{(t+1)^{\frac{1}{4}}}.
\end{equation}

Next, we see that $w$ satisfies
\begin{equation}\label{wequation}
\partial_{t}w-\Delta w+u\cdot\nabla  u+\nabla p=0,\,\,\,\,\nabla\cdot w=0\,\,\,\, w(\cdot,-1)=0.  
\end{equation}
Using \eqref{L4assumption}, \eqref{semigroupLinfinity} and \eqref{semigroupL4} we infer the following. Namely, $w\in C([0,1]; L^{2}_{\sigma}(\mathbb{R}^3))\cap L^{2}(0,1; \dot{H}^{1}(\mathbb{R}^3))$ and satisfies the energy equality\footnote{This can be shown by utilizing arguments in \cite{lemarie2016navier} (Lemma 7.2 in \cite{lemarie2016navier}) and \cite{sohr2001navier} (Theorem 2.3.1 in \cite{sohr2001navier}), for example.} for $t\in [0,1]$:
\begin{align*}
&\frac{1}{2}\|w(\cdot,t)\|_{L^{2}_{x}}^{2}+\int\limits_{0}^{t}\int\limits_{\mathbb{R}^3} |\nabla w|^2 dxdt'
\\& =\frac{1}{2}\|w(\cdot,0)\|_{L^{2}_{x}}^{2}+\int\limits_{0}^{t}\int\limits_{\mathbb{R}^3} e^{(t'+1)\Delta} u(\cdot,-1)\otimes (w+e^{(t'+1)\Delta}u(\cdot,-1)):\nabla w dxdt'.
\end{align*}
Using H\"{o}lder's inequality followed by Young's inequality we see that
\begin{align*}
\|w(\cdot,t)\|_{L^{2}_{x}}^{2}+\int\limits_{0}^{t}\int\limits_{\mathbb{R}^3} |\nabla w|^2 dxdt'
\leq\ & \|w(\cdot,0)\|_{L^{2}_{x}}^{2}+C\int\limits_{0}^{t}\int\limits_{\mathbb{R}^3} |e^{(t'+1)\Delta}u(\cdot,-1)|^4 dxdt'\\
&+C\int\limits_{0}^{t}\|w(\cdot,t')\|_{L^{2}_{x}}^{2}\|e^{(t'+1)\Delta}u(\cdot,-1)\|_{L^{\infty}_{x}}^2 dt' .
\end{align*}
Using this, together with \eqref{wkineticenergyest}-\eqref{semigroupL4} and the fact $M>1$, we obtain
\begin{equation}\label{wdissipationest}
\int\limits_{0}^{1}\int\limits_{\mathbb{R}^3} |\nabla w|^2 dx dt'\leq C_{univ}M^{6}.
\end{equation}
Here, $C_{univ}$ is a universal constant.
\\
\noindent\textbf{Step 2: higher integrability via weak-strong uniqueness.} 
From \eqref{wdissipationest}, the pigeonhole principle, the Sobolev embedding theorem and \eqref{semigroupL62}, there exists $t_{1}\in [0,\frac{1}{2}]$ such that
\begin{equation}\label{pidgeonholeepoch}
{\|u(\cdot,t_{1})\|_{L^{6}(\mathbb{R}^3)}}\leq C_{univ}M^{3}.
\end{equation}
This, \eqref{typeIboundepoch} and Lebesgue interpolation (see Lemma 2.2 in \cite{mccormick2013generalised} for example) 
implies that $u_{0}\in L^{4}_{\sigma}(\mathbb{R}^3)$. This and \eqref{pidgeonholeepoch} allows us to apply Proposition \ref{prop.mild} and Remark \ref{persistency}. In particular, there exists $C'_{univ}\in (0,\infty)$ and a mild solution $U:\mathbb{R}^3\times [t_1,t_1+\frac{C'_{univ}}{M^{12}}]\rightarrow\mathbb{R}^3$ to the Navier-Stokes equations, with initial data $u(\cdot,t_1)$, which satisfies the following properties. Specifically,
\begin{equation}\label{epochmildL6}
\|U\|_{L^{\infty}_{t}L^{6}_{x}(\mathbb{R}^3\times [t_1,t_1+\frac{C'_{univ}}{M^{12}}])}\leq CM^3
\end{equation}
and
\begin{equation}\label{epochmildL4}
U\in L^{\infty}_{t}L^{4}_{x}(\mathbb{R}^3\times [t_1,t_1+\frac{C'_{univ}}{M^{12}}]).
\end{equation}
Let $W:\mathbb{R}^3\times [t_1,t_1+\frac{C'_{univ}}{M^{12}}]\rightarrow\mathbb{R}^3$ be defined by
\begin{equation}\label{WdefuminusU}
W:=u-U=-\int\limits_{t_1}^{t}e^{(t-s)\Delta}\mathbb{P}\nabla\cdot(u\otimes u-U\otimes U)(\cdot,s) ds\,\,\,\textrm{for}\,\,\,t\in[t_1,t_1+\frac{C'_{univ}}{M^{12}}]. 
\end{equation}
Using \eqref{L4assumption} and \eqref{epochmildL4}, we see that
\begin{equation}\label{Wenergyspace}
W\in C\Big(\Big[t_1,t_1+\frac{C'_{univ}}{M^{12}}\Big]; L^{2}_{\sigma}(\mathbb{R}^3)\Big)\cap L^{2}_{t}\Big(t_1,t_1+\frac{C'_{univ}}{M^{12}};\dot{H}^{1}(\mathbb{R}^3)\Big)
\end{equation}
and $W$ satisfies the energy equality for $t\in[t_1,t_1+\frac{C'_{univ}}{M^{12}}]$: 
\begin{equation}\label{Wenergyequality}
\|W(\cdot,t)\|_{L^{2}_{x}}^{2}+2\int\limits_{t_{1}}^{t}\int\limits_{\mathbb{R}^3} |\nabla W|^2 dxdt'= 2\int\limits_{t_1}^{t}\int\limits_{\mathbb{R}^3} U\otimes W:\nabla W dxdt'.
\end{equation}
Using this, \eqref{epochmildL6} and known weak-strong uniqueness arguments from \cite{Leray}, we infer that $W\equiv 0$ on $\mathbb{R}^3\times [t_1,t_1+\frac{C'_{univ}}{M^{12}}]$. Using this together with \eqref{epochmildL6},
 we get that for $\tau(s):=t_{1}+\frac{ s C_{univ}'}{M^{12}}$:
\begin{equation}\label{usubcriticalestepoch}
\|u\|_{L^{\infty}_{t}L^{6}_{x}(\mathbb{R}^3\times(\tau(0),\tau(1)))}\leq CM^3.
\end{equation}
Using that the pressure is given by a Riesz transform acting on $u\otimes u$, we can apply Calder\'{o}n-Zygmund to get that the pressure $p$ associated to $u$ satisfies
\begin{equation}\label{pressuresubcriticalepoch}
\|p\|_{L^{\infty}_{t}L^{3}_{x}(\mathbb{R}^3\times(\tau(0),\tau(1)))}\leq C M^6.
\end{equation}
\\
\noindent\textbf{Step 3: higher derivative estimates.} Here, the arguments differ from those utilized in \cite{Tao19}. Fix any $x\in \mathbb{R}^3$ and $t\in [\tau(\frac{1}{2}),\tau(1)]$.
Take any $r\in (0,\sqrt{\frac{C_{univ}'}{2M^{12}}}]$, which ensures that $t-r^2\in [\tau(0),\tau(1)]$.
Using this and \eqref{usubcriticalestepoch}-\eqref{pressuresubcriticalepoch} we see that 
\begin{equation}\label{CKNepoch}
\frac{1}{r^2}\int\limits_{Q_{(x,t)}(r)} |u|^3+|p|^{\frac{3}{2}} dxdt'\leq C_{univ}''r^{\frac{3}{2}}M^{9}=C_{univ}''(rM^{6})^{\frac{3}{2}}. 
\end{equation} 
Taking 
$$r=r_{0}:=\frac{1}{M^{6}}
\min{\Big(\sqrt{\frac{C_{univ}'}{2}}, \frac{\epsilon_{CKN}^{\frac{2}{3}}}{(C_{univ}'')^{\frac{2}{3}}}\Big)},$$ 
we can then apply the Caffarelli-Kohn-Nirenberg theorem \cite{CKN82} to get that for $j=0,1,\ldots$
$$\sup_{(x,t)\in\mathbb{R}^3\times [\tau(\frac{1}{2}),\tau(1)]} |\nabla^{j} u(x,t)|\leq \frac{C}{r_{0}^{j+1}}\simeq C(j)(M^6)^{j+1} .$$
This concludes the proof.
\end{proof}

\begin{lemma}[epoch of regularity, time slices]
There exists a universal constant $C_{4}\in[1,\infty)$ 
such that the following holds.
\label{epochtimeslice}
Suppose $u:[-1,0]\times \mathbb{R}^3\rightarrow \mathbb{R}^3$ and $p:[-1,0]\times \mathbb{R}^3\rightarrow \mathbb{R}$ 
 is a suitable finite-energy 
 solution to the Navier-Stokes equations. 
Furthermore, assume for some $M\geq 1$ and $t_{0}\in [-1,0)$ that
\begin{equation}\label{timesliceboundepoch}
\|u(\cdot,t_{0})\|_{L^{3}_{x}}\leq M
\end{equation}
and $u$ satisfies the energy inequality \eqref{energyinequalityturbulent} starting from $t'=t_{0}$.

We define $M^{\flat}$ as in \eqref{e.defM'theo}. Fix any $\alpha\geq \Mp$ and let
 \begin{equation}\label{s0restriction}
 s_{0}\in \Big[\frac{t_{0}}{2}, \frac{t_{0}}{4\alpha^{201}}\Big].
 \end{equation}
 Define
 \begin{equation}\label{Idef}
 I:=\Big[s_{0}, \frac{s_{0}}{2}\Big].
 \end{equation}
There exists 
 a subinterval $I'\subset I$ such that the following holds true.
Namely, 
\begin{equation}\label{epochtimeslicesmoothing} 
\|\nabla^j u\|_{L^{\infty}_{t}L^{\infty}_{x}(\mathbb{R}^3\times I')}\leq C_4 \alpha^{324} |I|^{\frac{-(j+1)}{2}}
\end{equation}
 for $j=0,1,2$ and
\begin{equation}\label{epochtimesliceintervallength}
|I'|\geq C_{4}^{-1}\alpha^{-216} |I|.
\end{equation}
\end{lemma}
\begin{remark}[estimates for applying Carleman inequalities, time slices]\label{rem.estcarlepotimeslice}
Let $I''\subset I'$ be such that $$|I''|= \frac{\alpha^{-648}}{4 C_4^2}|I'|.$$ 
Then $$|I'|^{-1}= \frac{\alpha^{-648}}{4 C_4^2}|I''|^{-1}.$$
Using \eqref{epochtimeslicesmoothing}-\eqref{epochtimesliceintervallength}, together with the fact that $C_4$ and $M\in [1,\infty)$, we see that
\begin{equation}\label{carlemantimesliceepochsmoothing}
\|\nabla^j u\|_{L^{\infty}_{t}L^{\infty}_{x}(\mathbb{R}^3\times I'')}\leq \frac{1}{2^{j+1}} |I''|^{\frac{-(j+1)}{2}}
\end{equation}
for $j=0,1,2$ and
\begin{equation}\label{carlemantimesliceepochbound}
|I''|\geq \frac{\alpha^{-864}}{4C_4^{3}}{|I|}.
\end{equation} 
\end{remark}
\begin{proof}
Define, 
$$\hat{t}:=\frac{t_{0}}{s_{0}}-1.$$
Note that \eqref{s0restriction} implies that 
\begin{equation}\label{thatnotclosetozero}
\hat{t}\in (1, 4\alpha^{201}-1).
\end{equation}
 By appropriate scalings and translations, we can assume without loss of generality that $u:\mathbb{R}^3\times (0,\hat T)\rightarrow \mathbb{R}^3$, for some $\hat T\in (0,\infty)$\footnote{The time $\hat T$ is the image of $0$ by the scalings and translations. Its precise value does not matter at all, since the proof is carried out on the time interval $[0,\hat t+\frac12]$.} and
 \begin{equation}\label{epochWLOGid}
 \|u(\cdot,0)\|_{L^{3}(\mathbb{R}^3)}\leq M,
 \end{equation}
\begin{equation}\label{energyinequalityWLOG}
 \|u(\cdot,t)\|_{L^{2}}^2+2\int\limits_{0}^{t}\int\limits_{\mathbb{R}^3}|\nabla u(y,s)|^2 dyds\leq \|u(\cdot,0)\|_{L^{2}}^2
 \end{equation}
 and
 \begin{equation}\label{epochWLOGI}
 I:=\Big[\hat{t}, \hat{t}+\frac{1}{2}\Big]\subset (1,\hat T).
 \end{equation}
On $\mathbb{R}^3\times (0,\hat T)$ we have
\begin{equation}\label{udecomptimeslice}
u= e^{t\Delta}u(\cdot,0)+w.
\end{equation}
Using O'Neil's convolution inequality once more gives for all $t\in(0,\infty)$,
\begin{equation}\label{semigroupL3}
\|e^{t\Delta}u(\cdot,0)\|_{L^{3}_{x}}\leq CM,
\end{equation}
\begin{equation}\label{semigroupL3L103}
\|e^{t\Delta}u(\cdot,0)\|_{L^{\frac{10}{3}}_{x}}\leq \frac{CM}{t^{\frac{1}{20}}},
\end{equation}
\begin{equation}\label{semigroupL3L4}
\|e^{t\Delta}u(\cdot,0)\|_{L^{4}_{x}}\leq \frac{CM}{t^{\frac{1}{8}}},
\end{equation}
and
\begin{equation}\label{semigroupL3L62}
\|e^{t\Delta}u(\cdot,0)\|_{L^{6,2}_{x}}\leq \frac{CM}{t^{\frac{1}{4}}}.
\end{equation}
Furthermore, arguments from \cite{Giga86} imply that
\begin{equation}\label{semigroupL3L5}
\|e^{t\Delta}u(\cdot,0)\|_{L^{5}(\mathbb{R}^3\times (0,\infty))}\leq CM.
\end{equation}
Moreover, similar arguments as those used in Proposition 2.2 of \cite{SeSv17} 
 yield that for $t\in (0,\hat T)$,
 
\begin{align*}
&\|w(\cdot,t)\|_{L^{2}_{x}}^{2}+\int\limits_{0}^{t}\int\limits_{\mathbb{R}^3} |\nabla w|^2 dxdt'
\\& \leq C\int\limits_{0}^{t}\int\limits_{\mathbb{R}^3} |e^{t\Delta}u(\cdot,0)|^4 dxdt'+C\int\limits_{0}^{t}\|w(\cdot,t')\|_{L^{2}_{x}}^{2}\|e^{t'\Delta}u(\cdot,0)\|_{L^{5}_{x}}^5 dt' .
\end{align*}
Note that the energy inequality for $w$, which is used to produce this estimate, can be justified rigorously using \eqref{semigroupL3L5} and similar arguments as those used in Proposition 14.3 in \cite{LR02}.

Using \eqref{semigroupL3}-\eqref{semigroupL3L5} and Gronwall's lemma, we infer that
\begin{equation}\label{wenergyesttimeslice}
\sup_{0<t<\hat{t}+\frac{1}{2}}\|w(\cdot,t)\|_{L^{2}_{x}}^2+\int\limits_{0}^{\hat{t}+\frac{1}{2}}\int\limits_{\mathbb{R}^3} |\nabla w|^2 dxdt'\leq (\Mp)^4(\hat{t}+\tfrac{1}{2})^{\frac{1}{2}}< 2\alpha^{105}.
\end{equation}
Here we used \eqref{thatnotclosetozero}. Now let $\Sigma\subset [0, \hat{T}]$ be such that \eqref{energyinequalityturbulent} is satisfied for all $t\in [t',\hat{T}]$ and $t'\in \Sigma$. Since $u$ is a suitable  finite-energy solution we have that $|\Sigma|=\hat{T}$. Furthermore, $\Sigma$ can be chosen without loss of generality such that
$$\int\limits_{\mathbb{R}^3}|\nabla w(x,t')|^2 dx<\infty\,\,\,\textrm{for}\,\,\textrm{all}\,\,t'\in\Sigma. $$
Using \eqref{wenergyesttimeslice}, the Sobolev embedding theorem, the pidgeonhole principle and \eqref{semigroupL3L62}, we see that there exists $t_{1}\in [\hat{t}, \hat{t}+\frac{1}{4}]\cap\Sigma$ such that 
\begin{equation}\label{L6timeslice}
\|u(\cdot,t_1)\|_{L^{6}}^2\leq C\alpha^{105}.
\end{equation}
Making use of the fact that $u$ satisfies the energy inequality starting from $t_1$ and \eqref{L6timeslice}, we can utilize similar arguments  to those used in Lemma \ref{epochTypeI} replacing $M$ by $\alpha^{18}$.
\end{proof}

\appendix

\section{Auxiliary results}
\label{sec.A}

We first state the existence result of mild solutions with subcritical data.

\begin{proposition}[\cite{W80,Giga86}]\label{prop.mild}
Let $n\in(3,\infty)$. There exists $k_0(n)\in (0,\infty)$, $K_0(n)\in[1,\infty)$ such that the following holds. For all $u_0\in L^n_\sigma(\R^3)$, we define 
\begin{equation*}
S_{mild}(u_0):=k_0\|u_0\|_{L^n}^{-\frac{2n}{n-3}}\in(0,\infty).
\end{equation*}
There exists a unique mild solution $a\in C([0,S_{mild});L^n)\cap L^\infty((0,S_{mild});L^n)$ with initial data $u_0$ such that 
\begin{multline}\label{e.propmild}
\sup_{t\in(0,S_{mild})}\big(\|a(\cdot,t)\|_{L^n}+t^{\frac3{2n}}\|a(\cdot,t)\|_{L^\infty}+t^\frac12\|\nabla a(\cdot,t)\|_{L^n}+t^{\frac12+\frac3{2n}}\|\nabla a(\cdot,t)\|_{L^\infty}\big)\\
+\|a\|_{L^{\frac{5n}3}(\R^3\times(0,S_{mild}))}\leq K_0\|u_0\|_{L^n}.
\end{multline}
\end{proposition}
\begin{remark}\label{persistency}
For $U, V:\mathbb{R}^3\times (0,T)\rightarrow\mathbb{R}^3$ define
\begin{equation}\label{bilineardef}
B(U,V)(\cdot,t):=\int\limits_{0}^{t}\mathbb{P}\partial_{i}e^{(t-s)\Delta}U_{i}(\cdot,s)V_{j}(\cdot,s)ds.
\end{equation}
Using \eqref{Oseenest}, one has the estimate
\begin{multline*}
\|B(U,V)\|_{L^\infty_{t}L^{4}_{x}(\mathbb{R}^3\times [0,T])}+\|B(V,U)\|_{L^\infty_{t}L^{4}_{x}(\mathbb{R}^3\times [0,T])}\\
\leq cT^{\frac{1}{4}}\|U\|_{L^{\infty}_{t}L^{6}_{x}(\mathbb{R}^3\times [0,T])}\|V\|_{L^{\infty}_{t}L^{4}_{x}(\mathbb{R}^3\times [0,T])}.
\end{multline*}
Using this, one can show that for $u_{0}\in L^{6}_{\sigma}(\mathbb{R}^3)\cap L^{4}_{\sigma}(\mathbb{R}^3)$ and $k_{0}$ sufficiently small the following \textit{persistency} property holds true. Namely, the mild solution in Proposition \ref{prop.mild} satisfies
\begin{equation}\label{mildpersistencyL4}
\sup_{t\in (0,S_{mild})} \|a(\cdot,t)\|_{L^{4}}\leq 2\|u_{0}\|_{L^{4}}
\end{equation}
in addition to \eqref{e.propmild} with $n=6$. This is utilized in the proof of Lemma \ref{epochTypeI}.
\end{remark}
The next result is the local energy bound for local energy solutions.\footnote{\label{footles}Notice that `local energy solutions' to the Navier-Stokes equations, are sometimes described in the literature as `Lemari\'{e}-Rieusset solutions' or `Leray solutions'. They were conceived by Lemari\'{e}-Rieusset in \cite{LR02}. In our paper, whenever we refer to `local energy solutions', we mean in the sense of Definition 2.1 in \cite{JS13}. Notice, in particular, that suitable finite-energy solutions (defined in Section \ref{subsec.not} `Notations') are local energy solutions.}

\begin{proposition}[{\cite[Lemma 2.1]{JS13}, \cite{LR02}}]\label{prop.lews}
There exist two universal constants $k_1\in(0,\infty)$, $K_1\in [1,\infty)$ such that the following holds. For all $M\in(0,\infty)$, we define
\begin{equation*}
S_{locen}(M):=k_1\min(M^{-4},1)\in (0,\infty).
\end{equation*}
For all $u_0\in L^2_{uloc}(\R^3)$ with $\|u_0\|_{L^2(B_{\bar x}(1))}\stackrel{|\bar x|\rightarrow\infty}{\longrightarrow}0$, for all local energy solution $(u,p)$ to \eqref{e.nse} with initial data $u_0$, if 
\begin{equation*}
\sup_{\bar x\in\R^3}\int\limits_{B_{\bar x}(1)}|u_0(x)|^2\, dx\leq M^2,
\end{equation*}
then
\begin{equation}\label{e.apriori}
\sup_{s\in(0,S_{locen})}\sup_{\bar x\in\R^3}\int\limits_{B_{\bar x}(1)}\frac{|u(x,s)|^2}{2}\, dx+\sup_{\bar x\in\R^3}\int\limits_0^{S_{locen}}\int\limits_{B_{\bar x}(1)}|\nabla u(x,s)|^2\, dx\, ds\leq K_1M^2.
\end{equation}
Moreover, we have the following decomposition of the pressure: for all $\bar x\in\R^3$ and $t\in(0,S_{locen})$, there exists $C_{\bar x}(t)\in\R$ such that\footnote{This decomposition is also valid for $p_{loc}$ defined in $B_{\bar x}(\frac12)\times(0,S_{locen})$ instead of $B_{\bar x}(\frac32)\times(0,S_{locen})$ as stated here. The constant $C_{\bar x}(t)$ has to be adapted.}
\begin{equation}\label{e.pressuredec}
p(x,t)-C_{\bar x}(t)=-\frac13|u(x,t)|^2+p_{loc}(x,t)+p_{nonloc}(x,t)
\end{equation}
for all $(x,t)\in B_{\bar x}(\frac32)\times (0,S_{locen})$,  with
\begin{equation}\label{plocdef}
p_{loc}(x,t)=-\int\limits_{\mathbb{R}^3} K_{ij}(x-y)\varphi(y)u_{i}(y,t)u_{j}(y,t) dy
\end{equation}
and
\begin{equation}\label{pnonlocdef}
p_{nonloc}(x,t)=-\int\limits_{\mathbb{R}^3} (K_{ij}(x-y)-K_{ij}(\bar{x}-y))(1-\varphi(y))u_{i}(y,t)u_{j}(y,t) dy.
\end{equation}
Here, $\varphi\in C_{0}^{\infty}(B_{\bar{x}}(4))$ (with $\varphi\equiv 1$ on $B_{\bar{x}}(3)$) and $K_{ij}(x):=\partial_{i}\partial_{j}\Big(\frac{1}{|x|}\Big).$

Moreover, we have the estimate
\begin{align}\label{e.pressurest}
\|p_{loc}\|_{L^\frac53(B_{\bar x}(\frac32)\times(0,S_{locen}))}+\|p_{nonloc}\|_{L^\infty(B_{\bar x}(\frac32)\times(0,S_{locen}))}
\leq K_1M^2.
\end{align}
\end{proposition}

\section{Carleman inequalities}
\label{sec.B}
The two statements below are taken directly from \cite{Tao19}. The first Carleman inequality in Proposition \ref{prop.firstcarl} corresponds to a quantitative backward uniqueness result.

\begin{proposition}[first Carleman inequality {\cite[Proposition 4.2]{Tao19}}]\label{prop.firstcarl}
Let $C_{Carl}\in[1,\infty)$, 
 $S\in (0,\infty)$, $0<r_-<r_+$ and we define the space-time annulus 
\begin{equation*}
\mathcal A:=\{(x,t)\in\mathbb R^3\times\mathbb R\, :\ t\in[0,S],\ r_-\leq |x|\leq r_+\}.
\end{equation*}
Let $w:\ \mathcal A\rightarrow\mathbb R^3$ be such that $w,\ \partial_tw,\ \nabla w$ and $\nabla^2w$ are continuous in space and time and such that $w$ satisfies the differential inequality
\begin{equation}\label{e.diffineq}
\left|(\partial_t+\Delta)w\right|\leq C_{Carl}^{-1}S^{-1}|w|+C_{Carl}^{-\frac12}S^{-\frac12}|\nabla w|\quad\mbox{on}\ \mathcal A.
\end{equation}
Assume
\begin{equation}\label{e.lowerr-}
r_-^2\geq 4C_{Carl}S.
\end{equation}
Then we have the following bound
\begin{equation}\label{e.conclcarlone}
\int\limits_0^{\frac S4}\int\limits_{10r_-\leq|x|\leq \frac{r_+}2}(S^{-1}|w|^2+|\nabla w|^2)\, dxdt\lesssim C_{Carl}^{3}e^{-\frac{r_-\cdot r_+}{4C_{Carl}S}}\big(X+e^{\frac{2r_+^2}{C_{Carl}S}}Y\big),
\end{equation}
where 
\begin{align*}
X:=\iint\limits_{\mathcal A}e^{\frac{2|x|^2}{C_{Carl}S}}(S^{-1}|w|^2+|\nabla w|^2)\, dxdt,\qquad Y:=\int\limits_{r_-\leq |x|\leq r_+}|w(x,0)|^2\, dx.
\end{align*}
\end{proposition}

The second Carleman inequality in Proposition \ref{prop.sndcarl} below corresponds to a quantitative unique continuation result.

\begin{proposition}[second Carleman inequality {\cite[Proposition 4.3]{Tao19}}]\label{prop.sndcarl}
Let $C_{Carl}\in[1,\infty)$, $S\in (0,\infty)$, $r>0$ and we define the space-time cylinder
\begin{equation*}
\mathcal C:=\{(x,t)\in\mathbb R^3\times\mathbb R\, :\ t\in [0,S],\ |x|\leq r\}.
\end{equation*}
Let $w:\ \mathcal C\rightarrow\mathbb R^3$ such that $w,\ \partial_tw,\ \nabla w$ and $\nabla^2w$ are continuous in space and time and such that $w$ satisfies the differential inequality
\begin{equation}\label{e.diffineqC}
\left|(\partial_t+\Delta)w\right|\leq C_{Carl}^{-1}S^{-1}|w|+C_{Carl}^{-\frac12}S^{-\frac12}|\nabla w|\quad\mbox{on}\ \mathcal C.
\end{equation}
Assume 
\begin{equation}\label{e.lowerr}
r^2\geq 4000S.
\end{equation}
Then, for all $0<\check{s}\leq\hat s<\frac{S}{10000}$ one has the bound
\begin{equation}\label{e.conclcarltwo}
\int\limits_{\hat s}^{2\hat s}\int\limits_{|x|\leq \frac r2}(S^{-1}|w|^2+|\nabla w|^2)e^{-\frac{|x|^2}{4t}}\, dxdt\lesssim e^{-\frac{r^2}{500\hat s}}X+(\hat s)^\frac32\Big(\frac{e\hat s}{\check s}\Big)^{\frac{O(1)r^2}{\hat s}}Y,
\end{equation}
where 
\begin{align*}
X:=\int\limits_0^S\int\limits_{|x|\leq r}(S^{-1}|w|^2+|\nabla w|^2)\, dxdt,\qquad Y:=\int\limits_{|x|\leq r}|w(x,0)|^2(\check s)^{-\frac32}e^{-\frac{|x|^2}{4\check s}}\, dx.
\end{align*}
\end{proposition}

Proposition \ref{prop.firstcarl} and Proposition \ref{prop.sndcarl} are proved in \cite{Tao19} for smooth functions. The proof works under the weaker smoothness assumption stated here. This is used in Section \ref{sec.further} in particular, where the results are stated for suitable finite-energy solutions.

\subsection*{Funding and conflict of interest.} 
The second author is partially supported by the project BORDS grant ANR-16-CE40-0027-01 and by the project SingFlows grant ANR-18-CE40-0027 of the French National Research Agency (ANR). 
The second author also acknow\-ledges financial support from the IDEX of the University of Bordeaux for the BOLIDE project. The authors declare that they have no conflict of interest. 
\subsection*{Acknowledgement}
The authors acknowledge stimulating discussions with Dallas Albritton about Terence Tao's paper \cite{Tao19} and for suggesting Leray's approach for showing existence of epochs of regularity. The authors also thank Jan Burczak for discussions about Type I singularities, and Dongho Chae and J\"org Wolf for bringing to our attention the paper \cite{chae2017removing}.
\small 
\bibliographystyle{abbrv}
\bibliography{concentration.bib}

\end{document}